\documentclass{amsart}
\usepackage[a4paper,margin=1.8cm]{geometry}
\usepackage[all]{xy}
\usepackage{graphicx}
\usepackage{indentfirst}
\usepackage{bm}
\usepackage{amsthm}
\usepackage{mathrsfs}
\usepackage{latexsym}
\usepackage{amsmath}
\usepackage{amssymb}
\usepackage{color}

\usepackage{microtype}
\usepackage{needspace}
\usepackage{tikz}
\usetikzlibrary{decorations.pathreplacing,decorations.markings,arrows.meta}
\tikzset{
 tpe forward/.style={postaction={decorate},decoration={markings,mark=at position .5 with {\arrow{>}}}},
 tpe reverse/.style={postaction={decorate},decoration={markings,mark=at position .5 with {\arrow{<}}}}
}
\usetikzlibrary{shapes.geometric} 
\usepackage{enumerate}
\usepackage{comment}
\usepackage{braket}
\usepackage{manfnt}
\definecolor{webblue}{rgb}{0.08,0.26,0.58}
\usepackage[colorlinks=true,allcolors=black,urlcolor=webblue]{hyperref}
\usepackage[utf8]{inputenc}
\usepackage{relsize}
\usepackage{listings}
\usepackage{fancyvrb}
\usepackage{spverbatim}
\usepackage{xurl}
\usepackage{enumitem}

\usepackage[T1]{fontenc}
\usepackage{lmodern}

\usepackage{graphics}

\allowdisplaybreaks

\begin{document}

\theoremstyle{plain}
\newtheorem{theorem}{Theorem}[section]
\newtheorem{main}{Main Theorem}
\newtheorem{proposition}[theorem]{Proposition}
\newtheorem{corollary}[theorem]{Corollary}
\newtheorem{lemma}[theorem]{Lemma}
\newtheorem{conjecture}[theorem]{Conjecture}
\newtheorem{claim}[theorem]{Claim}
\newtheorem{fact}[theorem]{Fact}
\newtheorem{question}[theorem]{Question}
\newtheorem*{st}{Statements}	

\numberwithin{equation}{section}

\theoremstyle{definition}
\newtheorem{definition}[theorem]{Definition}
\newtheorem{notation}[theorem]{Notation}
\newtheorem{convention}[theorem]{Convention}
\newtheorem{example}[theorem]{Example}
\newtheorem{remark}[theorem]{Remark}
\newtheorem{strategy}[theorem]{Strategy}
\newtheorem*{ac}{Acknowledgements}	

\newcommand{\ev}{\mathrm{ev}}
\newcommand{\coev}{\mathrm{coev}}
\newcommand{\Hom}{\mathrm{Hom}}
\newcommand{\PSL}{\mathrm{PSL}}
\newcommand{\FPdim}{\mathrm{FPdim}}
\newcommand{\Rep}{\mathrm{Rep}}
\newcommand{\VVec}{\mathrm{Vec}}
\newcommand{\End}{\mathrm{End}}
\newcommand\Warning{%
 \makebox[1.4em][c]{%
 \makebox[0pt][c]{\raisebox{.1em}{\footnotesize \textup{!}}}%
 \makebox[0pt][c]{\Large$\bigtriangleup$}}}%

\newcommand{\mC}{\mathcal{C}}
\newcommand{\mB}{B}
\newcommand{\hc}{\Hom_{\mC}}
\newcommand{\one}{{\bf 1}} 
\newcommand{\id}{\mathrm{id}} 
\newcommand{\tr}{\mathrm{tr}} 
\newcommand{\spec}{i_0} 
\newcommand{\Sp}{i_0} 
\newcommand{\SpecS}{I_{s}} 
\newcommand{\SpS}{I_{s}'} 
\newcommand{\field}{\Bbbk} 
\newcommand{\red}{\textcolor{red}{\bullet}} 
\newcommand{\black}{\bullet} 

\newcommand{\MoveSymbol}{
    \begin{tikzpicture}
        \node[fill, circle, inner sep=0.5pt] (a) at (0,.05) {};
        \node[fill, circle, inner sep=0.5pt] (b) at (0,-.05) {};
        \draw (a) -- ++(-0.15,0.15);
        \draw (a) -- ++(0,0.15);
        \draw (a) -- ++(0.15,0.15);
        \draw (b) -- ++(-0.15,-0.15);
        \draw (b) -- ++(0,-0.15);
        \draw (b) -- ++(0.15,-0.15);
    \end{tikzpicture}
}
\newcommand{\ThreeSymbol}{\hspace*{-.1cm}\raisebox{-0.8ex}{\MoveSymbol}}

\title[Triangular prism equations and categorification]{Triangular prism equations and categorification}

\author{Zhengwei Liu}
\address{Z. Liu, Yau Mathematical Sciences Center and Department of Mathematics, Tsinghua University, and Beijing Institute of Mathematical Sciences and Applications, Huairou District, Beijing, China}
\email{liuzhengwei@mail.tsinghua.edu.cn}

\author{Sebastien Palcoux}
\address{S. Palcoux, Beijing Institute of Mathematical Sciences and Applications, Huairou District, Beijing, China}
\email{sebastienpalcoux@gmail.com}
\urladdr{https://sites.google.com/view/sebastienpalcoux}

\author{Yunxiang Ren}
\address{Y. Ren, LinkedIn Corporation, Sunnyvale, CA, USA}
\email{renyunxiang@gmail.com}

\author{Gert Vercleyen}
\address{G. Vercleyen, Department of Mathematics, Purdue University, West Lafayette, IN 47907, USA}
\email{gvercley@purdue.edu}

\hypersetup{pdftitle={Triangular prism equations and categorification},pdfauthor={Zhengwei Liu, Sebastien Palcoux, Yunxiang Ren, Gert Vercleyen}}
\date{}
\begin{abstract}
We obtain the triangular prism equations by evaluating a labelled prism in two ways, as a sum of products of two or of three tetrahedral coefficients. For a spherical fusion category an explicit change of variables identifies these equations with the pentagon equations. Prisms with repeated edge labels yield local subsystems, simplified using tetrahedral rotations and the Fuchs--Runkel--Schweigert parity theorem. Additional prism equations couple these subsystems. We prove that $\mathcal F_{210}$ admits no fusion categorification in characteristic zero and no pivotal fusion categorification in positive characteristic. Two further obstructions, the zero- and one-spectrum criteria, follow directly from the pentagon equations and apply in every characteristic without a pivotal structure, proving that $\mathcal F_{660}$ admits no fusion categorification in any characteristic. Together with a bounded census and known unitary obstructions, these results classify, up to Grothendieck equivalence, the non-pointed unitary simple integral fusion categories of rank at most $8$ and Frobenius--Perron dimension at most $20000$ for which the Frobenius--Perron dimension of each simple object divides that of the category.
\end{abstract}
\maketitle

\section{Introduction}\label{sec:intro}
The notion of a \emph{based ring} was introduced by Lusztig \cite{Lus}. A \emph{fusion ring} is a unital based ring of finite rank; see \cite[Definition~3.1.7]{EGNO15}. Its categorical counterpart, the \emph{fusion category}, was systematically developed by Etingof, Nikshych and Ostrik \cite{ENO05}. A realization of a fusion ring as the Grothendieck ring of a fusion category is called a \emph{categorification}. Mac~Lane's pentagon identity expresses the coherence of the associator \cite{MacLane}. In chosen bases of morphism spaces it becomes the \emph{pentagon equations} (PE) for the $F$-symbols, generalizations of the $SU(2)$ $6j$-symbols and their quantum analogues \cite{Bon07,DaHaWa,wang}.

Determining which fusion rings admit categorifications is a central problem in fusion category theory. After choosing bases one seeks invertible associator matrices satisfying the PE, together with unit constraints and evaluation and coevaluation maps satisfying the unit and duality axioms. The individual entries of the matrices are the $F$-symbols. Solving for them directly is usually difficult. Many constructions of fusion categories therefore proceed by other methods and do not determine their $F$-symbols explicitly. Finding manageable equations for these coefficients remains useful both for constructing categories and for proving that proposed fusion rules cannot be categorified.

The geometric meaning of the coefficients provides a reason to study them. A spherical fusion category over $\mathbb C$ supplies the categorical input for the Turaev--Viro state-sum invariant of $3$-manifolds and the associated $(2+1)$-dimensional topological quantum field theory \cite{TurVir92,BW96}. In a triangulation, edges are labelled by simple objects, triangular faces carry the corresponding morphism data, and tetrahedral weights are expressed through $F$-symbols. The pentagon relation governs the local change between two and three tetrahedra. In a pointed fusion category, a choice of bases identifies the associator with a group $3$-cocycle and the pentagon with its cocycle identity.

The PE have the schematic form
\begin{equation}\tag{PE}\label{Equ:PE}
 \sum_m\star\,\star=\sum_s\sum_m\star\,\star\,\star,
\end{equation}
where each $\star$ denotes an $F$-symbol, $s$ ranges over simple objects, and $m$ denotes the basis indices in the relevant morphism spaces. The full indexed expression appears in \S\ref{sec:PE}.

The tetrahedral and triangular prism graphs are the two smallest connected simple planar trivalent graphs:
\[
\begin{tikzpicture}[scale=.85,line width=.45pt]
\draw (0,0)--(2,0)--(1,1.732)--cycle;
\draw (0,0)--(1,.577)--(2,0) (1,.577)--(1,1.732);
\foreach \p in {(0,0),(2,0),(1,1.732),(1,.577)}\fill \p circle[radius=1.2pt];
\node[below] at (1,-.1) {\small tetrahedron};
\end{tikzpicture}
\qquad\qquad
\begin{tikzpicture}[scale=.85,line width=.45pt]
\draw (0,0)--(2,0)--(1,1.732)--cycle;
\draw (.5,.2885)--(1.5,.2885)--(1,1.1545)--cycle;
\draw (0,0)--(.5,.2885) (2,0)--(1.5,.2885) (1,1.732)--(1,1.1545);
\foreach \p in {(0,0),(2,0),(1,1.732),(.5,.2885),(1.5,.2885),(1,1.1545)}\fill \p circle[radius=1.2pt];
\node[below] at (1,-.1) {\small triangular prism};
\end{tikzpicture}
\]
In \S\ref{sec:monotetra} and \S\ref{sec:monoTPE}, we define categorical evaluations of these graphs, with objects on edges and morphisms at vertices. The \emph{triangular prism equations} (TPE) arise by evaluating the same labelled prism as a contraction of two tetrahedra or of three tetrahedra. The graphical notation specifies the ordering of the tensor factors at each vertex and the duality conventions for the edges. Each decomposition determines which coefficients to multiply and which indices to sum over. Thurston's picture \cite[Figure~4]{thu}, reproduced in Figure~\ref{fig:thurston-prism}, displays its geometric basis. Categorically, the cuts are implemented by dual-basis contraction and the resolution of the identity (\S\ref{sub:reso}), including the local passage from three parallel strands to two paired trivalent vertices, schematically $|||\mapsto\ThreeSymbol$.

\begin{figure}[ht]
\centering
\resizebox{\textwidth}{!}{%
\begin{tikzpicture}[x=.01cm,y=-.01cm,line width=.45pt,line cap=round,line join=round]
\foreach \sx/\sy in {7/8,7/147,127/77}{
  \begin{scope}[shift={(\sx,\sy)}]
  \draw (0,0)--(0,116)--(101,58)--cycle;
  \draw (0,0)--(34,58)--(0,116) (34,58)--(101,58);
  \foreach \p in {(0,0),(0,116),(34,58),(101,58)} \fill \p circle[radius=.045cm];
  \end{scope}
}
\draw[->] (321,134)--(254,134); \draw (321,126)--(321,142);
\draw (344,8)--(344,263)--(444,205) (344,8)--(444,66);
\draw (344,8)--(378,66)--(444,66) (344,263)--(378,205)--(444,205);
\draw (378,66) .. controls (355,111) and (355,161) .. (378,205);
\draw (444,66) .. controls (354,68) and (354,203) .. (444,205);
\foreach \p in {(344,8),(344,263),(378,66),(378,205),(444,66),(444,205)} \fill \p circle[radius=.045cm];
\begin{scope}[shift={(464,77)}]
\draw (0,0)--(0,116)--(100,58)--cycle;
\draw (0,0)--(33,58)--(0,116) (33,58)--(100,58);
\foreach \p in {(0,0),(0,116),(33,58),(100,58)} \fill \p circle[radius=.045cm];
\end{scope}
\draw[->] (660,134)--(594,134); \draw (660,126)--(660,142);
\draw (680,8)--(680,263)--(780,205) (680,8)--(901,135);
\draw (680,8)--(714,66) (680,263)--(714,205)--(780,205);
\draw (714,66) .. controls (691,110) and (691,161) .. (714,205);
\draw (714,66) .. controls (756,63) and (810,92) .. (833,135);
\draw (780,205) .. controls (728,185) and (696,154) .. (722,115)
                     .. controls (753,68) and (804,108) .. (800,193);
\draw (800,193)--(833,135)--(901,135)--cycle;
\foreach \p in {(680,8),(680,263),(714,66),(714,205),(780,205),(800,193),(833,135),(901,135)} \fill \p circle[radius=.045cm];
\draw[->] (987,134)--(921,134); \draw (987,126)--(987,142);
\draw (1004,8)--(1004,263)--(1224,135)--cycle;
\draw (1037,66)--(1037,205)--(1158,135)--cycle;
\draw (1004,8)--(1037,66) (1004,263)--(1037,205) (1158,135)--(1224,135);
\foreach \p in {(1004,8),(1004,263),(1037,66),(1037,205),(1158,135),(1224,135)} \fill \p circle[radius=.045cm];
\node at (1278,135) {$=$};
\draw (1331,30)--(1457,30)--(1394,66)--cycle;
\draw (1331,234)--(1457,234)--(1394,198)--cycle;
\draw (1394,66)--(1394,198);
\draw (1331,30) .. controls (1387,105) and (1387,160) .. (1331,234);
\draw (1457,30) .. controls (1401,105) and (1401,160) .. (1457,234);
\foreach \p in {(1331,30),(1457,30),(1394,66),(1331,234),(1457,234),(1394,198)} \fill \p circle[radius=.045cm];
\draw[->] (1466,135)--(1532,135); \draw (1466,127)--(1466,143);
\draw (1559,34)--(1663,34)--(1612,126)--cycle;
\draw (1559,34)--(1612,66)--(1663,34) (1612,66)--(1612,126);
\draw (1559,234)--(1663,234)--(1612,144)--cycle;
\draw (1559,234)--(1612,205)--(1663,234) (1612,205)--(1612,144);
\foreach \p in {(1559,34),(1663,34),(1612,66),(1612,126),(1559,234),(1663,234),(1612,205),(1612,144)} \fill \p circle[radius=.045cm];
\end{tikzpicture}}
\caption{Two decompositions of a triangular prism, redrawn from Thurston's picture \cite[Figure~4]{thu}. Read from the central prism outward: three tetrahedra on the left, two on the right. The intermediate graphs record the cuts underlying the two evaluations.}
\label{fig:thurston-prism}
\end{figure}

The geometry also organizes the symmetries. In the spherical setting, rotations of the regular tetrahedron act through $A_4$, and rotations of the regular triangular prism through the dihedral group of order six, isomorphic to $S_3$. Under a rotation, each edge keeps its arrow and its object label. Reading the rotated diagram in the original convention replaces a label by its dual when the edge is reversed and cyclically rotates the vertex morphism when its marked corner changes (Proposition~\ref{prop:A4Sym}). Consequently, rotationally related decorated tetrahedra give the same scalar, and rotated prism configurations yield corresponding equations. These identifications can reduce both the number of variables and the number of equations that must be retained in a computation.

The change of variables in Theorem~\ref{thm:PE-TPE} gives the precise relationship with associativity.

\begin{theorem}\label{thm:main}
For a spherical fusion category, the pentagon equations are equivalent to the triangular prism equations.
\end{theorem}

In the classical $6j$-symbol formalism, the identity for the move from two tetrahedra to three includes a sum over the new edge, weighted by its categorical dimension \cite[\S3, Proposition~5.2 and Lemma~5.4]{BW96}. Theorem~\ref{thm:TPE} gives the two cuts of a labelled prism in a pivotal fusion category, with the dual bases, dimension factors and pivotal maps explicit. We look for labellings with few distinct simple objects for which the two cuts still impose a relation on unknown coefficients. For example, we give the six edges of the triangular ends the same label and vary the three connecting labels. The tetrahedra that recur in the two cut diagrams tell us which coefficients to solve for together. A prism carrying two such labels can test whether solutions obtained separately are compatible.

For a finite abelian group $G$, He, Liu, Lu, Palcoux and Ren derive a variant of Izumi's equations by cutting a prism with the near-group fusion rules $G+|G|$. Over a field in which $|G|$ is invertible, their system characterizes spherical categorifications with $\nu_2(X)=1$ for every self-dual simple $X$ \cite[Theorem~2.3 and Remark~2.4]{LLPRnear}. Each summand contains two factors of $\beta$ in their equation~(2.7) and three in Izumi's corresponding equation. This form matches the hyperbolic gamma integral they use to construct examples for cyclic $G$ and gives $\beta(2h)=\beta(h)^4$ in characteristic two \cite[Remarks~2.2 and~2.6 and \S2.2]{LLPRnear}. Radchenko and Wheeler obtain a corresponding quadratic sum from their finite quantum dilogarithms and use it to derive Izumi's cubic-sum equation \cite[Theorem~2 and \S5.4]{RW26}.

\subsection*{General categorification criteria}
The geometric construction yields two general localization results in \S\ref{sec:FirstLoc}. First, choose a self-dual label $k$ whose tensor square has multiplicity-free, self-dual support $S_k$, and assume that some coefficient $N_{b,b^*}^k$ is odd. Theorem~\ref{thm:loc} then gives necessary equations for the coefficients for a spherical categorification, using its categorical dimensions. The Fuchs--Runkel--Schweigert theorem supplies the indicator normalization, and tetrahedral symmetry gives, in particular, $y(a,b)=y(b,a)$. Corollary~\ref{cor:loc} restricts exposed labels to a chosen $S'_k\subseteq S_k$ while retaining the full internal summation support.

Second, two such choices $k\ne l$, with $l\in S'_k\cap S'_l$, satisfy the compatibility equation of Theorem~\ref{thm:locextra} when their coefficients come from compatible vertex bases in the same category. Separate local systems can therefore have solutions while their coupled system has none. For $\mathcal F_{210}$ this gives two systems of twelve equations in ten variables. Each has fourteen complex solutions, and the mixed prism rules out all $196$ pairs. This excludes categorification in characteristic zero.

The zero- and one-spectrum criteria (Theorems~\ref{Thm: Ob-0} and~\ref{Thm: Ob-1}) use only the fusion coefficients. The \emph{spectrum} here is the set of simple labels allowed by the fusion rules on the three-factor side of a PE. An empty spectrum or a single summand of multiplicity one can force a contradiction in any characteristic, without pivotality. The first exclusions for both criteria occur in rank four (Proposition~\ref{prop:spectrum-rank}), and two nonintegral rank-six examples show that neither criterion implies the other (Proposition~\ref{prop:spectrum-incomparable}). The zero-spectrum criterion excludes $\mathcal F_{660}$ from fusion categorification in every characteristic. Its proof uses semisimplicity and leaves open realization as the Grothendieck ring of a nonsemisimple finite tensor category. Earlier examples of such exclusions are recorded in \cite[\S5.2]{LPRmult1}.

The supplementary \href{https://github.com/sebastienpalcoux/triangular-prism-categorification/blob/f816c697fde088b261e1fe1ce703980e5dbdab51/GUIDE.md}{\texttt{GUIDE.md}} explains how to apply these criteria to other fusion rings, including localization's dimension input and coupling conventions. The \href{https://github.com/sebastienpalcoux/triangular-prism-categorification/blob/f816c697fde088b261e1fe1ce703980e5dbdab51/criteria.py}{Python interface} searches both spectrum criteria and constructs characteristic-zero localization systems with specified or symbolic categorical dimensions. A solution of a necessary subsystem alone does not establish categorifiability.

\subsection*{Applications and motivation}
We recall the terminology used in the classification statement. A fusion ring $\mathcal F$ is \emph{integral} if the Frobenius--Perron dimensions of its basis elements are integers. It is \emph{$1$-Frobenius} if $\FPdim(\mathcal F)/\FPdim(b)$ is an algebraic integer for every basis element $b$; in the integral case this means that each $\FPdim(b)$ divides $\FPdim(\mathcal F)$. This is the usual Frobenius divisibility property. A fusion category has these properties when its Grothendieck ring does. It is \emph{simple} if its only fusion subcategories are itself and the trivial category, and \emph{pointed} if every simple object is invertible. Two fusion categories are \emph{Grothendieck equivalent} when their Grothendieck rings are isomorphic as based rings. We use the standard conventions for fusion rings and fusion categories in \cite{Lus,ENO05,EGNO15}.

Kaplansky's sixth conjecture \cite{Kap75} asserts that for every finite-dimensional semisimple Hopf algebra $H$ over $\mathbb C$, the integral fusion category $\Rep(H)$ is $1$-Frobenius. For fusion categories over $\mathbb C$, this divisibility is a special case of the strong Frobenius property asked about in \cite[Question~1]{ENO11}.

Over $\mathbb C$, Etingof, Nikshych and Ostrik ask whether every weakly integral fusion category is weakly group-theoretical \cite[Question~2]{ENO11}. A non-pointed simple weakly group-theoretical fusion category is equivalent to $\Rep(G)$ for a nonabelian finite simple group $G$ \cite[Proposition~9.11]{ENO11}. A simple integral fusion category whose Grothendieck ring is not a group character ring would therefore give a negative answer. More generally, it suffices to find a nontrivial integral fusion category whose Grothendieck ring contains no nontrivial group character fusion subring. Indeed, a minimal nontrivial fusion subcategory of a weakly group-theoretical fusion category is again weakly group-theoretical. In the pointed case its Grothendieck ring is the character ring of a cyclic group of prime order. In the non-pointed case it is the representation category of a nonabelian finite simple group \cite[Propositions~4.1 and~9.11]{ENO11}. It is also unknown whether a non-weakly-group-theoretical fusion category can be Grothendieck equivalent to the representation category of a finite group. We therefore seek simple integral fusion rings other than group character rings (or, more generally, integral fusion rings containing no nontrivial group character fusion subring) that pass the known categorification tests.

The fusion rings $\mathcal F_{210}$ and $\mathcal F_{660}$ arise among the candidates in bounded searches for such examples. The fusion ring $\mathcal F_{210}$ is the $q=6$ member of the interpolation of the character rings of $\PSL(2,q)$ constructed in \cite{LPRinter}. This interpolation is a construction of fusion rings from character-table formulas; it does not presume an interpolating family of tensor categories. Its arithmetic properties explain why several familiar necessary conditions for categorification fail to exclude it. Our localization argument rules out any characteristic-zero categorification. The positive-characteristic pivotal obstruction follows from a general lifting criterion for commutative fusion rings with a unique nonvanishing linear character (Proposition~\ref{prop:unique-nonvanishing}).

In \cite{LPW20}, Wu and two of the present authors applied the Schur product theorem from subfactor theory \cite[Theorem~4.1]{Liuex} as a unitary categorification test, the \emph{Schur product criterion}. The primary $3$-criterion of \cite{HLPW23} refines this approach. Every simple integral fusion ring of rank at most three is pointed (see \cite[\S1]{DPPvan}). Dong, Palcoux and Plessis proved that no simple integral fusion ring of rank four is $1$-Frobenius \cite[Theorem~1.7]{DPPvan}. For ranks five through twelve, the census of $1$-Frobenius simple integral fusion rings initiated in \cite{LPW20} gives these bounds \cite[\S1]{BP24}.
\[
\begin{array}{c|cccccccc}
\text{Rank}&5&6&7&8&9&10&11&12\\\hline
\FPdim\le&10^7&10^6&10^5&20000&10000&5000&3000&1000.
\end{array}
\]
The census records $505$ non-pointed fusion rings. Eight are group character rings, and the primary $3$-criterion excludes $477$, leaving $20$ further candidates. Only two of the latter have rank at most $8$, namely $\mathcal F_{210}$ and $\mathcal F_{660}$. Their unitary categorifiability was asked in \cite[Question~1.1]{LPW20}. Our two exclusions answer that question negatively and yield the following conclusion.

\begin{theorem}\label{thm:ClassSimpleIntegral}
Every non-pointed, unitary, $1$-Frobenius, simple integral fusion category of rank at most $8$ with $\FPdim\le 20000$ is Grothendieck equivalent to $\Rep(\PSL(2,q))$, where $q\in\{4,7,9,11\}$.
\end{theorem}

The four groups are $A_5$, $\PSL(2,7)$, $A_6$ and $\PSL(2,11)$, of ranks
$5,6,7,8$ and dimensions $60,168,360,660$, respectively. The parameter $q=5$
also gives $A_5$, whereas $q=8$ gives rank $9$. The theorem concerns
Grothendieck equivalence, not uniqueness of the categorification of each
surviving fusion ring. Its enumeration input is the bounded census.

The interpolated character rings of \cite{LPRinter} also suggest a constructive direction. We hope that the geometric equations and their localized systems will help construct a non-pointed simple integral fusion category outside the representation categories of finite groups. The present obstruction excludes one candidate from that family, while leaving this broader motivation intact. Remark~\ref{rk:504} discusses the next rank-$9$ test case and further localizations, together with an alternative approach through induction to the Drinfeld center.

The TPE and small-spectrum criteria have contributed to the classification of multiplicity-free Grothendieck rings \cite[\S4.1]{LPRmult1} and the computations through rank seven reported in \cite{SliVer,gvfuscat,VS26,gjwiki}. Together with the near-group construction discussed above, these applications illustrate that the geometric presentation is useful both for extracting obstructions and for organizing equations in construction problems.

\subsection*{Organization and verification}
\S\S\ref{sec:pre}--\ref{sec:pre2} fix the duality, pivotal and trace conventions. \S\S4--5 develop the tetrahedral calculus and identify its coefficients with those of the associator. \S6 develops localization, and \S\ref{sec:app} applies the resulting obstructions.

The \href{https://github.com/sebastienpalcoux/triangular-prism-categorification/blob/f816c697fde088b261e1fe1ce703980e5dbdab51/GUIDE.md#4-verify-the-papers-applications}{reader's guide} lists the input data and verifiers for each computational claim and distinguishes checking a certificate from generating it. The enumeration of the $505$ fusion rings comes from the cited census. The supplement \href{https://github.com/sebastienpalcoux/triangular-prism-categorification/blob/f816c697fde088b261e1fe1ce703980e5dbdab51/census/verify_census.py}{independently checks} the \href{https://github.com/sebastienpalcoux/triangular-prism-categorification/blob/f816c697fde088b261e1fe1ce703980e5dbdab51/census/rank_at_most_8.json}{$33$ entries of rank at most eight} and provides the \href{https://github.com/sebastienpalcoux/triangular-prism-categorification/blob/f816c697fde088b261e1fe1ce703980e5dbdab51/CERTIFICATE.md}{characteristic-zero certificate for $\mathcal F_{210}$} and the \href{https://github.com/sebastienpalcoux/triangular-prism-categorification/blob/f816c697fde088b261e1fe1ce703980e5dbdab51/verify.py#L79-L87}{zero-spectrum check for $\mathcal F_{660}$}.

\tableofcontents

\section{Duality and graphical conventions}\label{sec:pre}

All categories are locally small. We follow the duality conventions of
\cite[Chapters~2 and~4]{EGNO15}. Unless otherwise specified, $\field$ is an
algebraically closed field of arbitrary characteristic. A fusion category $\mC$ over
$\field$ is a $\field$-linear rigid semisimple monoidal category with
finite-dimensional morphism spaces, finitely many isomorphism classes of simple
objects, and $\End_{\mC}(\one)=\field$. We allow zero global dimension in positive
characteristic.

We suppress associativity and unit constraints by monoidal coherence. Chosen left duals are denoted by $X^*$, with
\[
 \ev_X:X^*\otimes X\longrightarrow\one,
 \qquad \coev_X:\one\longrightarrow X\otimes X^*.
\]
They satisfy the two zigzag identities. We use compatible duals of tensor products,
so that $(X\otimes Y)^*=Y^*\otimes X^*$, and normalize $\one^*=\one$.
Right duals are denoted by ${}^*X$. Whenever they are used, we suppress the canonical
isomorphisms $({}^*X)^*\simeq X$ and $ {}^*(X^*)\simeq X$ and their iterates.
In particular, an inverse rotation of a last factor $X^*$ uses
$X^*\simeq({}^*X)^{**}$. We write $g\circ f$ for composition,
read from right to left; ordinary string diagrams are read from top to bottom.

For $f:X\to Y$ and $g:Y\to Z$, our elementary conventions are
\[
\ev_X=
\begin{tikzpicture}[baseline=(current bounding box.center),scale=.7]
\draw (0,.8) node[above] {$X^*$} -- (0,0) arc(180:360:.6) -- (1.2,.8) node[above] {$X$};
\end{tikzpicture},\qquad
\coev_X=
\begin{tikzpicture}[baseline=(current bounding box.center),scale=.7]
\draw (0,-.8) node[below] {$X$} -- (0,0) arc(180:0:.6) -- (1.2,-.8) node[below] {$X^*$};
\end{tikzpicture},\qquad
 g\circ f=
\begin{tikzpicture}[baseline=(current bounding box.center),scale=.7]
\node[draw,fill=white,inner sep=2pt] (first) at (0,.69) {$f$};
\node[draw,fill=white,inner sep=2pt] (last) at (0,-.69) {$g$};
\draw (0,1.40) node[above] {$X$} -- (first.north);
\draw (first.south) -- node[midway,right=1pt,font=\scriptsize] {$Y$} (last.north);
\draw (last.south) -- (0,-1.40) node[below] {$Z$};
\end{tikzpicture}.
\]
The decorated polyhedral graphs of \S\ref{bitetra} will be read as string
diagrams using the conventions given there.

\subsection{Pivotal traces and local identities}
A pivotal structure on a rigid monoidal category $\mC$ is a monoidal natural isomorphism
$a:\id_{\mC}\longrightarrow(-)^{**}$. Our trace and dimension conventions are
\begin{equation}\label{def:tr}
\begin{aligned}
 \tr(f)&=
 \vcenter{\hbox{\begin{tikzpicture}[scale=.8]
 \draw (0,0) .. controls (0,.65) and (1.7,.65) .. (1.7,0)
       -- (1.7,-1.6) .. controls (1.7,-2.2) and (0,-2.2) .. (0,-1.6)
       -- (0,0);
 \node[draw,fill=white,minimum width=.55cm,inner sep=2pt] at (0,-.45) {$f$};
 \node[draw,fill=white,minimum width=.55cm,inner sep=2pt] at (0,-1.15) {$a_X$};
 \node[left,font=\scriptsize] at (0,.02) {$X$};
 \node[left,font=\scriptsize] at (0,-1.65) {$X^{**}$};
 \node[right,font=\scriptsize] at (1.7,-.8) {$X^*$};
 \end{tikzpicture}}}
 =\ev_{X^*}\circ((a_X\circ f)\otimes\id_{X^*})\circ\coev_X,\\[2pt]
 d_X&=\dim(X)=\tr(\id_X),\qquad f\in\End_{\mC}(X).
\end{aligned}
\end{equation}
The trace in \eqref{def:tr} is the left pivotal trace of $f$.
Its right pivotal trace is $\tr(f^*)$, the left pivotal trace of the dual
endomorphism $f^*:X^*\to X^*$. A pivotal monoidal category is
\emph{spherical} if these traces agree for every $f$.\label{def:sph}

\begin{remark}\label{rk:sph}
A pivotal fusion category is spherical if and only if $d_X=d_{X^*}$ for every simple object $X$ \cite[Definition~4.7.14 and Theorem~4.7.15]{EGNO15}. These are the two pivotal traces of $\id_X$.
\end{remark}

\begin{lemma}[Duality identities]\label{lem:LeftRight}\label{lem:ev*}\label{lem:trans}\label{lem:switch}\label{lem:pivo*}\label{lem:trf*}
For $f:X\to Y$, the duality identities are
\[
 \ev_Y\circ(\id_{Y^*}\otimes f)=\ev_X\circ(f^*\otimes\id_X),\qquad
 (f\otimes\id_{X^*})\circ\coev_X=(\id_Y\otimes f^*)\circ\coev_Y.
\]
Graphically, the two identities are
\[
\begin{tikzpicture}[baseline=(current bounding box.center),scale=.70]
\draw (0,1.23) node[above] {$Y^*$} -- (0,-.2) arc(180:360:.7) -- (1.4,1.23) node[above] {$X$};
\node[draw,fill=white,inner sep=2pt] at (1.4,.57) {$f$};
\node[right=1pt,font=\scriptsize] at (1.4,-.05) {$Y$};
\end{tikzpicture}
=\begin{tikzpicture}[baseline=(current bounding box.center),scale=.70]
\draw (0,1.23) node[above] {$Y^*$} -- (0,-.2) arc(180:360:.7) -- (1.4,1.23) node[above] {$X$};
\node[draw,fill=white,inner sep=2pt] at (0,.57) {$f^*$};
\node[left=2.5pt,font=\scriptsize] at (0,-.105) {$X^*$};
\end{tikzpicture},\qquad
\begin{tikzpicture}[baseline=(current bounding box.center),scale=.70]
\draw (0,-1.23) node[below] {$Y$} -- (0,.2) arc(180:0:.7) -- (1.4,-1.23) node[below] {$X^*$};
\node[draw,fill=white,inner sep=2pt] at (0,-.57) {$f$};
\node[left=1pt,font=\scriptsize] at (0,.05) {$X$};
\end{tikzpicture}
=\begin{tikzpicture}[baseline=(current bounding box.center),scale=.70]
\draw (0,-1.23) node[below] {$Y$} -- (0,.2) arc(180:0:.7) -- (1.4,-1.23) node[below] {$X^*$};
\node[draw,fill=white,inner sep=2pt] at (1.4,-.57) {$f^*$};
\node[right=2.5pt,font=\scriptsize] at (1.4,.105) {$Y^*$};
\end{tikzpicture}.
\]
For compatible duality data, $(\ev_X)^*=\coev_{X^*}$ and
$(\coev_X)^*=\ev_{X^*}$. For $\alpha:Y\to X^*$ and
$\beta:\one\to X$,
\begin{equation}\label{eq:dual-transposition}
 \ev_X\circ(\alpha\otimes\beta)
   =\ev_{X^*}\circ(\beta^{**}\otimes\alpha).
\end{equation}
In a pivotal monoidal category the trace is cyclic, so
$\tr(g\circ f)=\tr(f\circ g)$ for $f:X\to Y$, $g:Y\to X$.
The following identities also hold, with $f\in\End_{\mC}(X)$ in the last formula:
\begin{equation}\label{eq:pivotal-duality}
 \beta^{**}=a_X\circ\beta,\qquad
 a_X^*=a_{X^*}^{-1},\qquad a_X^{**}=a_{X^{**}},\qquad
 \tr(f^*)=\ev_X\circ(\id_{X^*}\otimes(f\circ a_X^{-1}))\circ\coev_{X^*}.
\end{equation}
\end{lemma}
\begin{proof}
The first identities are duality naturality and the zigzag identities
\cite[\S2.10]{EGNO15}. Both sides of \eqref{eq:dual-transposition}
equal $\beta^*\circ\alpha:Y\to\one$. The pivotal identities follow from
naturality, \cite[Exercise~4.7.9]{EGNO15}, and the trace formulas
\cite[\S4.7]{EGNO15}. Sliding $f$ past the cup and cap by the first two
identities replaces $f$ by $f^{**}$, and $f^{**}\circ a_X=a_Y\circ f$
gives cyclicity.
\end{proof}
For example, the pivotal compatibility at a cap is
\[
\raisebox{-.7cm}{\begin{tikzpicture}[scale=.8]
\draw (0,0) node[above,font=\scriptsize] {$X^*$} -- (0,-1.15)
 .. controls (0,-1.7) and (1.5,-1.7) .. (1.5,-1.15)
 -- (1.5,0) node[above,font=\scriptsize] {$X^{**}$};
\node[draw,fill=white,inner sep=2pt] at (0,-.65) {$a_{X^*}$};
\end{tikzpicture}}
\ =\
\raisebox{-.7cm}{\begin{tikzpicture}[scale=.8]
\draw (0,0) node[above,font=\scriptsize] {$X^*$} -- (0,-1.15)
 .. controls (0,-1.7) and (1.5,-1.7) .. (1.5,-1.15)
 -- (1.5,0) node[above,font=\scriptsize] {$X^{**}$};
\node[draw,fill=white,inner sep=2pt] at (1.5,-.65) {$a_X^{-1}$};
\end{tikzpicture}}.
\]
We write $a_X^2=a_{X^{**}}\circ a_X:X\to X^{****}$.

\begin{lemma}[Spherical closure]\label{lem:sph2}
In a spherical monoidal category, the two closures of $f\in\End_{\mC}(X)$ agree:
\[
\underbrace{\begin{tikzpicture}[x=.85cm,y=-.85cm,baseline=(current bounding box.center)]
\draw (0,.3) .. controls (0,-.2) and (1.7,-.2) .. (1.7,.3)
 -- (1.7,1.8) .. controls (1.7,2.3) and (0,2.3) .. (0,1.8);
\draw (0,.3)--(0,1.8);
\node[draw,fill=white,minimum width=.52cm,minimum height=.3cm,inner sep=1pt] at (0,.725) {$f$};
\node[draw,fill=white,minimum width=.52cm,minimum height=.3cm,inner sep=1pt] at (0,1.325) {$a_X$};
\node[left] at (0,.3) {\scriptsize $X$};
\node[left] at (0,1.8) {\scriptsize $X^{**}$};
\node[right] at (1.7,1.05) {\scriptsize $X^*$};
\end{tikzpicture}}_{\tr(f)}
\quad=\quad
\underbrace{\begin{tikzpicture}[x=.85cm,y=-.85cm,baseline=(current bounding box.center)]
\draw (1.7,.3) .. controls (1.7,-.2) and (0,-.2) .. (0,.3)
 -- (0,1.8) .. controls (0,2.3) and (1.7,2.3) .. (1.7,1.8);
\draw (1.7,.3)--(1.7,1.8);
\node[draw,fill=white,minimum width=.52cm,minimum height=.3cm,inner sep=1pt] at (1.7,.725) {$a_X^{-1}$};
\node[draw,fill=white,minimum width=.52cm,minimum height=.3cm,inner sep=1pt] at (1.7,1.325) {$f$};
\node[right=2.5pt] at (1.7,.22) {\scriptsize $X^{**}$};
\node[right] at (1.7,1.8) {\scriptsize $X$};
\node[left] at (0,1.05) {\scriptsize $X^*$};
\end{tikzpicture}}_{\tr(f^*)}.
\]
For $\alpha:\one\to Y^*\otimes X^*$ and
$\beta:\one\to X\otimes Y$, cutting the same closed diagram in different places gives
\[
\begin{tikzpicture}[x=.9cm,y=-.9cm,baseline=(current bounding box.center)]
\draw (.85,0)--(.85,.55) .. controls (.85,1.05) and (2.45,1.05) .. (2.45,.55)--(2.45,0);
\draw (.35,0)--(.35,.65) .. controls (.35,1.6) and (2.95,1.6) .. (2.95,.65)--(2.95,0);
\node[left] at (.35,.68) {\scriptsize $Y^*$};
\node[right] at (.85,.62) {\scriptsize $X^*$};
\node[left] at (2.45,.62) {\scriptsize $X$};
\node[right] at (2.95,.68) {\scriptsize $Y$};
\draw[densely dashed] (.68,.31)--(1.02,.31) (2.78,.31)--(3.12,.31);
\node[draw,fill=white,minimum width=.8cm,minimum height=.34cm,inner sep=1pt] at (.6,0) {$\alpha$};
\node[draw,fill=white,minimum width=.8cm,minimum height=.34cm,inner sep=1pt] at (2.7,0) {$\beta$};
\end{tikzpicture}
\;=\;
\begin{tikzpicture}[x=.9cm,y=-.9cm,baseline=(current bounding box.center)]
\draw (.85,0)--(.85,.7);
\draw (.85,.7)--(.85,1.2) .. controls (.85,1.65) and (1.45,1.65) .. (1.45,1.2)
 -- (1.45,-.12) .. controls (1.45,-.72) and (-.3,-.72) .. (-.3,-.12)
 -- (-.3,1.2) .. controls (-.3,2.5) and (3.6,2.5) .. (3.6,1.2)
 -- (3.6,-.12) .. controls (3.6,-.72) and (1.9,-.72) .. (1.9,-.12)
 -- (1.9,1.2) .. controls (1.9,1.65) and (2.45,1.65) .. (2.45,1.2)--(2.45,0);
\draw (.35,0)--(.35,1.0) .. controls (.35,2.1) and (2.95,2.1) .. (2.95,1.0)--(2.95,0);
\node[left] at (.35,.7) {\scriptsize $Y^*$};
\node[right] at (2.95,.7) {\scriptsize $Y$};
\node[draw,fill=white,minimum width=.8cm,minimum height=.34cm,inner sep=1pt] at (.6,0) {$\alpha$};
\node[draw,fill=white,minimum width=.8cm,minimum height=.34cm,inner sep=1pt] at (2.7,0) {$\beta$};
\node[draw,fill=white,minimum width=.72cm,minimum height=.43cm,inner sep=1pt] at (.85,.7) {$a_{X^*}^{2}$};
\end{tikzpicture}
\;=\;
\begin{tikzpicture}[x=.9cm,y=-.9cm,baseline=(current bounding box.center)]
\draw (.85,0)--(.85,.55) .. controls (.85,1.05) and (2.45,1.05) .. (2.45,.55)--(2.45,0);
\draw (2.95,0)--(2.95,.8);
\draw (.35,0)--(.35,1.15) .. controls (.35,1.65) and (-.25,1.65) .. (-.25,1.15)
 -- (-.25,-.12) .. controls (-.25,-.95) and (3.55,-.95) .. (3.55,-.12)
 -- (3.55,1.15) .. controls (3.55,1.65) and (2.95,1.65) .. (2.95,1.15)--(2.95,.8);
\node[right] at (.85,.42) {\scriptsize $X^*$};
\node[left] at (2.45,.42) {\scriptsize $X$};
\node[draw,fill=white,minimum width=.8cm,minimum height=.34cm,inner sep=1pt] at (.6,0) {$\alpha$};
\node[draw,fill=white,minimum width=.8cm,minimum height=.34cm,inner sep=1pt] at (2.7,0) {$\beta$};
\node[draw,fill=white,minimum width=.72cm,minimum height=.43cm,inner sep=1pt] at (2.95,.8) {$a_Y^{2}$};
\end{tikzpicture}.
\]
The dashed cuts on $X^*$ and $Y$ correspond to the middle and last
pictures, respectively.
\end{lemma}
\begin{proof}
The first display is sphericality in the conventions of \eqref{def:tr}
and \eqref{eq:pivotal-duality} \cite[Theorem~4.7.15]{EGNO15}.
Reclosing either marked cut on the opposite side inserts two turns and hence
$a_{U^{**}}\circ a_U=a_U^2$ on the indicated leg.
\end{proof}

\begin{lemma}\label{lem:simple}
For a simple object $X$ in a pivotal fusion category, $d_X\ne0$ and
$f=d_X^{-1}\tr(f)\id_X$ for $f\in\End_{\mC}(X)$.
\end{lemma}
\begin{proof}
The dimension is nonzero by \cite[Proposition~4.8.4]{EGNO15}.
Schur's lemma gives $f=\lambda\id_X$; taking its trace gives
$\lambda=d_X^{-1}\tr(f)$.
\end{proof}
\subsection{Oriented cups, caps, and indicators}\label{sec:bio}
In a pivotal monoidal category, write the four duality morphisms as follows
\cite{Bart}.
\begin{align*}
 \ev_{X,+}&=\ev_X,&
 \coev_{X,+}&=\coev_X,\\
 \ev_{X,-}&=\ev_{X^*}\circ(a_X\otimes\id_{X^*}),&
 \coev_{X,-}&=(\id_{X^*}\otimes a_X^{-1})\circ\coev_{X^*}.
\end{align*}
The minus-sign evaluation has source $X\otimes X^*$ and the minus-sign
coevaluation has target $X^*\otimes X$. Their graphical conventions are
\[
\ev_{X,+}=\raisebox{-.35cm}{\begin{tikzpicture}
\draw[tpe reverse] (0,0) arc(-180:0:.4);
\node[below] at (.4,-.45) {\scriptsize $X$};
\end{tikzpicture}},\quad
\ev_{X,-}=\raisebox{-.35cm}{\begin{tikzpicture}
\draw[tpe forward] (0,0) arc(-180:0:.4);
\node[below] at (.4,-.45) {\scriptsize $X$};
\end{tikzpicture}},\quad
\coev_{X,+}=\raisebox{-.35cm}{\begin{tikzpicture}
\draw[tpe reverse] (-.4,-.4) arc(180:0:.4);
\node[above] at (0,.05) {\scriptsize $X$};
\end{tikzpicture}},\quad
\coev_{X,-}=\raisebox{-.35cm}{\begin{tikzpicture}
\draw[tpe forward] (-.4,-.4) arc(180:0:.4);
\node[above] at (0,.05) {\scriptsize $X$};
\end{tikzpicture}}.
\]
The oriented identity strands are
\[
\id_X=\begin{tikzpicture}[baseline=(current bounding box.center)]
\draw[tpe forward] (0,.55)--(0,-.55);
\node[left] at (0,0) {$X$};
\end{tikzpicture},\qquad
\id_{X^*}=\begin{tikzpicture}[baseline=(current bounding box.center)]
\draw[tpe reverse] (0,.55)--(0,-.55);
\node[left] at (0,0) {$X$};
\end{tikzpicture}.
\]
The oriented zigzag and loop identities are
\[
\begin{tikzpicture}[baseline=-.1cm,scale=.62]
\draw[tpe forward] (0,1)--(0,-.2)
 .. controls (0,-.8) and (.65,-.8) .. (.65,-.2)--(.65,.2)
 .. controls (.65,.8) and (1.3,.8) .. (1.3,.2)--(1.3,-1);
\node[left] at (0,.7) {$X$};
\end{tikzpicture}
=
\begin{tikzpicture}[baseline=-.1cm,scale=.62]
\draw[tpe forward] (0,1)--(0,-1);
\node[left] at (0,.7) {$X$};
\end{tikzpicture},
\qquad
\begin{tikzpicture}[baseline=-.1cm,scale=.62]
\draw[tpe reverse] (0,1)--(0,-.2)
 .. controls (0,-.8) and (.65,-.8) .. (.65,-.2)--(.65,.2)
 .. controls (.65,.8) and (1.3,.8) .. (1.3,.2)--(1.3,-1);
\node[left] at (0,.7) {$X$};
\end{tikzpicture}
=
\begin{tikzpicture}[baseline=-.1cm,scale=.62]
\draw[tpe reverse] (0,1)--(0,-1);
\node[left] at (0,.7) {$X$};
\end{tikzpicture},
\qquad
\begin{tikzpicture}[baseline=-.1cm]
\draw[tpe forward] (0,.4) arc(90:450:.4);
\node[left] at (-.4,0) {$X$};
\end{tikzpicture}=d_X,
\qquad
\begin{tikzpicture}[baseline=-.1cm]
\draw[tpe reverse] (0,.4) arc(90:450:.4);
\node[left] at (-.4,0) {$X$};
\end{tikzpicture}=d_{X^*}.
\]
Also
$(\ev_{X,\pm})^*=\coev_{X^*,\pm}$ and
$(\coev_{X,\pm})^*=\ev_{X^*,\pm}$.
\begin{proof}[Proof of the oriented identities]
Substituting the four definitions reduces the oriented zigzags to the ordinary
zigzag identities. The duality formulas follow from Lemma~\ref{lem:LeftRight}
and the loop values from \eqref{def:tr} and \eqref{eq:pivotal-duality}.
\end{proof}
For a simple object $X$ of a pivotal fusion category, with $X$ and $X^*$ fixed, the zigzag identities and the
prescribed nonzero pivotal dimensions determine these four morphisms up to
the simultaneous rescaling
\[
 (\ev_{X,+},\ev_{X,-},\coev_{X,+},\coev_{X,-})
 \longmapsto
 (\lambda\ev_{X,+},\lambda\ev_{X,-},
  \lambda^{-1}\coev_{X,+},\lambda^{-1}\coev_{X,-}),
 \qquad \lambda\in\field^*.
\]
\begin{proof}[Proof of the rescaling assertion]
Rigidity and simplicity make each of these morphism spaces one-dimensional.
Writing the two cap scalings as $\lambda_+,\lambda_-$, the zigzags force
the cup scalings to be their inverses. The nonzero loop value then gives
$\lambda_+\lambda_-^{-1}=1$, hence $\lambda_+=\lambda_-$.
\end{proof}

For $n\ge 1$ in a $\field$-linear pivotal monoidal category, the cyclic operator on $\hc(\one,X^{\otimes n})$ and the indicator
of \cite[Definition~3.1, p.~71, and Lemma~3.4, p.~72]{NgSch} are
\[
E_X^{(n)}(\alpha)=
\raisebox{-.8cm}{\begin{tikzpicture}[scale=.85]
\draw (-.6,0)--(-.6,-.5)
 arc(0:-180:.3 and .2)--(-1.2,.3)
 .. controls (-1.2,1.15) and (2.05,1.15) .. (2.05,.3)--(2.05,-.8);
\draw (.05,0)--(.05,-.8);\draw (.75,0)--(.75,-.8);
\node[font=\scriptsize] at (-.29,-.49) {$X$};
\node at (.4,-.55) {$\cdots$};
\node[draw,fill=white,minimum width=1.85cm,minimum height=.4cm] at (.05,0) {$\alpha$};
\node[draw,fill=white,inner sep=2pt] at (2.05,-.25) {$a_X^{-1}$};
\node[right] at (2.05,.4) {\scriptsize $X^{**}$};
\node[below] at (.05,-.8) {\scriptsize $X$};
\node[below] at (.75,-.8) {\scriptsize $X$};
\node[below] at (2.05,-.8) {\scriptsize $X$};
\end{tikzpicture}},
\qquad
\nu_n(X)=\operatorname{Tr}\bigl(E_X^{(n)}\bigr),
\]
where $\operatorname{Tr}$ denotes the ordinary linear trace whenever this space is
finite-dimensional.

\begin{proposition}\label{prop:nupm}
The operator $E_X^{(n)}$ satisfies $(E_X^{(n)})^n=\id$.
If $\hc(\one,X^{\otimes n})$ is one-dimensional, it acts by
$\nu_n(X)$ and $\nu_n(X)^n=1$.
\end{proposition}
\begin{proof}
The first identity is \cite[Theorem~5.1, p.~75]{NgSch}. In a one-dimensional
space the operator equals its trace, so the remaining assertions follow.
\end{proof}

\begin{lemma}\label{lem:evnu}
If $\hc(X\otimes X,\one)$ is one-dimensional, then every
$\kappa:X\to X^*$ satisfies
\begin{equation}\label{eq:indicator-transpose}
 \kappa^*\circ a_X=\nu_2(X^*)\kappa,
 \quad\text{equivalently}\quad
 \ev_{X^*}\circ(a_X\otimes\kappa)
 =\nu_2(X^*)\ev_X\circ(\kappa\otimes\id_X).
\end{equation}
The second identity is the following equality of caps:
\[
\begin{tikzpicture}[baseline=(current bounding box.center),scale=.85]
\draw (0,1)--(0,-.3) arc(180:360:.55)--(1.1,1);
\node[above] at (0,1) {$X$};\node[above] at (1.1,1) {$X$};
\node[draw,fill=white,inner sep=2pt] at (0,.45) {$a_X$};
\node[draw,fill=white,inner sep=2pt] at (1.1,.45) {$\kappa$};
\node[left] at (0,-.12) {\scriptsize $X^{**}$};
\node[right] at (1.1,-.12) {\scriptsize $X^*$};
\end{tikzpicture}
=\nu_2(X^*)
\begin{tikzpicture}[baseline=(current bounding box.center),scale=.85]
\draw (0,1)--(0,-.3) arc(180:360:.55)--(1.1,1);
\node[above] at (0,1) {$X$};\node[above] at (1.1,1) {$X$};
\node[draw,fill=white,inner sep=2pt] at (0,.45) {$\kappa$};
\node[left] at (0,-.12) {\scriptsize $X^*$};
\end{tikzpicture}.
\]
\end{lemma}
\begin{proof}
Bend $\kappa$ to $\Phi(\kappa)=(\kappa\otimes\id_{X^*})\circ\coev_X$.
Rotating its first leg gives
\[
 E_{X^*}^{(2)}(\Phi(\kappa))=
 \begin{tikzpicture}[baseline=(current bounding box.center),x=.7cm,y=.7cm,font=\scriptsize]
 \draw (0,-.8)--(0,.30) .. controls (0,1.05) and (1.7,1.05) .. (1.7,.30)--(1.7,-.8);
 \node[left] at (0,.37) {$X^{**}$};\node[right] at (1.7,.37) {$X^{***}$};
 \node[draw,fill=white,inner sep=2pt] at (0,-.18) {$\kappa^*$};
 \node[draw,fill=white,inner sep=2pt] at (1.7,-.18) {$a_{X^*}^{-1}$};
 \end{tikzpicture}
 =
 \begin{tikzpicture}[baseline=(current bounding box.center),x=.7cm,y=.7cm,font=\scriptsize]
 \draw (0,-.8)--(0,.30) .. controls (0,1.05) and (1.7,1.05) .. (1.7,.30)--(1.7,-.8);
 \node[left] at (0,.37) {$X$};\node[right] at (1.7,.37) {$X^*$};
 \node[draw,fill=white,inner sep=2pt] at (0,-.18) {$\kappa^*\!\circ a_X$};
 \end{tikzpicture}
 =\Phi(\kappa^*\circ a_X),
\]
using $(a_X^{-1})^*=a_{X^*}$ at the cup. On this line the rotation
acts by $\nu_2(X^*)$; unbending gives \eqref{eq:indicator-transpose}.
\end{proof}

\begin{proposition}\label{prop:pivoFS}
Let $X$ be a simple object of a pivotal fusion category with $X\simeq X^*$. Then
$\nu_2(X)^2=1$. If the chosen dual object is literally $X^*=X$, then
$a_X=\nu_2(X)\id_X$ and $a_X^2=\id_X$.
\end{proposition}
\begin{proof}
Rigidity and simplicity give $\dim_{\field}\hc(\one,X\otimes X)=1$,
so Proposition~\ref{prop:nupm} gives $\nu_2(X)^2=1$.
If $X^*=X$, apply Lemma~\ref{lem:evnu} with $\kappa=\id_X$.
\end{proof}
\begin{lemma}\label{lem:biFS}
Under the latter hypothesis,
\[
 \ev_{X,-}=\nu_2(X)\ev_{X,+},\qquad
 \coev_{X,-}=\nu_2(X)\coev_{X,+}.
\]
More generally, if $X^{**}=X$ and $a_X=\epsilon\id_X$ with
$\epsilon^2=1$, then
$\ev_{X,-}=\epsilon\ev_{X^*,+}$ and
$\coev_{X,-}=\epsilon\coev_{X^*,+}$.
\end{lemma}
\begin{proof}
Substitute $a_X=\nu_2(X)\id_X$, respectively $a_X=\epsilon\id_X$,
into the definitions of the four oriented cups and caps; each scalar equals
its inverse.
\end{proof}
Literal equality with a chosen dual is not invariant
under changing duality data. In particular, in characteristic different from
$2$, the simultaneous choices $X^*=X$ and $a_X=\id_X$ are incompatible
with $\nu_2(X)=-1$. This is a compatibility statement about fixed data, not
an obstruction to choosing a self-dual object as its own dual.

\section{Pairings, resolutions, and the parity theorem}\label{sec:pre2}

We identify $\End_{\mC}(\one)=\field\id_{\one}$ with $\field$.
We specify the morphism space for each basis and the pairing used to
define its dual basis.

\subsection{Dual bases and resolution of the identity}\label{sub:reso}
\begin{lemma}\label{lem:bili}
For any object $Z$ of a fusion category, the pairing
\[
 b_Z:\hc(\one,Z^*)\times\hc(\one,Z)\longrightarrow\field,
 \qquad b_Z(\alpha,\beta)=\ev_Z\circ(\alpha\otimes\beta),
\]
is nondegenerate. Its graphical form is
\[
b_Z(\alpha,\beta)=
\begin{tikzpicture}[baseline=(current bounding box.center),scale=.7]
\draw (0,.5)--(0,-.15) arc(180:360:1)--(2,.5);
\node[draw,fill=white,minimum width=.7cm,inner sep=2pt] at (0,.5) {$\alpha$};
\node[draw,fill=white,minimum width=.7cm,inner sep=2pt] at (2,.5) {$\beta$};
\node[left=2pt,font=\scriptsize] at (0,-.12) {$Z^*$};\node[right=2pt,font=\scriptsize] at (2,-.12) {$Z$};
\end{tikzpicture}.
\]
Thus any basis $(e_i)$ of $\hc(\one,Z)$ has a unique
dual basis $(e_i')$ in $\hc(\one,Z^*)$, and
\[
 b_Z(\alpha,\beta)=\sum_i b_Z(\alpha,e_i)b_Z(e_i',\beta).
\]
\end{lemma}
\begin{proof}
The duality adjunction \cite[Proposition~2.10.8]{EGNO15} identifies
$\hc(\one,Z^*)$ with $\hc(Z,\one)$,
under which $b_Z$ becomes composition. This is a perfect pairing on the
multiplicity space of $\one$ in the semisimple object $Z$.
\end{proof}
The prime denotes linear duality for this pairing, not the categorical dual
$e_i^*:Z^*\to\one$.

\begin{lemma}\label{lem:mouv}\label{lem:basis}
In a pivotal fusion category, define
$f:\hc(\one,Y^*\otimes X^*)\to\hc(X,Y^*)$ and
$g:\hc(\one,X\otimes Y)\to\hc(Y^*,X)$ by the following bends:
\[
f(\alpha)=
\begin{tikzpicture}[baseline=(current bounding box.center),scale=.75]
\draw (-.45,.45)--(-.45,-1.15) node[below] {$Y^*$};
\draw (.45,.45)--(.45,-.35) arc(180:360:.55)--(1.55,1.2) node[above] {$X$};
\node[draw,fill=white,minimum width=1.35cm,inner sep=2pt] at (0,.45) {$\alpha$};
\node[left,font=\scriptsize] at (.45,-.1) {$X^*$};
\end{tikzpicture},\qquad
g(\beta)=
\begin{tikzpicture}[baseline=(current bounding box.center),scale=.75]
\draw (-.45,.45)--(-.45,-1.15) node[below] {$X$};
\draw (.45,.45)--(.45,-.35) arc(180:360:.55)--(1.55,1.2) node[above] {$Y^*$};
\node[draw,fill=white,minimum width=1.35cm,inner sep=2pt] at (0,.45) {$\beta$};
\node[draw,fill=white,inner sep=1pt] at (.45,-.29) {$a_Y$};

\end{tikzpicture}.
\]
They are isomorphisms and satisfy
$\tr(f(\alpha)\circ g(\beta))=\tr(g(\beta)\circ f(\alpha))=b_{X\otimes Y}(\alpha,\beta)$.
In particular, if $(u_i),(v_j)$ are dual for $b_{X\otimes Y}$, then
$\tr(f(u_i)\circ g(v_j))=\delta_{ij}$.
\end{lemma}
\begin{proof}
Both bends are duality adjunctions \cite[Proposition~2.10.8]{EGNO15}.
Closing their composite gives
\[
 \underbrace{\begin{tikzpicture}[baseline=(current bounding box.center),x=.78cm,y=.73cm,font=\scriptsize]
   \draw (1.74,1.02)--(1.74,.48) .. controls (1.74,.03) and (2.30,.03) .. (2.30,.48)
     --(2.30,1.14) .. controls (2.30,1.74) and (3.02,1.74) .. (3.02,1.14)
     --(3.02,-.80) .. controls (3.02,-1.44) and (-.30,-1.44) .. (-.30,-.93)--(-.30,-.08);
   \draw (.32,-.08)--(.32,-.66) .. controls (.32,-1.04) and (1.12,-1.04) .. (1.12,-.66)--(1.12,1.02);
   \node[draw,fill=white,minimum width=.72cm,inner sep=1pt] at (1.43,1.05) {$\alpha$};
   \node[draw,fill=white,minimum width=.72cm,inner sep=1pt] at (0,-.05) {$\beta$};
   \node[draw,fill=white,inner sep=1pt] at (.32,-.55) {$a_Y$};
   \node[draw,fill=white,inner sep=1pt] at (-.30,-.93) {$a_X$};
   \node[right] at (3.02,.20) {$X^*$};
 \end{tikzpicture}}_{\tr(g(\beta)\circ f(\alpha))}
 \;\overset{\rm zigzag}{=}\;
 \begin{tikzpicture}[baseline=(current bounding box.center),x=.77cm,y=.77cm,font=\scriptsize]
 \draw (0,.70)--(0,-.05) .. controls (0,-1.20) and (2.65,-1.20) .. (2.65,-.05)--(2.65,.70);
 \draw (.66,.70)--(.66,-.05) .. controls (.66,-.57) and (1.99,-.57) .. (1.99,-.05)--(1.99,.70);
 \node[draw,fill=white,minimum width=.72cm,inner sep=1pt] at (.33,.71) {$\beta$};
 \node[draw,fill=white,minimum width=.72cm,inner sep=1pt] at (2.32,.71) {$\alpha$};
 \node[draw,fill=white,inner sep=1pt] at (0,-.04) {$a_X$};
 \node[draw,fill=white,inner sep=1pt] at (.66,.28) {$a_Y$};
 \end{tikzpicture}
 \;\overset{\eqref{eq:dual-transposition}}{=}\;
 \underbrace{\begin{tikzpicture}[baseline=(current bounding box.center),x=.77cm,y=.77cm,font=\scriptsize]
 \draw (0,.70)--(0,-.05) .. controls (0,-1.20) and (2.65,-1.20) .. (2.65,-.05)--(2.65,.70);
 \draw (.66,.70)--(.66,-.05) .. controls (.66,-.57) and (1.99,-.57) .. (1.99,-.05)--(1.99,.70);
 \node[draw,fill=white,minimum width=.72cm,inner sep=1pt] at (.33,.71) {$\alpha$};
 \node[draw,fill=white,minimum width=.72cm,inner sep=1pt] at (2.32,.71) {$\beta$};
 \end{tikzpicture}}_{b_{X\otimes Y}(\alpha,\beta)} .
\]
Here $(a_X\otimes a_Y)\circ\beta=\beta^{**}$. Cyclicity gives the other order.
\end{proof}

\begin{proposition}[Resolution of the identity]\label{prop:basis2}
Let $T$ be an object of a pivotal fusion category. For each isomorphism class
of simple summands $X$ of $T$, choose trace-dual bases
$b_{X,i}:T\to X$ and $b'_{X,i}:X\to T$, and set
$\tau_{X,i,j}=b'_{X,i}\circ b_{X,j}$. Then
\begin{align*}
 \tau_{X,i,j}\circ\tau_{Y,k,\ell}
   &=\delta_{X,Y}\delta_{j,k}d_X^{-1}\tau_{X,i,\ell},\\
 \gamma&=\sum_{X,i,j}d_X\tr(\gamma\circ\tau_{X,j,i})\tau_{X,i,j}
       \qquad(\gamma\in\End_{\mC}(T)),\\
 \id_T&=\sum_{X,i}d_X\tau_{X,i,i}.
\end{align*}
The last identity is the graphical resolution
\[
\begin{tikzpicture}[baseline=(current bounding box.center),scale=.7]
\draw (0,1.46) node[above] {$T$} -- (0,-1.46) node[below] {$T$};
\end{tikzpicture}
=\sum_{X,i}d_X
\begin{tikzpicture}[baseline=(current bounding box.center),scale=.7]
\node[draw,fill=white,inner sep=2pt] (projection) at (0,.77) {$b_{X,i}$};
\node[draw,fill=white,inner sep=2pt] (inclusion) at (0,-.77) {$b'_{X,i}$};
\draw (0,1.46) node[above] {$T$} -- (projection.north);
\draw (projection.south) -- node[midway,right=1.5pt,font=\scriptsize] {$X$} (inclusion.north);
\draw (inclusion.south) -- (0,-1.46) node[below] {$T$};
\end{tikzpicture}.
\]
\end{proposition}
\begin{proof}
Write $T$ as a direct sum of simple objects, with inclusions $q_{X,i}$ and
projections $p_{X,i}$. One may take $b_{X,i}=p_{X,i}$ and
$b'_{X,i}=d_X^{-1}q_{X,i}$. The elements $d_X\tau_{X,i,j}$ are the ordinary
matrix units in the corresponding simple block of $\End_{\mC}(T)$.
Since $\tr(\tau_{X,i,j})=\delta_{i,j}$, the trace pairs
$\tau_{X,i,j}$ with $d_X\tau_{X,j,i}$. This proves all three formulas,
which are unchanged by replacing the bases by any trace-dual pair.
\end{proof}

\subsection{The Fuchs--Runkel--Schweigert parity theorem}\label{sub:FRS}
Fuchs, Runkel and Schweigert prove a parity theorem for absolutely simple
tensor factors over a field. They observe that no braiding is needed when the
factors in reverse order are their respective duals
\cite[second theorem, p.~L257, and Remark~(2), p.~L259]{FuRuSc}. The following formulation holds in every characteristic and allows $X$ to be arbitrary.

\begin{theorem}\label{thm:FrobSchur}
Let $\mC$ be a pivotal fusion category, $X$ any object, and $Y$ a simple
self-dual object. If $\dim_{\field}\hc(X^*\otimes X,Y)$ is odd, then
$\nu_2(Y)=1$.
\end{theorem}
We first prove the symmetry of the relevant pairing without a
semisimplicity assumption.
\begin{lemma}[Parity from a pivotal bilinear form]\label{prop:binu}\label{cor:inter}
Let $\mC$ be a $\field$-linear pivotal monoidal category and suppose that $\hc(Y,Y^*)$ is one-dimensional.
Choose a nonzero $\kappa:Y\to Y^*$, put $Z=X\otimes Y\otimes X^*$, and
define
\[
 f=a_X\otimes\kappa\otimes\id_{X^*}:Z\longrightarrow Z^*,
 \qquad
 \omega(\alpha,\beta)=\ev_Z\circ((f\circ\alpha)\otimes\beta)
 \quad(\alpha,\beta\in\hc(\one,Z)).
\]
Then $\omega(\alpha,\beta)=\nu_2(Y^*)\omega(\beta,\alpha)$ as morphisms
in $\End_{\mC}(\one)$. If moreover $\End_{\mC}(\one)=\field$, the space
$\hc(\one,Z)$ is finite-dimensional of odd dimension, and $\omega$ is
nondegenerate, then $\nu_2(Y^*)=1$.
\end{lemma}
\begin{proof}
In each coupon the three outgoing legs are $X,Y,X^*$, from left to right.
Write $K=\kappa^*\circ a_Y$ and $L=a_X^*\circ a_{X^*}$. Transpose the
nested caps by \eqref{eq:dual-transposition} and slide $f$ across:
\newcommand{\paritycaps}[4]{%
 \vcenter{\hbox{\begin{tikzpicture}[x=.84cm,y=.76cm,font=\scriptsize]
 \foreach \x in {0,.65,1.3,2.0,2.65,3.3}{\draw (\x,.94)--(\x,-.08);}
 \foreach \a/\b/\h in {1.3/2.0/.43,.65/2.65/.79,0/3.3/1.14}{
   \draw (\a,-.08)..controls(\a,-\h) and (\b,-\h)..(\b,-.08);}
 \node[draw,fill=white,minimum width=1.40cm,inner sep=2pt] at (.65,.94){$#1$};
 \node[draw,fill=white,minimum width=1.40cm,inner sep=2pt] at (2.65,.94){$#2$};
 \node[draw,fill=white,inner sep=1.4pt] at (0,.24){$a_X$};
 \node[draw,fill=white,inner sep=1.4pt] at (.65,.24){$#3$};
 \if\relax\detokenize{#4}\relax\else\node[draw,fill=white,inner sep=1.4pt] at (1.30,.24){$#4$};\fi
 \end{tikzpicture}}}}
\[
 \underbrace{\paritycaps{\alpha}{\beta}{\kappa}{}}_{\omega(\alpha,\beta)}
 \;=\;\paritycaps{\beta}{\alpha}{K}{L}
 \;=\;\nu_2(Y^*)\,\underbrace{\paritycaps{\beta}{\alpha}{\kappa}{}}_{\omega(\beta,\alpha)}.
\]
Indeed $L=\id_{X^*}$ by \eqref{eq:pivotal-duality}, and
$K=\nu_2(Y^*)\kappa$ by Lemma~\ref{lem:evnu} and adjunction.
Since $\nu_2(Y^*)^2=1$ by Proposition~\ref{prop:nupm}, the sign is
$1$ in characteristic $2$. In any other characteristic a nondegenerate
skew-symmetric form has even dimension
\cite[\S5, no.~1, Theorem~1]{BourbakiForms}, so the sign is $1$ in odd dimension.
\end{proof}

\begin{proof}[Proof of Theorem~\ref{thm:FrobSchur}]
Choose an isomorphism $\kappa:Y\to Y^*$. Then $f$ is invertible, so
Lemma~\ref{lem:bili} makes $\omega$ nondegenerate. By rigidity,
$\hc(\one,X\otimes Y\otimes X^*)\simeq\hc(X^*\otimes X,Y)$ has odd
dimension. Lemma~\ref{prop:binu} gives $\nu_2(Y)=\nu_2(Y^*)=1$.
\end{proof}

\begin{remark}\label{rk:WangCounter}
Positive multiplicity alone does not imply $\nu_2(Y)=1$, even for $X$ simple \cite{Mas,yilong}. In Wang's example, $G=\mathrm{PSU}(3,2)\cong\mathrm{SmallGroup}(72,41)$ has irreducible complex representations $X,Y$ of degrees $8,2$, respectively, with $\nu_2(Y)=-1$ and $[X^*\otimes X:Y]=2$. This is the case $p=3$ of Miquel Mart\'inez's family \cite[Theorem~A]{MiqMar}. The \href{https://github.com/sebastienpalcoux/triangular-prism-categorification/blob/f816c697fde088b261e1fe1ce703980e5dbdab51/optional/check_small_groups.g}{exhaustive GAP search} establishes that $72$ is the least possible group order for irreducible $X,Y$. Its \href{https://github.com/sebastienpalcoux/triangular-prism-categorification/blob/f816c697fde088b261e1fe1ce703980e5dbdab51/certificates/small-groups-output.txt}{output} lists just three further groups through order $128$, with identifiers $(128,764)$, $(128,801)$ and $(128,802)$. The first is Mason's example \cite{Mas}. Among finite simple groups the \href{https://github.com/sebastienpalcoux/triangular-prism-categorification/blob/f816c697fde088b261e1fe1ce703980e5dbdab51/optional/check_simple_groups.g}{corresponding search} gives $\mathrm{PSU}(3,5)$, of order $126000$, as the first example.
\end{remark}
\tikzset{tpe half arrow/.style={postaction={decorate},decoration={markings,mark=at position 1 with {\arrow{>}}}},
 tpe half reverse/.style={postaction={decorate},decoration={markings,mark=at position 1 with {\arrow{<}}}}}
\section{Monoidal tetrahedron and triangular prism}
\subsection{Monoidal tetrahedron} \label{sec:monotetra}
Two cyclic permutations of the vertices of a tetrahedron, $(2,3,4)$ and $(3,2,1)$, generate its rotation group $A_4$. We first describe the corresponding identities for morphisms.

\begin{definition} \label{def:3Rot}
Let $\mC$ be a monoidal category with left duals. Let $F$, $G$ and $H$ be functors from $\mC^3$ to $\textbf{Set}$ given on objects by $F(X,Y,Z) = \hc(\one,X \otimes Y \otimes Z)$, $G(X,Y,Z)= \hc(\one, Y \otimes Z \otimes X^{**})$ and $H(X,Y,Z)= \hc(\one,X \otimes Y \otimes Z^{**})$, for all objects $X,Y,Z$ in $\mC$. Let $\rho: F \to G$ and $\rho': G \to F$  be natural transformations defined by
$$\rho(\alpha)=
\raisebox{-.5cm}{
\begin{tikzpicture}
\draw (-0.1,0) -- ++(0,-.64) node[pos=.76,right, font=\scriptsize, inner sep=1.5pt]{$X$} arc (0:-180:.4 and .15) --++(0,.74) arc (180:0:1.5 and .4) --++(0,-.74);
\draw (0+0.6,0)--++(0,-.64) node[pos=.76,right, font=\scriptsize, inner sep=1.5pt]{$Y$} ;
\draw (.7+0.6,0)--++(0,-.64) node[pos=.76,right, font=\scriptsize, inner sep=1.5pt]{$Z$} ;
\draw (0.6,0) node [draw,fill=white,minimum width=2cm] {$\alpha$};
\end{tikzpicture}}
\ \ \text{and}  \ \
\rho'(\beta)=
\raisebox{-.5cm}{
\begin{tikzpicture}
\draw (-.7+0.6,0)--++(0,-.64) node[pos=.76,left, font=\scriptsize, inner sep=1.5pt]{$Y$} ;
\draw (0+0.6,0)--++(0,-.64) node[pos=.76,left, font=\scriptsize, inner sep=1.5pt]{$Z$} ;
\draw (.7+0.6,0)--++(0,-.64) node[pos=.76,left, font=\scriptsize, inner sep=1.5pt]{$X^{**}$}  arc (-180:0:.4 and .15) --++(0,.74) arc (0:180:1.5 and .4) --++(0,-.74);
\draw (0.6,0) node [draw,fill=white,minimum width=2cm] {$\beta$};
\end{tikzpicture}}
$$
for all $\alpha$ in $\hc(\one,X \otimes Y \otimes Z)$ and for all  $\beta$ in $\hc(\one, Y \otimes Z \otimes X^{**})$. Next, assuming the existence of a pivotal structure $a$, let $\sigma: F \to H$ and $\sigma': H \to F$  be natural transformations defined by $$\sigma(\alpha) = (\id_{X \otimes Y} \otimes a_Z) \circ \alpha \ \ \text{ and } \ \ \sigma'(\gamma) = (\id_{X \otimes Y} \otimes a^{-1}_Z) \circ \gamma,$$ for all $\alpha$ in $\hc(\one,X \otimes Y \otimes Z)$ and for all  $\gamma$ in $\hc(\one,X \otimes Y \otimes Z^{**})$.
\end{definition}

For reference, the first graphical formula is the following typed composite:
\begin{equation}\label{eq:rho-algebraic}
\rho(\alpha)=
 (\ev_X\otimes\id_{Y\otimes Z\otimes X^{**}})
 \circ(\id_{X^*}\otimes\alpha\otimes\id_{X^{**}})
 \circ\coev_{X^*}
 \in\hc(\one,Y\otimes Z\otimes X^{**}).
\end{equation}
Thus a rotation changes the ordered tensor factors; it is not a permutation supplied by a braiding. Iterated symbols such as $\sigma^2$ include the appropriate change of object: on the last factor $Z$ they mean $a_{Z^{**}}\circ a_Z:Z\to Z^{****}$, not an untyped square of $a_Z$.

\begin{lemma} \label{lem:3RotBij}
Following Definition \ref{def:3Rot}, $\rho$ and $\sigma$ are natural isomorphisms, with $\rho^{-1} = \rho'$ and $\sigma^{-1} = \sigma'$.
\end{lemma}
\begin{proof}
Put $W=Y\otimes Z$. With the left-dual convention of \cite[Proposition~2.10.8(a), p.~42]{EGNO15}, the two adjunctions identify $\hc(\one,X\otimes W)$ first with $\hc(X^*,W)$ and then with $\hc(\one,W\otimes X^{**})$. Their composite is $\rho$; the opposite bends give its inverse. The zigzag identities give
$$(\rho' \circ \rho)(\alpha) =
\raisebox{-.5cm}{
\begin{tikzpicture}
\draw (-.7+0.6,0)--++(0,-.4) arc (0:-180:.4 and .15);
\draw[color=red] (-.9,-.4) --++(0,.5) arc (180:0:1.5 and .4) --++(0,-.5) arc (-180:0:.4 and .15) --++(0,.5);
\draw (3-.1,.1) arc (0:180:2.3 and .6) --++(0,-.5);
\draw (0+0.6,0)--++(0,-.4);
\draw (.7+0.6,0)--++(0,-.4);
\draw (0.6,0) node [draw,fill=white,minimum width=2cm] {$\alpha$};
\end{tikzpicture}}
=
\raisebox{-.5cm}{
\begin{tikzpicture}
\draw[color=blue] (-.7+0.6,0)--++(0,-.4) arc (0:-180:.4 and .15) --++(0,.5) arc (0:180:.4 and .15) --++(0,-.5);
\draw (0+0.6,0)--++(0,-.4);
\draw (.7+0.6,0)--++(0,-.4);
\draw (0.6,0) node [draw,fill=white,minimum width=2cm] {$\alpha$};
\end{tikzpicture}}
=
\raisebox{-.35cm}{
\begin{tikzpicture}
\draw[color=blue] (-.7+0.6,0)--++(0,-.4);
\draw (0+0.6,0)--++(0,-.4);
\draw (.7+0.6,0)--++(0,-.4);
\draw (0.6,0) node [draw,fill=white,minimum width=2cm] {$\alpha$};
\end{tikzpicture}}
= \alpha.
$$
The same identities give $\rho\circ\rho^{\prime}=\id_G$. Naturality follows from duality naturality, and $\sigma^{-1}=\sigma'$ follows from $a_Z^{-1}\circ a_Z=\id_Z$.
\end{proof}

\begin{remark} \label{rk:FS3}
In a $\field$-linear pivotal monoidal category, the definitions give
$(\sigma^{-1}\circ\rho)(\alpha)=E_X^{(3)}(\alpha)$ when $X=Y=Z$.
If $\hc(\one,X^{\otimes3})$ is one-dimensional, Proposition~\ref{prop:nupm}
thus gives $\rho(\alpha)=\nu_3(X)\sigma(\alpha)$ and $\nu_3(X)^3=1$.
\end{remark}

\begin{definition}[Monoidal tetrahedron] \label{def:TetraCat}
Let $\mC$ be a monoidal category with left duals. Let $(X_i)_{i=1,\dots,6}$ be objects in $\mC$. Consider morphisms $\alpha \in \hc(\one,X_1 \otimes X_2 \otimes X_3)$, $\beta \in \hc(\one,X_4^* \otimes X_3^* \otimes X_5^*)$, $\gamma \in \hc(\one,X_5 \otimes X_2^* \otimes X_6^*)$ and $\delta \in \hc(\one,X_6 \otimes X_1^* \otimes X_4)$. Let us define the (labeled) \emph{monoidal tetrahedron} as follows:
$$
T(\alpha, \beta, \gamma, \delta):=
\raisebox{-2cm}{
\begin{tikzpicture}[scale=.9]
\draw (-.7+0.6,0)--++(0,-.4) node [left] {\tiny ${X_1}$} arc (0:-180:.4 and .25) --++(0,.5) arc (180:0:1.5 and .7) arc (-180:0:.6 and .4)  --++(0,2.5);
\draw (0+0.6,0)--++(0,-.4) node [right] {\tiny ${X_2}$} arc (0:-180:1 and .5) --++(0,1.7) arc (180:0:.65 and .3) arc (-180:0:.35 and .2) --++(0,1.3);
\draw (.7+0.6,0)--++(0,-.4) node [right] {\tiny ${X_3}$} arc (0:-180:1.7 and .75) --++(0,3);

\draw (-2.1-.7,2.6)--++(0,-.45) --++(0,-2.7) arc (-180:0:3.4 and 1);
\draw (-2.1+.7,2.6)--++(0,-.45) arc (-180:0:.65 and .3);

\draw (0.6-.7,2.6)--++(0,-.45)  node [left] {\tiny ${X_5}$};
\draw (0.6+.7,2.6)--++(0,-.45) arc (-180:0:.65 and .3);

\draw (3.3-.7,2.6)--++(0,-.45) node [left] {\tiny ${X_6}$};
\draw (3.3+.7,2.6)--++(0,-.45) node [right] {\tiny ${X_4}$} --++(0,-2.7);

\draw (0.6,0) node [draw,fill=white,minimum width=2cm] {$\alpha$};
\draw (-2.1,2.6) node [draw,fill=white,minimum width=2cm] {$\beta$};
\draw (0.6,2.6) node [draw,fill=white,minimum width=2cm] {$\gamma$};
\draw (3.3,2.6) node [draw,fill=white,minimum width=2cm] {$\delta$};
\end{tikzpicture}}.
$$
\end{definition}
There is also a direct algebraic reading of this diagram. Set $Z=X_1\otimes X_2\otimes X_3$ and
\begin{align}
 C_{\beta,\gamma,\delta}
 &=\bigl(\id_{X_4^*\otimes X_3^*}\otimes\ev_{X_5}
       \otimes\id_{X_2^*}\otimes\ev_{X_6}
       \otimes\id_{X_1^*\otimes X_4}\bigr)
       \circ(\beta\otimes\gamma\otimes\delta),\notag\\
 q_\alpha&=\ev_Z\circ(\id_{Z^*}\otimes\alpha):Z^*\longrightarrow\one.\notag
\end{align}
Here $C_{\beta,\gamma,\delta}\in\hc(\one,X_4^*\otimes Z^*\otimes X_4)$, and
\begin{equation}\label{eq:T-algebraic}
 T(\alpha,\beta,\gamma,\delta)
 =\ev_{X_4}\circ(\id_{X_4^*}\otimes q_\alpha\otimes\id_{X_4})
      \circ C_{\beta,\gamma,\delta}\in\End_{\mC}(\one).
\end{equation}
Associativity and unit constraints are suppressed as in \S\ref{sec:pre}. Formula~\eqref{eq:T-algebraic} contracts $X_5$ and $X_6$ first, then $Z$, then $X_4$; it provides a way to check every source and target without reading a tetrahedral picture in a preferred direction.

\begin{remark}
  From this point on, the unsimplified zigzags in Definition \ref{def:TetraCat} (and other figures from this point onwards) are only there to make the figures easier to draw.
\end{remark}

\begin{proposition} \label{prop:rel1}
Following Definition \ref{def:TetraCat}, $$T(\alpha, \beta, \gamma, \delta) = T(\rho(\alpha), \delta^{**}, \beta, \gamma) = T(\rho^{2}(\alpha),\gamma^{**}, \delta^{**}, \beta) = T(\rho^{3}(\alpha),\beta^{**},\gamma^{**}, \delta^{**}) = T(\alpha^{**},\beta^{**},\gamma^{**}, \delta^{**}),$$
where $\rho$ is the natural transformation from Definition \ref{def:3Rot}.
\end{proposition}
\begin{proof}
Apply the duality identities of Lemma~\ref{lem:trans} to the three legs of $\delta$, with $X=X_6\otimes X_1^*\otimes X_4$ and $Y=\one$. This gives

$$T(\alpha, \beta, \gamma, \delta)=
\raisebox{-2.5cm}{\begin{tikzpicture}[scale=.9]
\draw (-.7+0.6,0)--++(0,-.4) node [left] {\tiny ${X_1}$} arc (0:-180:.4 and .25) --++(0,.5) arc (180:0:1.5 and .7) --++(0,-.5);
\draw (0+0.6,0)--++(0,-.4) node [right] {\tiny ${X_2}$} arc (0:-180:1 and .5) --++(0,1.7) arc (180:0:.65 and .3) arc (-180:0:.35 and .2) --++(0,1.3);
\draw (.7+0.6,0)--++(0,-.4) node [right] {\tiny ${X_3}$} arc (0:-180:1.7 and .75) --++(0,3);

\draw (-2.1-.7,2.6)--++(0,-.45);
\draw (-2.1+.7,2.6)--++(0,-.45) arc (-180:0:.65 and .3);

\draw (0.6-.7,2.6)--++(0,-.45)  node [left] {\tiny ${X_5}$};
\draw (0.6+.7,2.6)--++(0,-1.3) arc (-180:0:.35 and .2) arc (180:0:.65 and .3) --++(0,-1.7);

\draw (-4.8-.7,2.6)--++(0,-3) arc (-180:0:4.4 and 1.5);
\draw (-4.8+.7,2.6)--++ (0,-.45) arc (-180:0:.65 and .3);

\draw (0.6,0) node [draw,fill=white,minimum width=2cm] {$\alpha$};
\draw (-2.1,2.6) node [draw,fill=white,minimum width=2cm] {$\beta$};
\draw (0.6,2.6) node [draw,fill=white,minimum width=2cm] {$\gamma$};
\draw (-4.8,2.6) --++(0,-3) arc (-180:0:3.45 and 1);
\node[draw,fill=white,minimum width=2cm] at (-4.8,2.6) {$\delta^{**}$};
\end{tikzpicture}}
= T(\rho(\alpha), \delta^{**}, \beta, \gamma).
$$
Iterating gives the next two equalities. After three rotations each leg of
$\alpha$ has passed once through the duality adjunction, so the resulting
morphism is $\alpha^{**}$ under the canonical identification
$\one^{**}\simeq\one$. This proves the last equality.
\end{proof}

\begin{proposition} \label{prop:rel2}
If $\mC$ is a spherical monoidal category, then
\begin{align*}
T(\alpha, \beta, \gamma, \delta)
& = T(\rho^{-3}(\beta), \rho^{-1}(\sigma^{2}(\gamma)), \rho(\alpha), \rho(\delta)) 
\end{align*}
where $\rho$ and $\sigma$ are the natural transformations from Definition \ref{def:3Rot}.
\end{proposition}
\begin{proof}
Cut $X_6$ and reclose it on the other side of $X_2\otimes X_3\otimes X_4$
(Lemma~\ref{lem:sph2}). This leaves the indicated $a_{X_6^*}^2$ on $\gamma$:
$$T(\alpha, \beta, \gamma, \delta)=
\raisebox{-2.5cm}{
\begin{tikzpicture}[scale=.9]
\draw (-.7+0.6,0)--++(0,-.4) node [left=.015cm] {\tiny ${X_1}$} arc (0:-180:.4 and .25) --++(0,.5) arc (180:0:1.5 and .7) arc (-180:0:.6 and .4)  --++(0,2.5);
\draw (0+0.6,0)--++(0,-.4) node [right=.015cm] {\tiny ${X_2}$} arc (0:-180:1 and .5) --++(0,1.7) arc (180:0:.65 and .3) arc (-180:0:.35 and .2) --++(0,1.3);
\draw (.7+0.6,0)--++(0,-.4) node [right=.015cm] {\tiny ${X_3}$} arc (0:-180:1.7 and .75) --++(0,3);

\draw (-2.1-.7,2.6)--++(0,-.45) --++(0,-2.5) arc (-180:0:3.4 and 1);
\draw (-2.1+.7,2.6)--++(0,-.45) arc (-180:0:.65 and .3);

\draw (0.6-.7,2.6)--++(0,-.45) node [left] {\tiny ${X_5}$};
\draw (0.6+.7,2.6)--++(0,-1)  node [draw,fill=white] {\tiny $a^{2}_{X^*_6}$} --++(0,-.5) arc (-180:0:.25) --++ (0,.3) --++ (0,1.45) arc (0:180:2.55 and .5);

\draw (3.3-.7,2.6)--++(0,-.4) arc (0:-180:.25) --++ (0,.6) arc (180:0:1.2 and .4)  --++ (0,-1.25) --++ (0,-.55) node [right=.015cm] {\tiny ${X_6^{**}}$} --++ (0,-1.45) arc (0:-180:3.9 and 1) --++ (0,3.3);
\draw (3.3+.7,2.6)--++(0,-.60) node[left,inner sep=1.8pt] {\tiny $X_4$} --++(0,-2.35);

\draw (0.6,0) node [draw,fill=white,minimum width=2cm] {$\alpha$};
\draw (-2.1,2.6) node [draw,fill=white,minimum width=2cm] {$\beta$};
\draw (0.6,2.6) node [draw,fill=white,minimum width=2cm] {$\gamma$};
\draw (3.3,2.6) node [draw,fill=white,minimum width=2cm] {$\delta$};
\end{tikzpicture}}.
$$
Straighten the five unchanged strands:
$$T(\alpha, \beta, \gamma, \delta)=
\raisebox{-2.5cm}{
\begin{tikzpicture}[scale=.9]
\draw (-.7,0) --++ (0,-.5) arc (-180:0:2 and 1) --++ (0,3) ;
\draw (0,0) --++ (0,-.5) arc (-180:0:1 and .5) --++ (0,1.75) arc (0:180:.5 and .25) arc (0:-180:.5 and .25) --++ (0,1.25);
\draw (.7,0) --++ (0,-.5) arc (-180:0:.5 and .25)  --++ (0,.75) arc (0:180:1.6 and .5) arc (0:-180:1.6 and .5) --++ (0,2.25);

\draw (2.6,2.5) --++ (0,-.5);
\draw (1.9,2.5) --++ (0,-.3) arc (0:-180:.3) --++ (0,.6) arc (180:0:1.3 and .5) --++ (0,-3.5) arc (0:-180:4.6 and 1.5);

\draw (-.7,2.5) --++ (0,-1.35);
\draw (-1.4,2.5) --++ (0,-.5)  arc (0:-180:.3) --++ (0,.8) arc (180:0:1.3 and .5) --++ (0,-.8) arc (-180:0:1 and .35) ;

\draw (-4,2.5) --++ (0,-1.35) arc (-180:0:1.65 and .5);
\draw (-3.3,2.5) --++ (0,-.75) node [draw,fill=white] {\tiny $a^{2}_{X^*_6}$} --++ (0,-.5) arc (-180:0:.3) --++ (0,1.5) arc (0:180:1.3 and .5) --++ (0,-3.5);

\draw (0,0) node [draw,fill=white,minimum width=2cm] {$\beta$};
\draw (2.6,2.5) node [draw,fill=white,minimum width=2cm] {$\delta$};
\draw (-.7,2.5) node [draw,fill=white,minimum width=2cm] {$\alpha$};
\draw (-4,2.5) node [draw,fill=white,minimum width=2cm] {$\gamma$};
\end{tikzpicture}}
$$
Bend the legs by \eqref{eq:rho-algebraic}; the four corners, in reference order, are
$\rho^{-3}(\beta),\ \rho^{-1}(\sigma^2(\gamma)),\ \rho(\alpha),\ \rho(\delta)$:
$$T(\alpha, \beta, \gamma, \delta)=
\raisebox{-2.5cm}{
\begin{tikzpicture}[scale=.9]
\draw (-1.4,-.5) --++ (0,-.5) arc (0:-180:.4 and .25) --++ (0,.6) arc (180:0:1.5 and .5) arc (-180:0:1.25 and .5) --++ (0,3);
\draw (-.7,-.5) --++ (0,-.5) arc (0:-180:1 and .5) --++ (0,1.55) arc (180:0:.8 and .25) arc (-180:0:.55 and .25) --++ (0,2);
\draw (0,-.5) --++ (0,-.5) arc (0:-180:2.35 and .75) --++ (0,3.5);

\draw (2.6,2.5) --++ (0,-.5);
\draw (1.9,2.5) --++ (0,-.3) arc (0:-180:.3) --++ (0,.6) arc (180:0:1.3 and .5) --++ (0,-3.5) arc (0:-180:4.6 and 1.5);

\draw (-.7,2.5) --++ (0,-1);
\draw (-1.4,2.5) --++ (0,-.5)  arc (0:-180:.3) --++ (0,.8) arc (180:0:1.3 and .5) --++ (0,-.8) arc (-180:0:1 and .35) ;

\draw (-4,2.5) --++ (0,-1) arc (-180:0:1.65 and .5);
\draw (-3.3,2.5) --++ (0,-.75) node [draw,fill=white] {\tiny $a^{2}_{X^*_6}$} --++ (0,-.4) arc (-180:0:.3 and .15) --++ (0,1.5) arc (0:180:1.3 and .5) --++ (0,-3.6);

\draw (-.7,-.5) node [draw,fill=white,minimum width=2cm] {$\rho^{-3}(\beta)$};
\draw (2.6,2.5) node [draw,fill=white,minimum width=2cm] {$\delta$};
\draw (-.7,2.5) node [draw,fill=white,minimum width=2cm] {$\alpha$};
\draw (-4,2.5) node [draw,fill=white,minimum width=2cm] {$\gamma$};
\end{tikzpicture}}$$
\[
= T(\rho^{-3}(\beta), \rho^{-1}(\sigma^{2}(\gamma)), \rho(\alpha), \rho(\delta)).\qedhere
\]
\end{proof}

\begin{remark} \label{rk:A4}
The permutations $(2,3,4)$ and $(3,2,1)$ in Propositions~\ref{prop:rel1}
and~\ref{prop:rel2} generate $A_4$. Indeed, the subgroup they generate
is transitive on the four vertices, lies in $A_4$, and has at least
three elements fixing vertex~$1$.
\end{remark}

\subsection{Monoidal triangular prism and equation} \label{sec:monoTPE}
The triangular prism equation compares two evaluations of the same closed diagram.

\begin{definition}[Monoidal triangular prism] \label{def:TPCat}
Let $\mC$ be a monoidal category with left duals. Let $(X_i)_{i=1,\dots,9}$ be objects in $\mC$. Consider morphisms $\alpha_i \in \hc(\one,W_i)$, with $W_1 = X_4 \otimes X_1^{*} \otimes X_6^{***}$, $W_2 = X_5 \otimes X_2^{*} \otimes X_4^{*}$, $W_3 = X_6 \otimes X_3^{*} \otimes X_5^{*}$, $W_4 = X_9 \otimes X_1 \otimes X_7^{*}$, $W_5 = X_7 \otimes X_2 \otimes X_8^{*}$, $W_6 = X_8 \otimes X_3 \otimes X_9^{***}$. The (labeled) \emph{monoidal triangular prism} is defined as:
$$TP(\alpha_1, \alpha_2, \alpha_3, \alpha_4, \alpha_5, \alpha_6):=
\raisebox{-2.5cm}{
\begin{tikzpicture}
\draw (-.7,0) --++ (0,-.4) node [right=.015cm] {\tiny ${X_4}$} --++ (0,-.9) arc (-180:0:.7 and .3);
\draw (0,0) --++ (0,-.4) arc (-180:0:1.5 and .5);
\draw (.7,0) --++ (0,-.4) arc (-180:0:.25) --++ (0,.5) arc (0:180:1.2 and .5);

\draw (-.7,-2) --++ (0,-.4) node [right=.015cm] {\tiny ${X_5}$} --++ (0,-.9) arc (-180:0:.7 and .3);
\draw (0,-2) --++ (0,-.4) arc (-180:0:1.5 and .5);
\draw (.7,-2) --++ (0,-.4) arc (-180:0:.25) --++ (0,1.1) arc (0:180:.25);

\draw (-.7,-4) --++ (0,-.4)  node [right=.015cm] {\tiny ${X_6}$} arc (0:-180:.25) --++ (0,4.5);
\draw (0,-4) --++ (0,-.4) arc (-180:0:1.5 and .5);
\draw (.7,-4) --++ (0,-.4) arc (-180:0:.25) --++ (0,1.1) arc (0:180:.25);

\draw (3-.7,0) --++ (0,-.4)  node [right=.015cm] {\tiny ${X_9}$} arc (0:-180:.25) --++ (0,.5);
\draw (3,0) --++ (0,-.4)  node [right=.015cm] {\tiny ${X_1}$};
\draw (3+.7,0) --++ (0,-1.3) arc (0:-180:.7 and .3);

\draw (3-.7,-2) --++ (0,-.4)  node [right=.015cm] {\tiny ${X_7}$} arc (0:-180:.25) --++ (0,1.1) arc (180:0:.25);
\draw (3,-2) --++ (0,-.4)  node [right=.015cm] {\tiny ${X_2}$};
\draw (3+.7,-2) --++ (0,-1.3) arc (0:-180:.7 and .3);

\draw (3-.7,-4) --++ (0,-.4)  node [right=.015cm] {\tiny ${X_8}$} arc (0:-180:.25) --++ (0,1.1) arc (180:0:.25);
\draw (3,-4) --++ (0,-.4) node [right=.015cm] {\tiny ${X_3}$};
\draw (3+.7,-4) --++ (0,-.4) arc (-180:0:.25) --++ (0,4.5) arc (0:180:1.2 and .5);

\draw (0,0)  node [draw,fill=white,minimum width=2cm] {$\alpha_1$};
\draw (0,-2)  node [draw,fill=white,minimum width=2cm] {$\alpha_2$};
\draw (0,-4)  node [draw,fill=white,minimum width=2cm] {$\alpha_3$};
\draw (3,0)  node [draw,fill=white,minimum width=2cm] {$\alpha_4$};
\draw (3,-2)  node [draw,fill=white,minimum width=2cm] {$\alpha_5$};
\draw (3,-4)  node [draw,fill=white,minimum width=2cm] {$\alpha_6$};
\end{tikzpicture}}.
$$
\end{definition}

Here is an algebraic reading of the prism. Write $Z=X_1\otimes X_2\otimes X_3$.
For each vertex $\alpha_i:\one\to P\otimes Q\otimes R^*$ let
$f_i=(\id_{P\otimes Q}\otimes\ev_R)\circ(\alpha_i\otimes\id_R):
R\to P\otimes Q$. Along the two triangular faces these maps compose to give
\[
\begin{aligned}
 F&=(f_3\otimes\id_{X_2^*\otimes X_1^*})
       \circ(f_2\otimes\id_{X_1^*})\circ f_1
       :X_6^{**}\longrightarrow X_6\otimes Z^*,\\
 G&=(f_4\otimes\id_{X_2\otimes X_3})
       \circ(f_5\otimes\id_{X_3})\circ f_6
       :X_9^{**}\longrightarrow X_9\otimes Z.
\end{aligned}
\]
For $h:R^{**}\to R\otimes W$ put
$c_R(h)=(\ev_R\otimes\id_W)\circ(\id_{R^*}\otimes h)\circ\coev_{R^*}
:\one\to W$. Closing the two faces and evaluating the three joining edges gives
\begin{equation}\label{eq:TP-algebraic}
 TP(\alpha_1,\ldots,\alpha_6)
   =\ev_Z\circ\bigl(c_{X_6}(F)\otimes c_{X_9}(G)\bigr)
   \in\End_{\mC}(\one).
\end{equation}

\begin{remark} \label{rk:SymTP}
Number the vertices of the two triangular faces cyclically as $(1,2,3)$ and $(4,5,6)$, with corresponding vertices joined by edges. The rotations of the regular triangular prism are generated by
\[
 r=(1\,2\,3)(4\,5\,6),\qquad s=(1\,4)(2\,6)(3\,5).
\]
The middle drawing shows the initial numbering. In each adjacent drawing the
labels have moved with the indicated rotation.
\newcommand{\numberedprism}[6]{%
 \begin{tikzpicture}[baseline=5mm,x=6.3mm,y=5.9mm,
                      every node/.style={font=\scriptsize,inner sep=1.5pt}]
 \draw (0,2)--(1,1)--(0,0)--cycle (3,2)--(2,1)--(3,0)--cycle;
 \draw (0,2)--(3,2) (1,1)--(2,1) (0,0)--(3,0);
 \foreach \p in {(0,2),(1,1),(0,0),(3,2),(2,1),(3,0)}
   \fill \p circle[radius=.7pt];
 \node[above left] at (0,2) {$#1$}; \node at (1.18,1.46) {$#2$};
 \node[below left] at (0,0) {$#3$}; \node[above right] at (3,2) {$#4$};
 \node at (1.82,1.46) {$#5$}; \node[below right] at (3,0) {$#6$};
 \end{tikzpicture}}
\[
 \numberedprism{3}{1}{2}{6}{4}{5}
 \quad\xleftarrow{\ r\ }\quad
 \numberedprism{1}{2}{3}{4}{5}{6}
 \quad\xrightarrow{\ s\ }\quad
 \numberedprism{4}{6}{5}{1}{3}{2}.
\]
The relations $r^3=s^2=1$ and $srs=r^{-1}$ exhibit the six rotations
$r^i,sr^i$ ($0\le i<3$), forming a group isomorphic to $S_3$.
The permutation $(1\,4)(2\,5)(3\,6)$ is the reflection across the plane midway between
the triangular faces. In a spherical monoidal category, rotating the
graph with its edge orientations, object labels and marked corners
preserves its evaluation by spherical isotopy.
\end{remark}

\begin{example}[A complete elementary prism evaluation]\label{ex:toy-TPE}
In $\mathrm{Vec}_{\mathbb C}$, use the canonical identifications with double duals. Let $V=\mathbb C^n$, set $X_1=X_2=X_3=\one$ and $X_4=\cdots=X_9=V$, and choose all six vertices in Definition~\ref{def:TPCat} to be $\coev_V=\sum_a e_a\otimes e^a$ (with the unit factors inserted). The dashed edges represent the unit:
\[
\begin{tikzpicture}[baseline=.7cm,scale=.65]
\draw (0,0)--(0,2)--(1,.9)--cycle;
\draw (3,0)--(3,2)--(2,.9)--cycle;
\draw[dashed] (0,0)--(3,0) node[midway,below] {$\one$};
\draw[dashed] (0,2)--(3,2) node[midway,above] {$\one$};
\draw[dashed] (1,.9)--(2,.9) node[midway,above] {$\one$};
\node[left] at (0,1) {$V$}; \node[right] at (3,1) {$V$};
\foreach \p in {(0,0),(0,2),(1,.9),(3,0),(3,2),(2,.9)}
  \fill \p circle[radius=1.6pt];
\end{tikzpicture}
\quad=\quad
\begin{tikzpicture}[baseline=.7cm,scale=.65]
\draw (0,1) ellipse (.48 and .85); \node[left] at (-.48,1) {$V$};
\draw (1.6,1) ellipse (.48 and .85); \node[right] at (2.08,1) {$V$};
\end{tikzpicture}
\quad=\quad n\cdot n.
\]
Cutting the three unit edges gives two tetrahedra, each a loop of value $n$.
For the other decomposition the only simple object is $\one$.
Identify each vertex space with $V\otimes V^*\simeq\End_{\mathrm{Vec}_{\mathbb C}}(V)$.
The matrix units $E_{ab}$ and $E_{ba}$ are trace-dual since
\[
 \operatorname{Tr}(E_{ab}\circ E_{dc})=\delta_{ac}\delta_{bd}.
\]
Each tetrahedron pairs a matrix unit on one cut with the dual basis element on the next, giving
\[
 \sum_{a,b,c,d,e,f=1}^n
    (\delta_{ac}\delta_{bd})
    (\delta_{ce}\delta_{df})
    (\delta_{ea}\delta_{fb})
 =\sum_{a,b=1}^n1=n^2.
\]
The simple-object sum has one term, while each vertex-space basis has $n^2$ elements.
\end{example}

\begin{theorem}[TPE] \label{thm:TPE}
Let $\mC$ be a pivotal fusion category. Let $a$ be the pivotal structure. Let $(X_i)_{i=1,\dots,9}$ and $(\alpha_i)_{i=1,\dots,6}$ be objects and morphisms in $\mC$ as in Definition \ref{def:TPCat}. Then  $$
\sum_{\beta_0 \in B_0} T\left(\rho^{-2}(\alpha_2),\rho(\alpha_3),\rho^{-1}(\alpha_1),\beta_0 \right) T\left(\rho^{-1}(\alpha_5),\beta'_0,\rho(\alpha_4),\rho^{-1}(\alpha_6)\right) = $$ $$ \sum_{X \in \mathcal{O}} \sum_{\substack{\beta_i \in B_i, \\ i=1,2,3}} \dim(X) T\left(\rho^{-2}(\sigma(\beta'_3)),\rho(\alpha_3),\rho^2(\beta_1),\rho^{-1}(\alpha_6) \right) T\left(\rho^{-1}(\alpha_4),\alpha_1,\beta'_1,\beta_2\right) T\left(\rho^{-1}(\alpha_5),\alpha_2,\beta'_2,\beta_3\right),$$
where $B_i$ and $B'_i$ are the bases of $\hc(\one,Z_i)$ and $\hc(\one,Z_i^*)$ from Lemma \ref{lem:bili} ($x \mapsto x'$ denotes the bijection from $B_i$ to $B'_i$), with $Z_0 := X_1 \otimes X_2 \otimes X_3$, and for $i=1,2,3$, $Z_i := X \otimes Y_i$, where $Y_1:= X_9 \otimes X_6^{*}$, $Y_2:=X_7 \otimes {}^*X_4$,  $Y_3:=X_8 \otimes {}^*X_5$, and $\mathcal{O}$ is the set of simple subobjects of all three objects $Y_1^*,Y_2^*$ and $Y_3^*$ (up to isomorphism). Thus $\mathcal O$ contains one representative of each common simple constituent, with no multiplicity. For $i=1,2,3$, the objects and bases in the inner sum depend on $X$: more explicitly, they are $Z_i(X)$, $B_i(X)$ and $B'_i(X)$. Only $Z_0,B_0,B'_0$ are independent of $X$. A basis and its dual are chosen separately for each indicated pairing. As before, $\rho$ and $\sigma$ are the natural transformations from Definition \ref{def:3Rot}.
\end{theorem}
\begin{proof}
Write $TP=TP(\alpha_1,\ldots,\alpha_6)$.

\emph{Step 1: expand the middle cut.} On cutting $X_1,X_2,X_3$, the two sides
of the prism give $u:\one\to Z_0^*$ and $v:\one\to Z_0$. Expand
$b_{Z_0}(u,v)=\sum_{\beta_0}b_{Z_0}(u,\beta_0)b_{Z_0}(\beta'_0,v)$
by Lemma~\ref{lem:bili}:
$$ TP =  \sum_{\beta_0 \in B_0}
\raisebox{-2.8cm}{
\begin{tikzpicture}
\draw (-.7,0) --++ (0,-.4) node [right=.015cm] {\tiny ${X_4}$} --++ (0,-.9) arc (-180:0:.7 and .3);
\draw (0,0) --++ (0,-.4) arc (-180:0:.85 and .5) --++ (0,.4);
\draw (.7,0) --++ (0,-.4) arc (-180:0:.25) --++ (0,.5) arc (0:180:1.2 and .5);

\draw (-.7,-2) --++ (0,-.4) node [right=.015cm] {\tiny ${X_5}$} --++ (0,-.9) arc (-180:0:.7 and .3);
\draw (.7,-2) --++ (0,-.4) arc (-180:0:.25) --++ (0,1.1) arc (0:180:.25);

\draw (-.7,-4) --++ (0,-.4)  node [right=.015cm] {\tiny ${X_6}$} arc (0:-180:.25) --++ (0,4.5);
\draw (0,-4) --++ (0,-.4) arc (-180:0:1.55 and .5) --++ (0,4.4);
\draw (.7,-4) --++ (0,-.4) arc (-180:0:.25) --++ (0,1.1) arc (0:180:.25);

\draw (0,-2) --++ (0,-.4) arc (-180:0:1.2 and .5) --++ (0,2.4) node [draw,fill=white,minimum width=2cm, minimum height=.6cm] {$\beta_0$};

\draw (0,0)  node [draw,fill=white,minimum width=2cm] {$\alpha_1$};
\draw (0,-2)  node [draw,fill=white,minimum width=2cm] {$\alpha_2$};
\draw (0,-4)  node [draw,fill=white,minimum width=2cm] {$\alpha_3$};
\end{tikzpicture}}
\raisebox{-2.8cm}{
\begin{tikzpicture}
\draw (3-.7,0) --++ (0,-.4)  node [right=.015cm] {\tiny ${X_9}$} arc (0:-180:.25) --++ (0,.5);
\draw (3,0) --++ (0,-.4)  node [right=.015cm] {\tiny ${X_1}$} arc (0:-180:.85 and .5) --++ (0,.4);
\draw (3+.7,0) --++ (0,-1.3) arc (0:-180:.7 and .3);

\draw (3-.7,-2) --++ (0,-.4)  node [right=.015cm] {\tiny ${X_7}$} arc (0:-180:.25) --++ (0,1.1) arc (180:0:.25);
\draw (3+.7,-2) --++ (0,-1.3) arc (0:-180:.7 and .3);

\draw (3-.7,-4) --++ (0,-.4)  node [right=.015cm] {\tiny ${X_8}$} arc (0:-180:.25) --++ (0,1.1) arc (180:0:.25);
\draw (3,-4) --++ (0,-.4) node [right=.015cm] {\tiny ${X_3}$} arc (0:-180:1.55 and .5) --++ (0,4.4);
\draw (3+.7,-4) --++ (0,-.4) arc (-180:0:.25) --++ (0,4.5) arc (0:180:1.2 and .5);

\draw (3,-2) --++ (0,-.4)  node [right=.015cm] {\tiny ${X_2}$} arc (0:-180:1.2 and .5) --++ (0,2.4) node [draw,fill=white,minimum width=2cm,minimum height=.6cm] {$\beta'_0$};

\draw (3,0)  node [draw,fill=white,minimum width=2cm] {$\alpha_4$};
\draw (3,-2)  node [draw,fill=white,minimum width=2cm] {$\alpha_5$};
\draw (3,-4)  node [draw,fill=white,minimum width=2cm] {$\alpha_6$};
\end{tikzpicture}}
$$
Apply the cyclic adjunction \eqref{eq:rho-algebraic} at the outer vertices
and straighten the adjacent cup--cap pairs:
$$ TP =  \sum_{\beta_0 \in B_0}
\raisebox{-2cm}{
\begin{tikzpicture}
\draw (-.7,0) --++ (0,-.4) node [right=.015cm] {\tiny ${X_4}$} --++ (0,-.9) arc (-180:0:.7 and .3);
\draw (0,0) --++ (0,-.4) arc (-180:0:.85 and .5) --++ (0,.4);
\draw (.7,0) --++ (0,-.4) arc (-180:0:.25) --++ (0,.5) arc (0:180:2.45 and .5);

\draw (-.7,-2) --++ (0,-.4) node [right=.015cm] {\tiny ${X_5}$} arc (0:-180:.55 and .5) --++ (0,2.4);
\draw (.7,-2) --++ (0,-.4) arc (-180:0:.25) --++ (0,1.1) arc (0:180:.25);

\draw (-2.5-.7,0) --++ (0,-.4)  node [right=.015cm] {\tiny ${X_6}$} arc (0:-180:.25) --++ (0,.5);
\draw (-2.5,0) --++ (0,-2.6) arc (-180:0:2.8 and .5) --++ (0,2.6);
\draw (-2.5+.7,0) --++ (0,-.4);

\draw (0,-2) --++ (0,-.4) arc (-180:0:1.2 and .5) --++ (0,2.4) node [draw,fill=white,minimum width=2cm, minimum height=.6cm] {$\beta_0$};

\draw (0,0)  node [draw,fill=white,minimum width=2cm] {$\alpha_1$};
\draw (0,-2)  node [draw,fill=white,minimum width=2cm] {$\alpha_2$};
\draw (-2.5,0)  node [draw,fill=white,minimum width=2cm] {$\alpha_3$};
\end{tikzpicture}}
\raisebox{-2cm}{
\begin{tikzpicture}
\draw (3-.7,0) --++ (0,-.4)  node [right=.015cm] {\tiny ${X_9}$} arc (0:-180:.25) --++ (0,.5);
\draw (3,0) --++ (0,-.4)  node [right=.015cm] {\tiny ${X_1}$} arc (0:-180:.85 and .5) --++ (0,.4);
\draw (3+.7,0) --++ (0,-1.3) arc (0:-180:.7 and .3);

\draw (3-.7,-2) --++ (0,-.4)  node [right=.015cm] {\tiny ${X_7}$} arc (0:-180:.25) --++ (0,1.1) arc (180:0:.25);
\draw (3+.7,-2) --++ (0,-.4) arc (-180:0:.55 and .5) --++ (0,2.4);

\draw (5.5-.7,0) --++ (0,-.4)  node [right=.015cm] {\tiny ${X_8}$};
\draw (5.5,0) --++ (0,-.4) node [right=.015cm] {\tiny ${X_3}$} --++ (0,-2.2) arc (0:-180:2.8 and .5) --++ (0,2.6);
\draw (5.5+.7,0) --++ (0,-.4) arc (-180:0:.25) --++ (0,.5) arc (0:180:2.45 and .5);

\draw (3,-2) --++ (0,-.4)  node [right=.015cm] {\tiny ${X_2}$} arc (0:-180:1.2 and .5) --++ (0,2.4) node [draw,fill=white,minimum width=2cm,minimum height=.6cm] {$\beta'_0$};

\draw (3,0)  node [draw,fill=white,minimum width=2cm] {$\alpha_4$};
\draw (3,-2)  node [draw,fill=white,minimum width=2cm] {$\alpha_5$};
\draw (5.5,0)  node [draw,fill=white,minimum width=2cm] {$\alpha_6$};
\end{tikzpicture}},
$$
The remaining zigzags straighten, giving the two reference tetrahedra:
$$
TP= \sum_{\beta_0 \in B_0} T\left(\rho^{-2}(\alpha_2),\rho(\alpha_3),\rho^{-1}(\alpha_1),\beta_0 \right) T\left(\rho^{-1}(\alpha_5),\beta'_0,\rho(\alpha_4),\rho^{-1}(\alpha_6)\right)=LHS
$$

\emph{Step 2: resolve the three side cuts.} Choose representatives of the simple objects of $\mC$ that include those in $\mathcal O$.
For $T_j=Y_j^*$, let $\mathcal O(T_j)$ consist of the representatives occurring in $T_j$.
For each $S_j\in\mathcal O(T_j)$, choose a basis $B_j(S_j)$ of $\hc(\one,S_j\otimes Y_j)$
and its evaluation-dual basis $B'_j(S_j)$ as in Lemma~\ref{lem:bili},
using the bases in the statement whenever $S_j\in\mathcal O$.
Insert the resolution of Proposition~\ref{prop:basis2} at each of the three cuts:
$$ TP =  \sum_{\substack{S_i \in \mathcal{O}(T_i), \\ i=1,2,3}} \sum_{\substack{\beta_i \in B_i(S_i), \\ i=1,2,3}} \prod_{i=1}^3 \dim(S_i)
\raisebox{-5cm}{
\begin{tikzpicture}
\draw (-1-.7,-.5) --++ (0,-.4) node [right=.015cm] {\tiny ${X_4}$} --++ (0,-2.15) arc (-180:0:.35 and .3) arc (180:0:.4 and .3);
\draw (0-1,-.5) --++ (0,-.4) arc (-180:0:2.5 and .8);
\draw (.7-1,-.5) --++ (0,-.5) arc (-180:0:.55 and .3);

\draw (-.7-1,-4) --++ (0,-.4) node [right=.015cm] {\tiny ${X_5}$} --++ (0,-2.15) arc (-180:0:.35 and .3) arc (180:0:.4 and .3);
\draw (0-1,-4)--++ (0,-.4) arc (-180:0:2.5 and .8);
\draw (.7-1,-4) --++ (0,-.5) arc (-180:0:.55 and .3);

\draw (-.7-1,-8+.5) --++ (0,-.4)  node [right=.015cm] {\tiny ${X_6}$} arc (0:-180:.25) --++ (0,8.35) arc (180:0:1 and .7);
\draw (0-1,-8+.5) --++ (0,-.4) arc (-180:0:2.5 and .8);
\draw (.7-1,-8+.5) --++ (0,-.5) arc (-180:0:.55 and .3);

\draw (4-.7,-.5) --++ (0,-.4)  node [right=.015cm] {\tiny ${X_9}$} --++ (0,-.1) arc (0:-180:.9 and .3);
\draw (4,-.5) --++ (0,-.4) node [right=.015cm] {\tiny ${X_1}$} --++ (0,0);
\draw (4+.7,-.5) --++ (0,-2) arc (0:-180:.9 and .5) arc (0:180:1.4 and .5) --++ (0,-.55);

\draw (4-.7,-4) --++ (0,-.4)   node [right=.015cm] {\tiny ${X_7}$} --++ (0,-.1) arc (0:-180:.9 and .3);
\draw (4,-4) --++ (0,-.4)  node [right=.015cm] {\tiny ${X_2}$} --++ (0,0) ;
\draw (4+.7,-4) --++ (0,-2) arc (0:-180:.9 and .5) arc (0:180:1.4 and .5) --++ (0,-.55);

\draw (4-.7,-8+.5) --++ (0,-.4)  node [right=.015cm] {\tiny ${X_8}$} --++ (0,-.1) arc (0:-180:.9 and .3);
\draw (4,-8+.5) --++ (0,-.4) node [right=.015cm] {\tiny ${X_3}$} --++ (0,0);
\draw (4+.7,-8+.5) --++ (0,-.4) arc (-180:0:.25) --++ (0,8.8);

\draw (-1,-.5)  node [draw,fill=white,minimum width=2cm] {$\alpha_1$};
\draw (-1,-4)  node [draw,fill=white,minimum width=2cm] {$\alpha_2$};
\draw (-1,-8+.5)  node [draw,fill=white,minimum width=2cm] {$\alpha_3$};
\draw (4,-.5)  node [draw,fill=white,minimum width=2cm] {$\alpha_4$};
\draw (4,-4)  node [draw,fill=white,minimum width=2cm] {$\alpha_5$};
\draw (4,-8+.5)  node [draw,fill=white,minimum width=2cm] {$\alpha_6$};
\draw (1.5-.7,1.25-.5) --++ (0,-.3) arc (0:-180:.225) --++ (0,.4) arc (180:0:1.15 and .5) --++ (0,-.72) node [draw,fill=white] {\tiny $a_{S_1}^{-1}$} --++ (0,-1.1) node [right=.015cm] {\tiny $S_1$};
\draw (1.5,1.25-.5) --++ (0,-.3) arc (0:-180:.7 and .35) --++ (0,.44)  arc (180:0:2.55 and .7);
\draw (1.5+.7,1.25-.5) --++ (0,-.3) arc (0:-180:1.2 and .5);

\draw (1.5-.7,0-.5) --++ (0,-.5);
\draw (1.5,0-.5)  --++ (0,-.5);
\draw (1.5+.7,0-.5) --++ (0,-.45) arc (-180:0:.225 and .15);
\draw (1.5-.35,-.6-.5);

\draw (1.5,1.25-.5) node [draw,fill=white,minimum width=2cm] {$\beta_1$};
\draw (1.5,0-.5) node [draw,fill=white,minimum width=2cm] {$\beta'_1$};
\draw (1.5-.7,1.25-4) --++ (0,-.3) arc (0:-180:.225) --++ (0,.4) arc (180:0:1.15 and .5) --++ (0,-.72) node [draw,fill=white] {\tiny $a_{S_2}^{-1}$} --++ (0,-1.1) node [right=.015cm] {\tiny ${S_2}$};
\draw (1.5,1.25-4) --++ (0,-.3) arc (0:-180:.7 and .35);
\draw (1.5+.7,1.25-4) --++ (0,-.3) arc (0:-180:1.2 and .5);

\draw (1.5-.7,-4) --++ (0,-.5);
\draw (1.5,-4)  --++ (0,-.5);
\draw (1.5+.7,-4) --++ (0,-.45) arc (-180:0:.225 and .15);
\draw (1.5-.35,-.6-4);

\draw (1.5,1.25-4) node [draw,fill=white,minimum width=2cm] {$\beta_2$};
\draw (1.5,-4) node [draw,fill=white,minimum width=2cm] {$\beta'_2$};
\draw (1.5-.7,1.25-8+.5) --++ (0,-.3) arc (0:-180:.225) --++ (0,.4) arc (180:0:1.15 and .5) --++ (0,-.72) node [draw,fill=white] {\tiny $a_{S_3}^{-1}$} --++ (0,-1.1) node [right=.015cm] {\tiny ${S_3}$};
\draw (1.5,1.25-8+.5) --++ (0,-.3) arc (0:-180:.7 and .35);
\draw (1.5+.7,1.25-8+.5) --++ (0,-.3) arc (0:-180:1.2 and .5);

\draw (1.5-.7,-8+.5) --++ (0,-.5);
\draw (1.5,-8+.5)  --++ (0,-.5);
\draw (1.5+.7,-8+.5) --++ (0,-.45) arc (-180:0:.225 and .15);
\draw (1.5-.35,-.6-8+.5);

\draw (1.5,1.25-8+.5) node [draw,fill=white,minimum width=2cm] {$\beta_3$};
\draw (1.5,-8+.5) node [draw,fill=white,minimum width=2cm] {$\beta'_3$};
\draw[decorate,decoration={brace,amplitude=3pt}] (-2.55,-3.15)--(-2.55,-.26)
 node[midway,left=4pt] {$B$};
\draw[decorate,decoration={brace,amplitude=3pt}] (-2.55,-6.65)--(-2.55,-3.76)
 node[midway,left=4pt] {$C$};
\node at (-2.75,.81) {$A$};
\node at (-2.75,-7.55) {$D$};
\end{tikzpicture}}.
$$
\emph{Step 3: close the simple-object blocks.} Keep the two outside strands;
contract each marked block to its coupon:
\begin{equation}\label{eq:prism-Schur}
\vcenter{\hbox{\begin{tikzpicture}
\draw (-.15,0) --++ (0,.3) arc (0:180:.2) --++ (0,-3.6);
\draw (.15,0) --++ (0,.3) arc (180:0:.2) --++ (0,-3.6);
\draw (0,0) --++ (0,-.5) node [right=.015cm] {\tiny ${S_1}$} --++ (0,-.5);
\draw (0,-1) --++ (0,-.5) node [right=.015cm] {\tiny ${S_2}$} --++ (0,-.5);
\draw (0,-2) --++ (0,-.5) node [right=.015cm] {\tiny ${S_3}$} --++ (0,-.5);
\draw (-.15,-3) --++ (0,-.3) arc (0:-180:.2);
\draw (.15,-3) --++ (0,-.3) arc (-180:0:.2);

\draw (0,0) node [draw,fill=white] {$A$};
\draw (0,-1) node [draw,fill=white] {$B$};
\draw (0,-2) node [draw,fill=white] {$C$};
\draw (0,-3) node [draw,fill=white] {$D$};
\end{tikzpicture}}}
 =\tr(Q\circ C\circ B)
 =\frac{\tr(Q)\tr(B)\tr(C)}{\dim(X)^2}
 \quad(S_1=S_2=S_3=X).
\end{equation}
Here $Q:S_3\to S_1$ closes $A,D$ along the outside strands. By simplicity,
$B:S_1\to S_2$ and $C:S_2\to S_3$ give zero unless $S_1=S_2=S_3=X\in\mathcal O$. The three bases are then $B_i(X)$, and
\eqref{eq:prism-Schur} follows from Lemma~\ref{lem:simple}.
The categorical dimensions from the three cuts therefore give $\dim(X)^3/\dim(X)^2=\dim(X)$:
\begin{align*}
TP &=\sum_{X \in \mathcal{O}} \sum_{\substack{\beta_i \in B_i, \\ i=1,2,3}} \dim(X)
\vcenter{\hbox{
\begin{tikzpicture}
\draw (-.7-1,-8+.5) --++ (0,-.4)  node [right=.015cm] {\tiny ${X_6}$} arc (0:-180:.25) --++ (0,1.35) arc (180:0:1 and .7);
\draw (0-1,-8+.5) --++ (0,-.4) arc (-180:0:2.5 and .8);
\draw (.7-1,-8+.5) --++ (0,-.5) arc (-180:0:.55 and .3);
\draw (4-.7,-8+.5) --++ (0,-.4)  node [right=.015cm] {\tiny ${X_8}$} --++ (0,-.1) arc (0:-180:.9 and .3);
\draw (4,-8+.5) --++ (0,-.4) node [right=.015cm] {\tiny ${X_3}$} --++ (0,0);
\draw (4+.7,-8+.5) --++ (0,-.4) arc (-180:0:.25) --++ (0,1.8);
\draw (1.5-.7,1.25-7.5) --++ (0,-.3) arc (0:-180:.225) --++ (0,.4) arc (180:0:1.15 and .5) --++ (0,-.72) node [draw,fill=white] {\tiny $a_{X}^{-1}$} --++ (0,-1.1);
\draw (1.5,1.25-7.5) --++ (0,-.3) arc (0:-180:.7 and .35) --++ (0,.44)  arc (180:0:2.55 and .7);
\draw (1.5+.7,1.25-7.5) --++ (0,-.3) arc (0:-180:1.2 and .5);
\draw (1.5-.7,-8+.5) --++ (0,-.5);
\draw (1.5,-8+.5)  --++ (0,-.5);
\draw (1.5+.7,-8+.5) --++ (0,-.45) arc (-180:0:.225 and .15);
\draw (1.5-.35,-.6-8+.5);
\draw (-1,-8+.5)  node [draw,fill=white,minimum width=2cm] {$\alpha_3$};
\draw (4,-8+.5)  node [draw,fill=white,minimum width=2cm] {$\alpha_6$};
\draw (1.5,1.25-7.5) node [draw,fill=white,minimum width=2cm] {$\beta_1$};
\draw (1.5,-8+.5) node [draw,fill=white,minimum width=2cm] {$\beta'_3$};
\end{tikzpicture}}}
\\*
&\quad{}\mathbin{\times}
\vcenter{\hbox{
\begin{tikzpicture}
\draw (-1-.7,-.5) --++ (0,-.4) node [right=.015cm] {\tiny ${X_4}$} --++ (0,-2.15) arc (-180:0:.35 and .3) arc (180:0:.4 and .3);
\draw (0-1,-.5) --++ (0,-.4) arc (-180:0:2.5 and .8);
\draw (.7-1,-.5) --++ (0,-.5) arc (-180:0:.55 and .3);
\draw (4-.7,-.5) --++ (0,-.4)  node [right=.015cm] {\tiny ${X_9}$} --++ (0,-.1) arc (0:-180:.9 and .3);
\draw (4,-.5) --++ (0,-.4) node [right=.015cm] {\tiny ${X_1}$} --++ (0,0);
\draw (4+.7,-.5) --++ (0,-2) arc (0:-180:.9 and .5) arc (0:180:1.4 and .5) --++ (0,-.55);
\draw (1.5-.7,0-.5) --++ (0,-.5);
\draw (1.5,0-.5)  --++ (0,-.5);
\draw (1.5+.7,0-.5) --++ (0,-.45) arc (-180:0:.225 and .15) --++ (0,.6) arc (180:0:1.25 and .3);
\draw (1.5-.35,-.6-.5) ;
\draw (1.5-.7,1.25-4) --++ (0,-.3) arc (0:-180:.225) --++ (0,.4) arc (180:0:1.15 and .5) --++ (0,-.2) arc (-180:0:1.25 and .5) --++ (0,2.5);
\draw (1.5,1.25-4) --++ (0,-.3) arc (0:-180:.7 and .35);
\draw (1.5+.7,1.25-4) --++ (0,-.3) arc (0:-180:1.2 and .5);
\draw (-1,-.5)  node [draw,fill=white,minimum width=2cm] {$\alpha_1$};
\draw (4,-.5)  node [draw,fill=white,minimum width=2cm] {$\alpha_4$};
\draw (1.5,0-.5) node [draw,fill=white,minimum width=2cm] {$\beta'_1$};
\draw (1.5,1.25-4) node [draw,fill=white,minimum width=2cm] {$\beta_2$};
\end{tikzpicture}}}
\mathbin{\times}
\vcenter{\hbox{
\begin{tikzpicture}
\draw (-.7-1,-4) --++ (0,-.4) node [right=.015cm] {\tiny ${X_5}$} --++ (0,-2.15) arc (-180:0:.35 and .3) arc (180:0:.4 and .3);
\draw (0-1,-4)--++ (0,-.4) arc (-180:0:2.5 and .8);
\draw (.7-1,-4) --++ (0,-.5) arc (-180:0:.55 and .3);
\draw (4-.7,-4) --++ (0,-.4)   node [right=.015cm] {\tiny ${X_7}$} --++ (0,-.1) arc (0:-180:.9 and .3);
\draw (4,-4) --++ (0,-.4)  node [right=.015cm] {\tiny ${X_2}$} --++ (0,0) ;
\draw (4+.7,-4) --++ (0,-2) arc (0:-180:.9 and .5) arc (0:180:1.4 and .5) --++ (0,-.55);
\draw (1.5-.7,-4) --++ (0,-.5);
\draw (1.5,-4)  --++ (0,-.5);
\draw (1.5+.7,-4) --++ (0,-.45) arc (-180:0:.225 and .15) --++ (0,.6) arc (180:0:1.25 and .3);
\draw (1.5-.35,-.6-4);
\draw (1.5-.7,1.25-8+.5) --++ (0,-.3) arc (0:-180:.225) --++ (0,.4) arc (180:0:1.15 and .5) --++ (0,-.2) arc (-180:0:1.25 and .5) --++ (0,2.5);
\draw (1.5,1.25-8+.5) --++ (0,-.3) arc (0:-180:.7 and .35);
\draw (1.5+.7,1.25-8+.5) --++ (0,-.3) arc (0:-180:1.2 and .5);
\draw (-1,-4)  node [draw,fill=white,minimum width=2cm] {$\alpha_2$};
\draw (4,-4)  node [draw,fill=white,minimum width=2cm] {$\alpha_5$};
\draw (1.5,-4) node [draw,fill=white,minimum width=2cm] {$\beta'_2$};
\draw (1.5,1.25-8+.5) node [draw,fill=white,minimum width=2cm] {$\beta_3$};
\end{tikzpicture}}}.
\end{align*}
\emph{Step 4: convert the three contractions into tetrahedra.} Use the zigzag identities for adjacent evaluations and coevaluations, then apply \eqref{eq:rho-algebraic} (or its inverse) at the vertices. Moving $a_X^{-1}$ across the cup acts on the last factor $X^*$ of $\beta'_3$ by $(a_X^{-1})^*=a_{X^*}$. Thus
\[
\sigma(\beta'_3)=(\id_{X_5\otimes X_8^*}\otimes a_{X^*})\circ\beta'_3.
\]
Applying $\rho^{-2}$ gives the first argument of the first tetrahedron. The duality compatibility of Lemma~\ref{lem:switch} then gives the three reference tetrahedra below.
\begin{align*}
TP &=\sum_{X \in \mathcal{O}} \sum_{\substack{\beta_i \in B_i, \\ i=1,2,3}} \dim(X)
\vcenter{\hbox{
\begin{tikzpicture}[scale=.9]
\draw (-.7+0.6,0)--++(0,-.4) node [left] {\tiny ${ }$} arc (0:-180:.4 and .25) --++(0,.5) arc (180:0:1.5 and .7) arc (-180:0:.6 and .4)  --++(0,2.5);
\draw (0+0.6,0)--++(0,-.4) node [right] {\tiny ${ }$} arc (0:-180:1 and .5) --++(0,1.7) arc (180:0:.65 and .3) arc (-180:0:.35 and .2) --++(0,1.3);
\draw (.7+0.6,0)--++(0,-.4) node [right] {\tiny ${ }$} arc (0:-180:1.7 and .75) --++(0,3);
\draw (-2.1-.7,2.6)--++(0,-.45) --++(0,-2.7) arc (-180:0:3.4 and 1);
\draw (-2.1+.7,2.6)--++(0,-.45) arc (-180:0:.65 and .3);
\draw (0.6-.7,2.6)--++(0,-.45)  node [left] {\tiny ${ }$};
\draw (0.6+.7,2.6)--++(0,-.45) arc (-180:0:.65 and .3);
\draw (3.3-.7,2.6)--++(0,-.45) node [left] {\tiny ${ }$};
\draw (3.3+.7,2.6)--++(0,-.45) node [right] {\tiny ${ }$} --++(0,-2.7);
\draw (0.6,0) node [draw,fill=white,minimum width=2cm] {$\rho^{-2}(\sigma(\beta'_3))$};
\draw (-2.1,2.6) node [draw,fill=white,minimum width=2cm] {$\rho(\alpha_3)$};
\draw (0.6,2.6) node [draw,fill=white,minimum width=2cm] {$\rho^2(\beta_1)$};
\draw (3.3,2.6) node [draw,fill=white,minimum width=2cm] {$\rho^{-1}(\alpha_6)$};
\end{tikzpicture}}}
\\*
&\quad{}\mathbin{\times}
\vcenter{\hbox{
\begin{tikzpicture}[scale=.9]
\draw (-.7+0.6,0)--++(0,-.4) node [left] {\tiny ${ }$} arc (0:-180:.4 and .25) --++(0,.5) arc (180:0:1.5 and .7) arc (-180:0:.6 and .4)  --++(0,2.5);
\draw (0+0.6,0)--++(0,-.4) node [right] {\tiny ${ }$} arc (0:-180:1 and .5) --++(0,1.7) arc (180:0:.65 and .3) arc (-180:0:.35 and .2) --++(0,1.3);
\draw (.7+0.6,0)--++(0,-.4) node [right] {\tiny ${ }$} arc (0:-180:1.7 and .75) --++(0,3);
\draw (-2.1-.7,2.6)--++(0,-.45) --++(0,-2.7) arc (-180:0:3.4 and 1);
\draw (-2.1+.7,2.6)--++(0,-.45) arc (-180:0:.65 and .3);
\draw (0.6-.7,2.6)--++(0,-.45)  node [left] {\tiny ${ }$};
\draw (0.6+.7,2.6)--++(0,-.45) arc (-180:0:.65 and .3);
\draw (3.3-.7,2.6)--++(0,-.45) node [left] {\tiny ${ }$};
\draw (3.3+.7,2.6)--++(0,-.45) node [right] {\tiny ${ }$} --++(0,-2.7);
\draw (0.6,0) node [draw,fill=white,minimum width=2cm] {$\rho^{-1}(\alpha_4)$};
\draw (-2.1,2.6) node [draw,fill=white,minimum width=2cm] {$\alpha_1$};
\draw (0.6,2.6) node [draw,fill=white,minimum width=2cm] {$\beta'_1$};
\draw (3.3,2.6) node [draw,fill=white,minimum width=2cm] {$\beta_2$};
\end{tikzpicture}}}
\mathbin{\times}
\vcenter{\hbox{
\begin{tikzpicture}[scale=.9]
\draw (-.7+0.6,0)--++(0,-.4) node [left] {\tiny ${ }$} arc (0:-180:.4 and .25) --++(0,.5) arc (180:0:1.5 and .7) arc (-180:0:.6 and .4)  --++(0,2.5);
\draw (0+0.6,0)--++(0,-.4) node [right] {\tiny ${ }$} arc (0:-180:1 and .5) --++(0,1.7) arc (180:0:.65 and .3) arc (-180:0:.35 and .2) --++(0,1.3);
\draw (.7+0.6,0)--++(0,-.4) node [right] {\tiny ${ }$} arc (0:-180:1.7 and .75) --++(0,3);
\draw (-2.1-.7,2.6)--++(0,-.45) --++(0,-2.7) arc (-180:0:3.4 and 1);
\draw (-2.1+.7,2.6)--++(0,-.45) arc (-180:0:.65 and .3);
\draw (0.6-.7,2.6)--++(0,-.45)  node [left] {\tiny ${ }$};
\draw (0.6+.7,2.6)--++(0,-.45) arc (-180:0:.65 and .3);
\draw (3.3-.7,2.6)--++(0,-.45) node [left] {\tiny ${ }$};
\draw (3.3+.7,2.6)--++(0,-.45) node [right] {\tiny ${ }$} --++(0,-2.7);
\draw (0.6,0) node [draw,fill=white,minimum width=2cm] {$\rho^{-1}(\alpha_5)$};
\draw (-2.1,2.6) node [draw,fill=white,minimum width=2cm] {$\alpha_2$};
\draw (0.6,2.6) node [draw,fill=white,minimum width=2cm] {$\beta'_2$};
\draw (3.3,2.6) node [draw,fill=white,minimum width=2cm] {$\beta_3$};
\end{tikzpicture}}}
\end{align*}
These are precisely the three factors on the right-hand side of the theorem. Combining the two evaluations gives $TP=\mathrm{LHS}=\mathrm{RHS}$.
\end{proof}
For an explicit type check, call the two tetrahedral factors on the left $L_1,L_2$ and the three on the right $R_1,R_2,R_3$, in the order displayed in the theorem. The following table gives the six objects $(U_1,\ldots,U_6)$ to substitute for $(X_1,\ldots,X_6)$ in Definition~\ref{def:TetraCat} for each factor. Here ${}^*U$ denotes a right dual, as in the theorem.
\begin{center}
\renewcommand{\arraystretch}{1.2}
\begin{tabular}{c|cccccc}
 & $U_1$ & $U_2$ & $U_3$ & $U_4$ & $U_5$ & $U_6$\\\hline
$L_1$ & ${}^*X_2$ & ${}^*X_4$ & $X_5$ & $X_3$ & $X_6^*$ & $X_1$\\
$L_2$ & ${}^*X_8$ & $X_7$ & $X_2$ & $X_3$ & $X_1$ & $X_9^*$\\
$R_1$ & ${}^*X_8$ & $X^*$ & $X_5$ & $X_3$ & $X_6^*$ & $X_9^*$\\
$R_2$ & ${}^*X_7$ & $X_9$ & $X_1$ & ${}^*X_4$ & $X_6^{**}$ & $X$\\
$R_3$ & ${}^*X_8$ & $X_7$ & $X_2$ & ${}^*X_5$ & $X_4$ & $X$.
\end{tabular}
\end{center}
For example, the first vertex of $R_1$ has type
$\rho^{-2}(\sigma(\beta'_3)):\one\to{}^*X_8\otimes X^*\otimes X_5$, while its third has type
$\rho^2(\beta_1):\one\to X_6^*\otimes X^{**}\otimes X_9^{**}$.
These are precisely $\hc(\one,U_1\otimes U_2\otimes U_3)$ and $\hc(\one,U_5\otimes U_2^*\otimes U_6^*)$. The other entries are checked in the same way by the explicitly typed rotation \eqref{eq:rho-algebraic}.

\subsection{Oriented monoidal tetrahedron and spherical TPE} \label{bitetra}
With the oriented duality conventions of \S\ref{sec:bio}, the action of $A_4$ on a tetrahedron determines how to change each vertex corner and edge label. We use these rules to express the TPE in the spherical graphical calculus \cite[\S3]{BW96}.
\begin{definition}  \label{def:TetraPiv} Following Definition \ref{def:TetraCat},
$$
\raisebox{-2.1cm}{
\begin{tikzpicture}[scale=2]
\draw[tpe half arrow] (2,0) -- (1,0) node [below] {\tiny ${X_4}$};
\draw (1,0) -- (0,0) node [below] {$\beta$};
\draw[tpe half arrow] (2,0) node [below] {$\delta$} -- (1.5,.866) node [right] {\tiny ${X_6}$};
\draw (1.5,.866) -- (1,1.732);
\draw[tpe half arrow] (1,1.732)	node [above] {$\gamma$} -- (.5,.866) node [left] {\tiny ${X_5}$};
\draw (.5,.866) -- (0,0);
\draw[tpe half arrow] (1,0.577) -- (.5,.2885) node [above] {\tiny ${X_3}$};
\draw (.5,.2885) -- (0,0);
\draw[tpe half arrow] (1,0.577) node [below] {$\alpha$} -- (1,1.1545) node [right] {\tiny ${X_2}$};
\draw (1,1.1545) -- (1,1.732);
\draw[tpe half arrow] (1,0.577) -- (1.5,.2885) node [above] {\tiny ${X_1}$};
\draw (1.5,.2885) -- (2,0);
\end{tikzpicture}}
:=T(\alpha, \beta, \gamma, \delta).$$
\end{definition}
The graph denotes the scalar $T(\alpha,\beta,\gamma,\delta)$ of Definition~\ref{def:TetraCat}; the drawing in that definition gives its top-to-bottom reading as a string diagram.

The marked corner determines the ordered tensor factors. With that
corner at the top, the local dictionary is
\[
\begin{tikzpicture}[baseline=(current bounding box.center),scale=.8]
\draw[tpe forward] (0,0)--(-.85,.5) node[left] {$X$};
\draw[tpe forward] (0,0)--(0,-1) node[below] {$Y$};
\draw[tpe forward] (0,0)--(.85,.5) node[right] {$Z$};
\node[above] at (0,.12) {$\theta$};
\end{tikzpicture}
\quad\longleftrightarrow\quad
\begin{tikzpicture}[baseline=(current bounding box.center),scale=.8]
\node[draw,fill=white,minimum width=1.8cm,anchor=south] at (0,0) {$\theta$};
\draw[tpe forward] (-.8,0)--(-.8,-1) node[below] {$X$};
\draw[tpe forward] (0,0)--(0,-1) node[below] {$Y$};
\draw[tpe forward] (.8,0)--(.8,-1) node[below] {$Z$};
\end{tikzpicture},
\qquad
\begin{tikzpicture}[baseline=(current bounding box.center),scale=.8]
\draw[tpe forward] (0,0)--(-.85,.5) node[left] {$X$};
\draw[tpe reverse] (0,0)--(0,-1) node[below] {$Y$};
\draw[tpe forward] (0,0)--(.85,.5) node[right] {$Z$};
\node[above] at (0,.12) {$\eta$};
\end{tikzpicture}
\quad\longleftrightarrow\quad
\begin{tikzpicture}[baseline=(current bounding box.center),scale=.8]
\node[draw,fill=white,minimum width=1.8cm,anchor=south] at (0,0) {$\eta$};
\draw[tpe forward] (-.8,0)--(-.8,-1) node[below] {$X$};
\draw[tpe reverse] (0,0)--(0,-1) node[below] {$Y$};
\draw[tpe forward] (.8,0)--(.8,-1) node[below] {$Z$};
\end{tikzpicture}.
\]
Here $\theta:\one\to X\otimes Y\otimes Z$ and
$\eta:\one\to X\otimes Y^*\otimes Z$. The whole marked vertex may be
rotated to this position; its corner must not be moved independently.
Object labels may be placed anywhere along their edges.

The following dictionary fixes the distinction between the two kinds of diagrams.
\begin{center}
\renewcommand{\arraystretch}{1.22}
\begin{tabular}{p{.25\textwidth}p{.66\textwidth}}
\hline
Notation & Categorical meaning \\
\hline
Box with ordered legs & A morphism with a specified source and target, read with the convention of \S\ref{sec:pre}. \\
Vertex with a marked corner & A morphism $\theta:\one\to U\otimes V\otimes W$; the corner determines the first factor, hence the ordered triple. \\
Arrow on an edge & Specifies which endpoint contributes $X$ and which contributes $X^*$ to the evaluation. Reversing the arrow also changes its object label to a dual and changes the adjacent corner data according to \eqref{R1}--\eqref{R2}. \\
$\rho(\theta)$ & The cyclic adjunction \eqref{eq:rho-algebraic}, with target $V\otimes W\otimes U^{**}$. \\
$\sigma(\theta)$ & Composition with $\id_{U\otimes V}\otimes a_W$, with target $U\otimes V\otimes W^{**}$. \\
$\theta'$ & The basis element dual to $\theta$ under the evaluation pairing of Lemma~\ref{lem:bili}. The categorical dual is the morphism $\theta^*$. \\
Closed tetrahedral graph & The scalar \eqref{eq:T-algebraic}, after conversion to the reference pattern. \\
\hline
\end{tabular}
\end{center}

We introduce two rules,~\eqref{R1} and~\eqref{R2}. Each changes an edge orientation together with a dual label and an adjacent vertex corner. They are not independent permissions to reverse an arrow or to move a label while leaving the other data fixed. The rule is applied at the indicated vertex. At the opposite endpoint, the reversed, dual-labelled edge represents the same tensor factor, so the morphism and its marked corner remain unchanged.
In Definition~\ref{def:TetraPiv}, the tetrahedral graph, together with its prescribed edge orientations and choice of vertex corners, is taken as the \emph{reference pattern}. Other decorated tetrahedra are interpreted using these corner rules, the spherical symmetries below, and the oriented cups and caps of \S\ref{sec:bio}, retaining every pivotal morphism introduced by the conversion to the reference pattern.

The indicated orientation represents the evaluation morphism
\[
\begin{tikzpicture}[baseline=-.08cm]
\draw[tpe reverse] (0,0)--(1.4,0) node[midway,above] {$X$};
\end{tikzpicture}
\quad\longleftrightarrow\quad
\begin{tikzpicture}[baseline=(current bounding box.center),scale=.65]
\draw (0,0) node[above] {$X^*$} -- (0,-.5)
 .. controls (0,-1.2) and (1.4,-1.2) .. (1.4,-.5)
 --(1.4,0) node[above] {$X$};
\end{tikzpicture}
=\ev_{X,+}.
\]
The corresponding cup--cap cancellations are the zigzag pictures of
\S\ref{sec:bio}. In the following local rules an unarrowed leg means that
both choices of orientation are permitted; it is not the orientation-free
specialization of Remark~\ref{rk:simplified}.
\begin{align}
\label{R1}
\raisebox{-.6cm}{
\begin{tikzpicture}
\draw (0,0)  node [above] {$\theta$} -- (0.866,0.5);
\draw[tpe half arrow] (0,0) -- (-0.433,0.25) node [below=1.5pt] {\tiny $X$};
\draw (-0.433,0.25)--(-0.866,0.5);
\draw (0,0)--(0,-1);
\end{tikzpicture}}
\hspace{.5cm} =& \hspace{1cm}
\raisebox{-.6cm}{
\begin{tikzpicture}
\draw (0,0) --(0.866,0.5);
\draw[tpe half arrow] (-0.866,0.5) -- (-0.433,0.25) node [above =.05cm] {\tiny $X^*$};
\draw (-0.433,0.25) -- (0,0);
\draw (0,0)--(0,-1);
\node[anchor=north east,inner sep=0pt] at (-.08,-.12) {\small $\rho(\theta)$};
\end{tikzpicture}}
\hspace{.5cm} , \hspace{.5cm}
\raisebox{-.6cm}{
\begin{tikzpicture}
\draw (0,0)  node [above] {$\theta$} --(-0.866,0.5);
\draw[tpe half arrow] (0.866,0.5)--(0.433,0.25) node [below=1.5pt] {\tiny $X^*$};
\draw (0.433,0.25)--(0,0);
\draw (0,0)--(0,-1);
\end{tikzpicture}}
\hspace{.5cm} = \hspace{.5cm}
\raisebox{-.6cm}{
\begin{tikzpicture}
\draw (0,0) --(-0.866,0.5);
\draw[tpe half arrow] (0,0) -- (0.433,0.25) node [above=1.5pt] {\tiny $X$} ;
\draw (0.433,0.25) -- (0.866,0.5);
\draw (0,0)--(0,-1);
\node[anchor=north west,inner sep=0pt] at (.08,-.12) {\small $\rho^{-1}(\theta)$};
\end{tikzpicture}}
\\  \label{R2}
\raisebox{-.6cm}{
\begin{tikzpicture}
\draw (0,0)  node [above] {$\theta$} -- (0.866,0.5);
\draw[tpe half arrow] (-0.866,0.5) -- (-0.433,0.25) node [below=1.5pt] {\tiny $X^*$};
\draw (-0.433,0.25)--(0,0);
\draw (0,0)--(0,-1);
\end{tikzpicture}}
\hspace{.5cm} =& \hspace{.3cm}
\raisebox{-.6cm}{
\begin{tikzpicture}
\draw (0,0) --(0.866,0.5);
\draw[tpe half arrow] (0,0) -- (-0.433,0.25) node [above =0cm] {\tiny $X$};
\draw (-0.433,0.25) -- (-0.866,0.5);
\draw (0,0) --(0,-1);
\node[anchor=north east,inner sep=0pt] at (-.08,-.12) {\small $\sigma^{-2}(\rho(\theta))$};
\end{tikzpicture}}
\hspace{.5cm} , \hspace{.5cm}
\raisebox{-.6cm}{
\begin{tikzpicture}
\draw (0,0)  node [above] {$\theta$} --(-0.866,0.5);
\draw[tpe half arrow] (0,0)--(0.433,0.25) node [below=1.5pt] {\tiny $X$};
\draw (0.433,0.25)--(0.866,0.5);
\draw (0,0)--(0,-1);
\end{tikzpicture}}
\hspace{.5cm} = \hspace{.5cm}
\raisebox{-.6cm}{
\begin{tikzpicture}
\draw (0,0) -- (-0.866,0.5);
\draw[tpe half arrow] (0.866,0.5) -- (0.433,0.25) node [above=1.5pt] {\tiny $X^*$} ;
\draw (0.433,0.25) -- (0,0);
\draw (0,0)--(0,-1);
\node[anchor=north west,inner sep=0pt] at (.08,-.12) {\small $\rho^{-1}(\sigma^{2}(\theta))$};
\end{tikzpicture}}
\end{align}
\begin{proof}[Proof of the local rules]
The first rule of \eqref{R1} is the defining duality adjunction for $\rho$:

\begin{center}
\raisebox{-.6cm}{
\begin{tikzpicture}
\draw (0,0)  node [above] {$\theta$} -- (0.866,0.5);
\draw[tpe half arrow] (0,0) -- (-0.433,0.25) node [below=1.5pt] {\tiny $X$};
\draw (-0.433,0.25)--(-0.866,0.5);
\draw (0,0)--(0,-1);
\end{tikzpicture}}
 =
\raisebox{-.5cm}{
\begin{tikzpicture}
\draw[tpe forward] (-.1,0)--++(0,-.4) arc(0:-180:.4) --++(0,.4);
\node[below,font=\tiny] at (-.5,-.8) {$X$};
\draw (0+0.6,0)--++(0,-.5);
\draw (.7+0.6,0)--++(0,-.5);
\draw (0.6,0) node [draw,fill=white,minimum width=2cm] {$\theta$};
\end{tikzpicture}}
 =
\raisebox{-.5cm}{
\begin{tikzpicture}
\draw (-.1,0)--++(0,-.4) arc(0:-180:.4) --++(0,.5) arc(180:0:1.5 and .4) --++(0,-.5);
\draw[tpe reverse] (2.1,-.4) arc(-180:0:.4);
\draw (2.9,-.4)--++(0,.4);\node[below,font=\tiny] at (2.5,-.8) {$X^*$};
\draw (0+0.6,0)--++(0,-.5);
\draw (.7+0.6,0)--++(0,-.5);
\draw (0.6,0) node [draw,fill=white,minimum width=2cm] {$\theta$};
\end{tikzpicture}}
=
\raisebox{-.5cm}{
\begin{tikzpicture}
\draw (-.7+0.6,0)--++(0,-.6);
\draw (0+0.6,0)--++(0,-.6);
\draw[tpe reverse] (1.3,0)--++(0,-.4) arc(-180:0:.4) --++(0,.4);
\node[below,font=\tiny] at (1.7,-.8) {$X^*$};
\draw (0.6,0) node [draw,fill=white,minimum width=2cm] {$\rho(\theta)$};
\end{tikzpicture}}
=
\raisebox{-.6cm}{
\begin{tikzpicture}
\draw (0,0) --(0.866,0.5);
\draw[tpe half arrow] (-0.866,0.5) -- (-0.433,0.25) node [above =.05cm] {\tiny $X^*$};
\draw (-0.433,0.25) -- (0,0);
\draw (0,0)--(0,-1);
\node[anchor=north east,inner sep=0pt] at (-.08,-.12) {\small $\rho(\theta)$};
\end{tikzpicture}}
\end{center}
Reversing this calculation gives the second rule. The same adjunction with
$X$ replaced by $X^*$ gives \eqref{R2}: the pivotal-dual identities of
Lemma~\ref{lem:LeftRight} identify the resulting double-dual legs, with
$a_{X^{**}}\circ a_X$ or its inverse on the affected leg. These are precisely
$\sigma^2$ and $\sigma^{-2}$ in \eqref{R2}.
\end{proof}

\begin{example}[Reading and changing a decorated tetrahedron]\label{ex:tetra-dictionary}
Use a strictly pivotal realization of a spherical monoidal category (Remark~\ref{rk:StrPiv}), so that $(-)^{**}=\mathrm{Id}$ and $a_U=\id_U$. Write $H(U,V,W)=\hc(\one,U\otimes V\otimes W)$.

\emph{(a) Read the reference graph.} Use the graph of Definition~\ref{def:TetraPiv}, with the six objects $X_1,\ldots,X_6$ and labels
\[
\begin{array}{c|c}
\text{label}&\text{Hom-space}\\\hline
\alpha&H(X_1,X_2,X_3)\\
\beta&H(X_4^*,X_3^*,X_5^*)\\
\gamma&H(X_5,X_2^*,X_6^*)\\
\delta&H(X_6,X_1^*,X_4).
\end{array}
\]
The evaluation is $T(\alpha,\beta,\gamma,\delta)$ from \eqref{eq:T-algebraic}.

\emph{(b) Reverse one edge by \eqref{R1}.} Reverse the $X_5$ edge at $\gamma$, replacing its label by $X_5^*$. The first graph below has
\[
 \gamma_1=\rho(\gamma)\in H(X_2^*,X_6^*,X_5).
\]
At $\beta$, the corresponding tensor factor is still $X_5^*$.

\emph{(c) Move another corner by \eqref{R2}.} Reverse the last edge $X_3$ at $\alpha$ and label it $X_3^*$. Since $\sigma=\id$, the central vertex of the second graph is
\[
 \alpha_2=\rho^{-1}(\alpha)\in H(X_3,X_1,X_2).
\]
With a general pivotal structure this label would be $\rho^{-1}(\sigma^2(\alpha))$, as prescribed by \eqref{R2}.
\begin{center}
\begin{tikzpicture}[scale=1.7]
\draw[tpe half arrow] (2,0)--(1,0) node[below]{\tiny $X_4$};\draw (1,0)--(0,0);
\draw[tpe half arrow] (2,0)--(1.5,.866) node[right]{\tiny $X_6$};\draw (1.5,.866)--(1,1.732);
\draw[tpe half arrow] (0,0)--(.5,.866) node[left]{\tiny $X_5^*$};\draw (.5,.866)--(1,1.732);
\draw[tpe half arrow] (1,.577)--(.5,.2885) node[above]{\tiny $X_3$};\draw (.5,.2885)--(0,0);
\draw[tpe half arrow] (1,.577)--(1,1.1545) node[right]{\tiny $X_2$};\draw (1,1.1545)--(1,1.732);
\draw[tpe half arrow] (1,.577)--(1.5,.2885) node[above]{\tiny $X_1$};\draw (1.5,.2885)--(2,0);
\node[below] at (1,.577){\small $\alpha$};\node[below] at (0,0){\small $\beta$};
\node at (.84,1.12){\scriptsize $\gamma_1$};\node[below] at (2,0){\small $\delta$};
\end{tikzpicture}
\qquad
\begin{tikzpicture}[scale=1.7]
\draw[tpe half arrow] (2,0)--(1,0) node[below]{\tiny $X_4$};\draw (1,0)--(0,0);
\draw[tpe half arrow] (2,0)--(1.5,.866) node[right]{\tiny $X_6$};\draw (1.5,.866)--(1,1.732);
\draw[tpe half arrow] (0,0)--(.5,.866) node[left]{\tiny $X_5^*$};\draw (.5,.866)--(1,1.732);
\draw[tpe half arrow] (0,0)--(.5,.2885) node[above]{\tiny $X_3^*$};\draw (.5,.2885)--(1,.577);
\draw[tpe half arrow] (1,.577)--(1,1.1545) node[right]{\tiny $X_2$};\draw (1,1.1545)--(1,1.732);
\draw[tpe half arrow] (1,.577)--(1.5,.2885) node[above]{\tiny $X_1$};\draw (1.5,.2885)--(2,0);
\node at (.79,.69){\small $\alpha_2$};\node[below] at (0,0){\small $\beta$};
\node at (.84,1.12){\scriptsize $\gamma_1$};\node[below] at (2,0){\small $\delta$};
\end{tikzpicture}
\end{center}
\emph{(d) Translate back.} Returning both corners to the reference pattern gives
\begin{equation}\label{eq:worked-tetra}
 T\bigl(\rho(\alpha_2),\beta,
         \rho^{-1}(\gamma_1),\delta\bigr)
 =T(\alpha,\beta,\gamma,\delta).
\end{equation}
For finite-dimensional vector spaces, choose bases for the six edge objects. In these coordinates
\[
 T(\alpha,\beta,\gamma,\delta)
   =\sum_{a,b,c,d,e,f}\alpha_{abc}\beta_{dce}\gamma_{ebf}\delta_{fad}.
\]
The new coefficients $(\alpha_2)_{cab}=\alpha_{abc}$ and $(\gamma_1)_{bfe}=\gamma_{ebf}$ give the same sum.
\end{example}

\begin{lemma} \label{lem:TetraSym}
Following Definition \ref{def:TetraPiv} and above rules in the spherical case
$$
\raisebox{-2.1cm}{
\begin{tikzpicture}[scale=2]
\draw (1,0) node [below] {\tiny ${X_6}$} -- (2,0);
\draw[tpe half arrow] (0,0) node [below] {$\delta$} -- (1,0);
\draw[tpe half arrow] (2,0) node [below] {$\gamma$} -- (1.5,.866) node [right] {\tiny ${X_5}$};
\draw (1.5,.866) -- (1,1.732);
\draw (.5,.866) -- (1,1.732) node [above] {$\beta$};
\draw[tpe half arrow] (0,0) -- (.5,.866) node [left] {\tiny ${X_4}$} ;
\draw[tpe half arrow] (1,0.577) -- (.5,.2885) node [above] {\tiny ${X_1}$};
\draw (.5,.2885) -- (0,0);
\draw[tpe half arrow] (1,0.577) -- (1,1.1545) node [right] {\tiny ${X_3}$};
\draw (.84,.69) node {$\alpha$};
\draw (1,1.1545) -- (1,1.732);
\draw[tpe half arrow] (1,0.577) -- (1.5,.2885) node [above] {\tiny ${X_2}$};
\draw (1.5,.2885) -- (2,0);
\end{tikzpicture}}
=
\raisebox{-2.15cm}{
\begin{tikzpicture}[scale=2]
\draw[tpe half arrow] (2,0) -- (1,0) node [below] {\tiny ${X_4}$};
\draw (1,0) -- (0,0) node [below] {$\beta$};
\draw[tpe half arrow] (2,0) node [below] {$\delta$} -- (1.5,.866) node [right] {\tiny ${X_6}$};
\draw (1.5,.866) -- (1,1.732);
\draw[tpe half arrow] (1,1.732)	node [above] {$\gamma$} -- (.5,.866) node [left] {\tiny ${X_5}$};
\draw (.5,.866) -- (0,0);
\draw[tpe half arrow] (1,0.577) -- (.5,.2885) node [above] {\tiny ${X_3}$};
\draw (.5,.2885) -- (0,0);
\draw[tpe half arrow] (1,0.577) node [below] {$\alpha$} -- (1,1.1545) node [right] {\tiny ${X_2}$};
\draw (1,1.1545) -- (1,1.732);
\draw[tpe half arrow] (1,0.577) -- (1.5,.2885) node [above] {\tiny ${X_1}$};
\draw (1.5,.2885) -- (2,0);
\end{tikzpicture}}
=
\raisebox{-2cm}{
\begin{tikzpicture}[scale=2]
\draw[tpe half arrow] (2,0) -- (1,0) node [below] {\tiny ${X_6}$};
\draw (1,0) -- (0,0);
\draw (.44,.11) node {$\gamma$};
\draw (2,0) -- (1.5,.866) node [right] {\tiny ${X_1}$};
\draw (1.56,.11) node {$\delta$};
\draw[tpe half arrow] (1,1.732) -- (1.5,.866);
\draw[tpe half arrow] (1,1.732) -- (.5,.866) node [left] {\tiny ${X_2}$};
\draw (1.12,1.32) node {$\alpha$};
\draw (0,0) -- (.5,.866);
\draw (1,0.577) -- (.5,.2885) node [above] {\tiny ${X_5}$};
\draw[tpe half arrow] (0,0) -- (.5,.2885);
\draw (1,0.577) node [below] {$\beta$} -- (1,1.1545) node [right] {\tiny ${X_3}$};
\draw[tpe half arrow] (1,1.732) -- (1,1.1545);
\draw (1,0.577) -- (1.5,.2885) node [above] {\tiny ${X_4}$};
\draw[tpe half arrow] (2,0) -- (1.5,.2885);
\end{tikzpicture}}
$$
\end{lemma}
\begin{proof}
By Proposition \ref{prop:rel1}, $T(\alpha, \beta, \gamma, \delta) = T(\rho(\alpha), \delta^{**}, \beta, \gamma)$, but by Definition \ref{def:TetraPiv}:
$$
T(\rho(\alpha), \delta^{**}, \beta, \gamma) =
\raisebox{-2.1cm}{
\begin{tikzpicture}[scale=2]
\draw[tpe half arrow] (2,0) -- (1,0) node [below] {\tiny ${X_6^*}$};
\draw (1,0) -- (0,0) node [below] {$\delta^{**}$};
\draw[tpe half arrow] (2,0) node [below] {$\gamma$} -- (1.5,.866) node [right] {\tiny ${X_5}$};
\draw (1.5,.866) -- (1,1.732);
\draw[tpe half arrow] (1,1.732)	node [above] {$\beta$} -- (.5,.866) node [left] {\tiny ${X_4^*}$};
\draw (.5,.866) -- (0,0);
\draw[tpe half arrow] (1,0.577) -- (.5,.2885);\node[anchor=south,inner sep=0pt] at (.56,.43) {\tiny ${X_1^{**}}$};
\draw (.5,.2885) -- (0,0);
\draw[tpe half arrow] (1,0.577) node [below=.05cm] {$\rho(\alpha)$} -- (1,1.1545) node [right] {\tiny ${X_3}$};
\draw (1,1.1545) -- (1,1.732);
\draw[tpe half arrow] (1,0.577) -- (1.5,.2885) node [above] {\tiny ${X_2}$};
\draw (1.5,.2885) -- (2,0);
\end{tikzpicture}}
$$
Now $\delta^{**} = \rho^3(\delta)$, so three applications of \eqref{R1} at its vertex and one at $\rho(\alpha)$ give the first equality. For the second, Proposition~\ref{prop:rel2} gives
\[
 T(\alpha,\beta,\gamma,\delta)
 =T\bigl(\rho^{-3}(\beta),\rho^{-1}(\sigma^2(\gamma)),
         \rho(\alpha),\rho(\delta)\bigr).
\]
The reference graph of the right-hand side is
\[
\begin{tikzpicture}[scale=1.7]
\draw[tpe half arrow] (2,0)--(1,0);\draw (1,0)--(0,0);
\draw[tpe half arrow] (2,0)--(1.5,.866);\draw (1.5,.866)--(1,1.732);
\draw[tpe half arrow] (1,1.732)--(.5,.866);\draw (.5,.866)--(0,0);
\draw[tpe half arrow] (1,.577)--(.5,.2885);\draw (.5,.2885)--(0,0);
\draw[tpe half arrow] (1,.577)--(1,1.1545);\draw (1,1.1545)--(1,1.732);
\draw[tpe half arrow] (1,.577)--(1.5,.2885);\draw (1.5,.2885)--(2,0);
\node[font=\small] at (1,.24) {$\rho^{-3}(\beta)$};
\node[below left] at (0,0) {$\rho^{-1}(\sigma^2(\gamma))$};
\node[above] at (1,1.732) {$\rho(\alpha)$};
\node[below right] at (2,0) {$\rho(\delta)$};
\end{tikzpicture}
\]
Here the edge labels are determined by the displayed morphisms and are
suppressed. Apply \eqref{R1} once at $\rho(\alpha)$, three times at
$\rho^{-3}(\beta)$, then \eqref{R2} once at
$\rho^{-1}(\sigma^2(\gamma))$, and \eqref{R1} once at $\rho(\delta)$.
This gives the last graph in the statement.
\end{proof}

\begin{remark} \label{rk:A4ction}
Observe that the three pictures in Lemma \ref{lem:TetraSym} use exactly the same labels. What is changing is  the ordering, the orientation of some edges and the corner of the label of some vertices. The twelve orientation-preserving rotations of the regular tetrahedron give equal evaluations, with all corner and orientation data transported. They form the $A_4$-orbit of the marked reference tetrahedron; coincident decorations can make distinct rotations yield the same picture. The two equalities of Lemma \ref{lem:TetraSym} correspond to two generators of the group $A_4$ (see Remark \ref{rk:A4}). A marked corner is equivalently a marked incident face of the spatial
tetrahedron. The shaded examples show an interior corner and the exterior
corner at $T$:
\[
\begin{tikzpicture}[baseline=.65cm,scale=1.1]
\fill[gray!18] (1,1.732)--(0,0)--(1,.577)--cycle;
\draw (0,0)--(2,0)--(1,1.732)--cycle;
\draw (0,0)--(1,.577)--(2,0) (1,.577)--(1,1.732);
\node[above] at (1,1.732) {$T$};\node[left] at (0,0) {$L$};
\node[right] at (2,0) {$R$};\node[below] at (1,.577) {$O$};
\node at (.83,1.08) {$\alpha$};
\node[below] at (1,-.3) {\small face $TLO$};
\end{tikzpicture}
\qquad
\begin{tikzpicture}[baseline=.65cm,scale=1.1]
\fill[gray!18] (1,1.732)--(0,0)--(2,0)--cycle;
\draw (0,0)--(2,0)--(1,1.732)--cycle;
\draw (0,0)--(1,.577)--(2,0) (1,.577)--(1,1.732);
\node[right] at (1,1.732) {$T$};\node[left] at (0,0) {$L$};
\node[right] at (2,0) {$R$};\node[below] at (1,.577) {$O$};
\node[above=3pt] at (1,1.732) {$\alpha$};
\node[below] at (1,-.3) {\small hidden face $TLR$};
\end{tikzpicture}.
\]
Thus a spatial rotation transports a vertex and its chosen incident face
together, including the face hidden by the planar projection.
\end{remark}

Let us restate this into the following proposition:

\begin{proposition} \label{prop:A4Sym}
In a spherical monoidal category, let the alternating group $A_4$ act on the set of labeled oriented tetrahedral graphs as described in Remark~\ref{rk:A4ction}, and consider the orbit of the graph in Definition~\ref{def:TetraPiv}. Any element of this orbit can be transformed into the reference pattern of Definition~\ref{def:TetraPiv} by applying the rules above one or more times. Moreover, its categorical interpretation agrees with that of the original graph.
\end{proposition}
\begin{proof}
The two graphical equalities of Lemma~\ref{lem:TetraSym} realize the
permutations $(2\,3\,4)$ and $(3\,2\,1)$, which generate $A_4$
(Remark~\ref{rk:A4}). Composing them preserves the evaluation, with every
edge and vertex decoration transported. The inverse local rules
\eqref{R1}--\eqref{R2} return each resulting decoration to the reference
pattern.
\end{proof}

Now let us reformulate the TPE in the spherical case using these new notations.

\begin{theorem}[Spherical TPE] \label{thm:sTPE}
Assume that the pivotal fusion category $\mC$ of Theorem~\ref{thm:TPE} is spherical. With the same objects, morphisms and dual bases for the evaluation pairing, the following identity holds:
$$
\sum_{\beta_0 \in B_0}
\raisebox{-2.1cm}{
\begin{tikzpicture}[scale=2]
\draw[tpe half arrow] (2,0) -- (1,0) node [below] {\tiny ${^*\hspace{-.07cm}X_6}$};
\draw (1,0) -- (0,0) node [below] {$\alpha_3$};
\draw[tpe half arrow] (2,0) node [below] {$\sigma^{-2}(\alpha_1)$} -- (1.5,.866) node [right] {\tiny ${X_4}$};
\draw (1.5,.866) -- (1,1.732);
\draw[tpe half arrow] (1,1.732)	node [above] {$\alpha_2$} -- (.5,.866) node [left] {\tiny ${X_5}$};
\draw (.5,.866) -- (0,0);
\draw[tpe half arrow] (1,0.577) -- (.5,.2885) node [above] {\tiny ${X_3}$};
\draw (.5,.2885) -- (0,0);
\draw[tpe half arrow] (1,0.577) node [below] {$\beta_0$} -- (1,1.1545) node [right] {\tiny ${X_2}$};
\draw (1,1.1545) -- (1,1.732);
\draw[tpe half arrow] (1,0.577) -- (1.5,.2885) node [above] {\tiny ${X_1}$};
\draw (1.5,.2885) -- (2,0);
\end{tikzpicture}}
\raisebox{-2.1cm}{
\begin{tikzpicture}[scale=2]
\draw[tpe half arrow] (2,0) -- (1,0) node [below] {\tiny ${^*\hspace{-.07cm}X_9}$};
\draw (1,0) -- (0,0) node [below] {$\alpha_4$};
\draw[tpe half arrow] (2,0) node [below] {$\sigma^{-2}(\alpha_6)$} -- (1.5,.866) node [right] {\tiny ${X_8}$};
\draw (1.5,.866) -- (1,1.732);
\draw[tpe half arrow] (1,1.732)	node [above] {$\alpha_5$} -- (.5,.866) node [left] {\tiny ${X_7}$};
\draw (.5,.866) -- (0,0);
\draw (1,0.577) -- (.5,.2885) node [above] {\tiny ${X_1}$};
\draw[tpe half arrow] (0,0) -- (.5,.2885);
\draw (1,0.577) node [below] {$\beta_0'$} -- (1,1.1545) node [right] {\tiny ${X_2}$};
\draw[tpe half arrow] (1,1.732) -- (1,1.1545);
\draw (1,0.577) -- (1.5,.2885) node [above] {\tiny ${X_3}$};
\draw[tpe half arrow] (2,0) -- (1.5,.2885);
\end{tikzpicture}}$$
$$
 =  \sum_{X \in \mathcal{O}} \sum_{\substack{\beta_i \in B_i, \\ i=1,2,3}} \dim(X)
 \raisebox{-2cm}{
\begin{tikzpicture}[scale=2]
\draw[tpe half arrow] (2,0) -- (1,0) node [below] {\tiny ${X_3}$};
\draw (1,0) -- (0,0);
\draw (.25,.25) node {\small  $\alpha_3$};
\draw (2,0) -- (1.5,.866) node [right] {\tiny ${X_9^{**}}$};
\draw (1.75,.25) node {\small $\alpha_6$};
\draw[tpe half arrow] (1,1.732) -- (1.5,.866);
\draw (1,1.732)	node [above] {$\rho^2(\beta_1)$} -- (.5,.866) node [left] {\tiny ${X_6}$};
\draw[tpe half arrow] (0,0) -- (.5,.866);
\draw (1,0.577) -- (.5,.2885) node [above] {\tiny ${X_5^*}$};
\draw[tpe half arrow] (0,0) -- (.5,.2885);
\draw[tpe half arrow] (1,0.577) node [below=.1cm] {\small $\rho(\beta'_3)$} -- (1,1.1545) node [right=.015cm] {\tiny ${X^{*}}$};
\draw(1,1.732) -- (1,1.1545);
\draw (1,0.577) -- (1.5,.2885) node [above] {\tiny ${X_8}$};
\draw[tpe half arrow] (2,0) -- (1.5,.2885);
\end{tikzpicture}}
\raisebox{-2cm}{
\begin{tikzpicture}[scale=2]
\draw[tpe half arrow] (2,0) -- (1,0) node [below] {\tiny ${X_1}$};
\draw (1,0) -- (0,0);
\draw (.25,.25) node {\small  $\alpha_1$};
\draw[tpe half arrow] (2,0) -- (1.5,.866) node [right] {\tiny ${X_7^*}$};
\draw (1.75,.25) node {\small $\alpha_4$};
\draw (1,1.732) -- (1.5,.866);
\draw[tpe half arrow] (1,1.732)	node [above] {$\rho^2(\beta_2)$} -- (.5,.866) node [left] {\tiny ${^*\hspace{-.07cm}X_4}$};
\draw (.5,.866) -- (0,0);
\draw (1,0.577) -- (.5,.2885) node [above] {\tiny ${X_6^{***}}$};
\draw[tpe half arrow] (0,0) -- (.5,.2885);
\draw[tpe half arrow] (1,0.577) node [below=.1cm] {\small $\rho(\beta'_1)$} -- (1,1.1545) node [right] {\tiny ${X^*}$};
\draw (1,1.732) -- (1,1.1545);
\draw (1,0.577) -- (1.5,.2885) node [above] {\tiny ${X_9}$};
\draw[tpe half arrow] (2,0) -- (1.5,.2885);
\end{tikzpicture}}
\raisebox{-2cm}{
\begin{tikzpicture}[scale=2]
\draw[tpe half arrow] (2,0) -- (1,0) node [below] {\tiny ${X_2}$};
\draw (1,0) -- (0,0);
\draw (.25,.25) node {\small  $\alpha_2$};
\draw[tpe half arrow] (2,0) -- (1.5,.866) node [right] {\tiny ${X_8^*}$};
\draw (1.75,.25) node {\small $\alpha_5$};
\draw (1,1.732) -- (1.5,.866);
\draw[tpe half arrow] (1,1.732)	node [above] {$\rho^2(\beta_3)$} -- (.5,.866) node [left] {\tiny ${^*\hspace{-.07cm}X_5}$};
\draw (.5,.866) -- (0,0);
\draw (1,0.577) -- (.5,.2885) node [above] {\tiny ${X_4^{*}}$};
\draw[tpe half arrow] (0,0) -- (.5,.2885);
\draw[tpe half arrow] (1,0.577) node [below=.1cm] {\small $\rho(\beta'_2)$} -- (1,1.1545) node [right] {\tiny ${X^*}$};
\draw (1,1.732) -- (1,1.1545);
\draw (1,0.577) -- (1.5,.2885) node [above] {\tiny ${X_7}$};
\draw[tpe half arrow] (2,0) -- (1.5,.2885);
\end{tikzpicture}}
$$
\end{theorem}

\begin{proof}
For the first left factor write $(u,v,w)=(\rho^{-2}(\alpha_2),\rho(\alpha_3),\rho^{-1}(\alpha_1))$.
Rotate the reference tetrahedron $T(u,v,w,\beta_0)$, fixing its left vertex,
then use \eqref{R1}--\eqref{R2} at the marked corners:
\newcommand{\cutedge}[4]{\path (#1)--coordinate[midway] (mid) (#2);
 \draw[tpe half arrow] (#1)--(mid);\draw (mid)--(#2);\node[#4,inner sep=3pt,font=\scriptsize] at (mid) {$#3$};}
\[
\vcenter{\hbox{\begin{tikzpicture}[x=1.62cm,y=1.62cm,font=\small]
\coordinate (L) at (0,0); \coordinate (R) at (2,0); \coordinate (T) at (1,1.732); \coordinate (O) at (1,.577);
\cutedge{R}{L}{X_3}{below}
\cutedge{R}{T}{X_1}{right}
\cutedge{T}{L}{X_6^*}{left}
\cutedge{O}{L}{X_5}{above}
\path (O)--coordinate[midway] (mid) (T);\draw[tpe half arrow] (O)--(mid);\draw (mid)--(T);\node[right,inner sep=2pt,font=\scriptsize] at (1,1.01) {${}^*X_4$};
\cutedge{O}{R}{{}^*X_2}{above}
\node[inner sep=.8pt] at (0,-0.14) {$ v $};
\node[inner sep=.8pt] at (1,1.85) {$ w $};
\node[inner sep=.8pt] at (1,0.415) {$ u $};
\node[inner sep=.8pt] at (2,-0.14) {$ \beta_0 $};
\end{tikzpicture}}}\ \overset{\mathrm{rot.}}{=}\ \vcenter{\hbox{\begin{tikzpicture}[x=1.62cm,y=1.62cm,font=\small]
\coordinate (L) at (0,0); \coordinate (R) at (2,0); \coordinate (T) at (1,1.732); \coordinate (O) at (1,.577);
\cutedge{O}{L}{X_3}{above}
\cutedge{O}{R}{X_1}{above}
\cutedge{R}{L}{X_6^*}{below}
\cutedge{T}{L}{X_5}{left}
\cutedge{T}{R}{{}^*X_4}{right}
\path (T)--coordinate[midway] (mid) (O);\draw[tpe half arrow] (T)--(mid);\draw (mid)--(O);\node[right,inner sep=2pt,font=\scriptsize] at (1,1.01) {${}^*X_2$};
\node[inner sep=.8pt] at (0.44,0.109) {$ v $};
\node[inner sep=.8pt] at (0.85,1.26) {$ u $};
\node[inner sep=.8pt] at (1,0.415) {$ \beta_0 $};
\node[inner sep=.8pt] at (1.56,0.109) {$ w $};
\end{tikzpicture}}}\ \overset{\eqref{R1},\eqref{R2}}{=}\ \vcenter{\hbox{\begin{tikzpicture}[x=1.62cm,y=1.62cm,font=\small]
\coordinate (L) at (0,0); \coordinate (R) at (2,0); \coordinate (T) at (1,1.732); \coordinate (O) at (1,.577);
\cutedge{O}{L}{X_3}{above}
\cutedge{O}{R}{X_1}{above}
\cutedge{R}{L}{{}^*X_6}{below}
\cutedge{T}{L}{X_5}{left}
\cutedge{R}{T}{X_4}{right}
\path (O)--coordinate[midway] (mid) (T);\draw[tpe half arrow] (O)--(mid);\draw (mid)--(T);\node[right,inner sep=2pt,font=\scriptsize] at (1,1.01) {$X_2$};
\node[inner sep=.8pt] at (0,-0.14) {$ \alpha_3 $};
\node[inner sep=.8pt] at (1,1.85) {$ \alpha_2 $};
\node[inner sep=.8pt] at (1,0.415) {$ \beta_0 $};
\node[inner sep=.8pt] at (2,-0.14) {$ \sigma^{-2}(\alpha_1) $};
\end{tikzpicture}}}.
\]
In the last move $v\mapsto\rho^{-1}(v)$, $u\mapsto\rho^2(u)$ and
$w\mapsto\sigma^{-2}(\rho(w))$: the base edge reverses at the left corner
and reverses back at the right.

For the first right factor put $t=\rho^2(\beta_1)$, $h=\rho(\beta'_3)$ and
$\widehat h=\rho(\sigma(\beta'_3))$. In its reference graph apply $\rho^{-1},\rho,\rho^3$
at the left, right and central vertices, respectively. Only the middle edge changes
in the following equality; the common outer edge labels are suppressed:
\[
\vcenter{\hbox{\begin{tikzpicture}[x=1.62cm,y=1.62cm,font=\small]
\coordinate (L) at (0,0); \coordinate (R) at (2,0); \coordinate (T) at (1,1.732); \coordinate (O) at (1,.577);
\cutedge{R}{L}{}{below}
\cutedge{L}{T}{}{left}
\cutedge{T}{R}{}{right}
\cutedge{L}{O}{}{above}
\cutedge{R}{O}{}{above}
\draw[line width=.75pt,tpe half arrow] (T)--(1,1.1545);\draw[line width=.75pt] (1,1.1545)--(O);\node[right,inner sep=2pt,font=\small] at (1,.94) {$X^{**}$};
\node[inner sep=.8pt] at (0.40,0.39) {$ \alpha_3 $};
\node[inner sep=.8pt] at (1.60,0.39) {$ \alpha_6 $};
\node[inner sep=.8pt] at (1,1.85) {$ t $};
\node[inner sep=.8pt] at (1,0.4) {$ \widehat h $};
\end{tikzpicture}}}\qquad=\qquad\vcenter{\hbox{\begin{tikzpicture}[x=1.62cm,y=1.62cm,font=\small]
\coordinate (L) at (0,0); \coordinate (R) at (2,0); \coordinate (T) at (1,1.732); \coordinate (O) at (1,.577);
\cutedge{R}{L}{}{below}
\cutedge{L}{T}{}{left}
\cutedge{T}{R}{}{right}
\cutedge{L}{O}{}{above}
\cutedge{R}{O}{}{above}
\draw[line width=.75pt,tpe half arrow] (O)--(1,1.1545);\draw[line width=.75pt] (1,1.1545)--(T);\node[right,inner sep=2pt,font=\small] at (1,1.010) {$X^*$};
\node[inner sep=.8pt] at (0.40,0.39) {$ \alpha_3 $};
\node[inner sep=.8pt] at (1.60,0.39) {$ \alpha_6 $};
\node[inner sep=.8pt] at (1,1.85) {$ t $};
\node[inner sep=.8pt] at (1,0.4) {$ h $};
\end{tikzpicture}}},
\]
since $\widehat h=(\id_{X_8^*}\otimes a_{X^*}\otimes\id_{X_5^{**}})\circ h$ and
$\ev_{X^*,-}=\ev_{X^{**},+}\circ(a_{X^*}\otimes\id_{X^{**}})$.
Apply the same corner rules to the remaining factors of Theorem~\ref{thm:TPE},
using the paired bases of Lemma~\ref{lem:bili}.
\end{proof}

\begin{remark} \label{rk:RepTP}
Here are three representations of the triangular prism graph,
$$
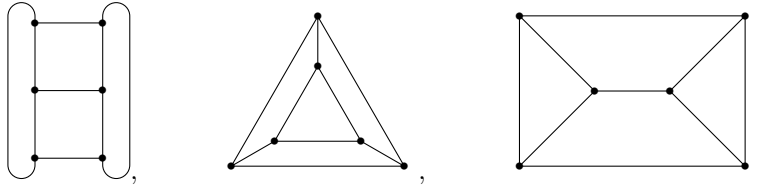
\begin{tikzpicture}[scale=.9]
\draw (1,0) node {\tiny $\bullet$}--(2,0)node {\tiny $\bullet$};
\draw (1,1) node {\tiny $\bullet$}--(2,1)node {\tiny $\bullet$};
\draw (1,2) node {\tiny $\bullet$}--(2,2)node {\tiny $\bullet$};
\draw (1,0) -- (1,2) --++ (0,.1) arc (0:180:.2) --++ (0,-2.2) arc (-180:0:.2) --++ (0,.1);
\draw (2,0) -- (2,2) --++ (0,.1) arc (180:0:.2) --++ (0,-2.2) arc (0:-180:.2) --++ (0,.1);
\end{tikzpicture}
, \hspace{1cm}
\begin{tikzpicture}[scale=1.15]
\draw (0,0) node {\tiny $\bullet$} --(2,0) node {\tiny $\bullet$} --(1,1.732) node {\tiny $\bullet$} --(0,0);
\draw (0,0)--(.5,.2885) node {\tiny $\bullet$} ;
\draw (1,1.732)--(1,1.1545) node {\tiny $\bullet$} ;
\draw (2,0)--(1.5,.2885)--(1,1.1545)--(.5,.2885)--(1.5,.2885) node {\tiny $\bullet$} ;
\end{tikzpicture}
, \hspace{1cm}
\begin{tikzpicture}[scale=1]
\draw (0,2)node {\tiny $\bullet$}--(3,2) node {\tiny $\bullet$};
\draw (1,1) node {\tiny $\bullet$}--(2,1) node {\tiny $\bullet$};
\draw (0,0) node {\tiny $\bullet$}--(3,0) node {\tiny $\bullet$};
\draw (0,0) -- (0,2) -- (1,1) -- (0,0);
\draw (3,0)--(3,2)--(2,1)--(3,0);
\end{tikzpicture},
$$
the first is the one we used in Definition \ref{def:TPCat}, the second is the usual one (already mentioned in \S \ref{sec:intro}), and the last is the one we will use in Definition \ref{def:TPPiv}.
\end{remark}

Following Remark \ref{rk:SymTP}, we can realize the orientation-preserving symmetry group $S_3$ of the regular triangular prism, as for the regular tetrahedron in Remark \ref{rk:A4ction}. The monoidal triangular prism in Definition \ref{def:TPCat} can be represented pictorially (as for Definition \ref{def:TetraPiv}) using Definition \ref{def:TPPiv}.

\begin{definition} \label{def:TPPiv}
Following Definition \ref{def:TPCat},
$$
\raisebox{-1.5cm}{
\begin{tikzpicture}[scale=1.4]
\draw (0,2)--(3,2);
\draw (1,1)--(2,1);
\draw (0,0)--(3,0);
\draw (0,0) -- (0,2) -- (1,1) -- (0,0);
\draw (3,0)--(3,2)--(2,1)--(3,0);
\draw[tpe half reverse](0,2) -- (1.5,2) node [above] {\tiny $X_1$};
\draw[tpe half reverse] (1,1) -- (1.5,1) node [above] {\tiny $X_2$};
\draw[tpe half reverse] (0,0) -- (1.5,0) node [above] {\tiny $X_3$};
\draw[tpe half arrow] (0,2) --++ (.5,-.5) node [right] {\tiny $X_4$};
\draw[tpe half arrow] (1,1) --++ (-.5,-.5) node [right] {\tiny $X_5$};
\draw[tpe half arrow] (0,0) -- (0,1) node [left] {\tiny $X_6$};
\draw[tpe half reverse] (3,2) --++ (-.5,-.5) node [left] {\tiny $X_7$};
\draw[tpe half reverse] (2,1) --++ (.5,-.5) node [left] {\tiny $X_8$};
\draw[tpe half reverse] (3,0) -- (3,1) node [right] {\tiny $X_9$};
\node at (0+.13,2-.3) {\tiny $\alpha_1$};
\node at (1-.2,1) {\tiny $\alpha_2$};
\node at (0+.13,0+.3) {\tiny $\alpha_3$};
\node at (3-.13,2-.3) {\tiny $\alpha_4$};
\node at (2+.2,1) {\tiny $\alpha_5$};
\node at (3-.13,0+.3) {\tiny $\alpha_6$};
\end{tikzpicture}}
:=TP(\alpha_1, \alpha_2, \alpha_3, \alpha_4, \alpha_5, \alpha_6).
$$
\end{definition}

\begin{remark}
One may similarly compare tetrahedral decompositions of other planar trivalent graphs, for example the cubical graph. As in Thurston's prism picture (Figure~\ref{fig:thurston-prism}), each decomposition gives a contraction of tetrahedral coefficients. When the graph has a regular geometric realization, its orientation-preserving symmetries act by transporting the edge and vertex decorations. These symmetries and the choice of cuts may expose useful local equations for categorification problems. In a spherical monoidal category the resulting identities follow from associativity and duality; \S\ref{sec:TPEvsPE} identifies the triangular prism equation itself with the pentagon relation after a change of coordinates. A reconstruction from scalar solutions must in addition impose associator invertibility, unit normalization, and compatible duality and spherical data. The smallest non-planar trivalent graph gives the braided variant discussed in Remark~\ref{rk:nonplanar}.
\end{remark}

\begin{remark} \label{rk:StrPiv}
By \cite[Definition~2.1 and Theorem~2.2, p.~70]{NgSch}, every spherical monoidal category is pivotally equivalent to a strictly pivotal spherical category. A pivotal equivalence transports both pivotal traces under its identification of endomorphisms of the unit, and therefore preserves sphericality \cite[Lemma~6.1, p.~78]{NgSch}. In that realization $a_X=\id_X$ for every object $X$, the duality functor has identity monoidal constraints, and $(-)^{**}=\mathrm{Id}$ on both objects and morphisms. Hence $\sigma=\id$. In the fusion setting we take the equivalence to be $\field$-linear, so the categorical dimensions agree as scalars in $\field$. A tensor equivalence preserves fusion multiplicities, and a $\field$-linear pivotal equivalence preserves the numerical Frobenius--Schur indicators \cite[Corollary~4.4, p.~74]{NgSch}. We may therefore use this realization for the following simplified TPE and the calculations that follow without changing the categorification problem.

The realization need not be skeletal: self-duality means $X\simeq X^*$, not equality with the chosen dual. For a simple self-dual $X$ and an isomorphism $\kappa:X\to X^*$, \eqref{eq:indicator-transpose} specializes to $\kappa^*=\nu_2(X)\kappa$. Transporting the duality along $\kappa$ to make $X$ its own chosen dual gives pivotal component $\kappa^{-1}\kappa^*=\nu_2(X)\id_X$, as in Proposition~\ref{prop:pivoFS}.
\end{remark}

In the strictly pivotal realization of Remark~\ref{rk:StrPiv}, reverse the edges
$X_i$ for $i=1,2,3,7,8,9$ of Theorem~\ref{thm:sTPE} and use the dual objects
as labels (Figure~\ref{fig:TPEconfig}). The corner changes now involve only
$\rho$ and $\rho^{-1}$. With a general pivotal structure the same changes
retain the pivotal morphisms in \eqref{R2}.
\begin{figure}[ht]
$$
\begin{tikzpicture}[scale=1.4]
\draw (0,2)--(3,2);
\draw (1,1)--(2,1);
\draw (0,0)--(3,0);
\draw (0,0) -- (0,2) -- (1,1) -- (0,0);
\draw (3,0)--(3,2)--(2,1)--(3,0);
\draw[tpe half arrow](0,2) -- (1.5,2) node [above] {\tiny $X_1$};
\draw[tpe half arrow] (1,1) -- (1.5,1) node [above] {\tiny $X_2$};
\draw[tpe half arrow] (0,0) -- (1.5,0) node [above] {\tiny $X_3$};
\draw[tpe half arrow] (0,2) --++ (.5,-.5) node [right] {\tiny $X_4$};
\draw[tpe half arrow] (1,1) --++ (-.5,-.5) node [right] {\tiny $X_5$};
\draw[tpe half arrow] (0,0) -- (0,1) node [left] {\tiny $X_6$};
\draw[tpe half arrow] (3,2) --++ (-.5,-.5) node [left] {\tiny $X_7$};
\draw[tpe half arrow] (2,1) --++ (.5,-.5) node [left] {\tiny $X_8$};
\draw[tpe half arrow] (3,0) -- (3,1) node [right] {\tiny $X_9$};
\node at (0+.13,2-.3) {\tiny $\alpha_1$};
\node at (1-.2,1) {\tiny $\alpha_2$};
\node at (0+.13,0+.3) {\tiny $\alpha_3$};
\node at (3-.13,2-.3) {\tiny $\alpha_4$};
\node at (2+.2,1) {\tiny $\alpha_5$};
\node at (3-.13,0+.3) {\tiny $\alpha_6$};
\end{tikzpicture}
$$
\caption{Triangular prism configuration}
\label{fig:TPEconfig}
\end{figure}
\begin{theorem}[Simplified TPE] \label{thm:sTPE2}
Let $\mC$ be a spherical fusion category in the strictly pivotal realization of Remark~\ref{rk:StrPiv}; thus $a_X=\id_X$ for every object and $(-)^{**}=\mathrm{Id}$ on objects and morphisms. Let $X_1, \dots, X_9$ be objects in $\mC$. Consider morphisms $\alpha_1 \in \hc(\one,X_4 \otimes X_1 \otimes X_6^*)$, $\alpha_2 \in \hc(\one,X_5 \otimes X_2 \otimes X_4^*)$, $\alpha_3 \in \hc(\one,X_6 \otimes X_3 \otimes X_5^*)$, $\alpha_4 \in \hc(\one,X_9^* \otimes X_1^* \otimes X_7)$, $\alpha_5 \in \hc(\one,X_7^* \otimes X_2^* \otimes X_8)$ and $\alpha_6 \in \hc(\one,X_8^* \otimes X_3^* \otimes X_9)$. Then the following TPE holds:
$$
\sum_{\beta_0 \in B_0} \hspace{-.5cm}
\raisebox{-1.4cm}{
\begin{tikzpicture}[scale=1.39]
\draw (2,0) -- (1,0) node [below] {\tiny ${X_6}$};
\draw[tpe half arrow] (0,0) node [below] {\tiny $\alpha_3$} -- (1,0);
\draw[tpe half arrow] (2,0) node [below] {\tiny $\alpha_1$} -- (1.5,.866) node [right] {\tiny ${X_4}$};
\draw (1.5,.866) -- (1,1.732);
\draw[tpe half arrow] (1,1.732)	node [above] {\tiny $\alpha_2$} -- (.5,.866) node [left] {\tiny ${X_5}$};
\draw (.5,.866) -- (0,0);
\draw (1,0.577) -- (.5,.2885) node [above] {\tiny ${X_3}$};
\draw[tpe half arrow] (0,0) -- (.5,.2885);
\draw (1,0.577) node [below=.015cm] {\tiny $\beta'_0$} -- (1,1.1545);\node[anchor=west,inner sep=2pt] at (1,.99) {\tiny ${X_2}$};
\draw[tpe half arrow] (1,1.732) -- (1,1.1545);
\draw (1,0.577) -- (1.5,.2885) node [above] {\tiny ${X_1}$};
\draw[tpe half arrow] (2,0) -- (1.5,.2885);
\end{tikzpicture}} \hspace{-.2cm}
\raisebox{-1.4cm}{
\begin{tikzpicture}[scale=1.39]
\draw[tpe half arrow] (2,0) -- (1,0) node [below] {\tiny ${X_9}$};
\draw (0,0) node [below] {\tiny $\alpha_4$} -- (1,0);
\draw (2,0) node [below] {\tiny $\alpha_6$} -- (1.5,.866) node [right] {\tiny ${X_8}$};
\draw[tpe half arrow] (1,1.732) -- (1.5,.866);
\draw (1,1.732)	node [above] {\tiny $\alpha_5$} -- (.5,.866) node [left] {\tiny ${X_7}$};
\draw[tpe half arrow] (0,0) -- (.5,.866);
\draw[tpe half arrow] (1,0.577) -- (.5,.2885) node [above] {\tiny ${X_1}$};
\draw (0,0) -- (.5,.2885);
\draw[tpe half arrow] (1,0.577) node [below=.015cm] {\tiny $\beta_0$} -- (1,1.1545);\node[anchor=west,inner sep=2pt] at (1,.99) {\tiny ${X_2}$};
\draw (1,1.732) -- (1,1.1545);
\draw[tpe half arrow] (1,0.577) -- (1.5,.2885) node [above] {\tiny ${X_3}$};
\draw (2,0) -- (1.5,.2885);
\end{tikzpicture}}
\hspace{-.4cm} = \hspace{-.1cm} \sum_{\substack{X \in \mathcal{O}, \\ \beta_i \in B_i, \\ i=1,2,3}} \hspace{-.1cm} d_X \hspace{-.3cm}
\raisebox{-1.4cm}{
\begin{tikzpicture}[scale=1.39]
\draw (2,0) -- (1,0) node [below] {\tiny ${X_1}$};
\draw[tpe half arrow] (0,0) -- (1,0);
\draw (.40,.40) node {\tiny $\alpha_1$};
\draw[tpe half arrow] (2,0) -- (1.5,.866) node [right] {\tiny ${X_7}$};
\draw (1.60,.40) node {\tiny $\alpha_4$};
\draw (1,1.732) -- (1.5,.866);
\draw (1,1.732)	node [above] {\tiny $\beta_1$} -- (.5,.866) node [left] {\tiny ${X_4}$};
\draw[tpe half arrow] (0,0) -- (.5,.866);
\draw[tpe half arrow] (1,0.577) -- (.5,.2885);\node[anchor=south,inner sep=0pt] at (.63,.49) {\tiny ${X_6}$};
\draw (0,0) -- (.5,.2885);
\draw (1,0.577) node [below=.015cm] {\tiny $\beta'_3$} -- (1,1.1545);\node[anchor=west,inner sep=2pt] at (1,.99) {\tiny ${X}$};
\draw[tpe half arrow] (1,1.732) -- (1,1.1545);
\draw[tpe half arrow] (1,0.577) -- (1.5,.2885);\node[anchor=south,inner sep=0pt] at (1.37,.49) {\tiny ${X_9}$};
\draw (2,0) -- (1.5,.2885);
\end{tikzpicture}}
\raisebox{-1.4cm}{
\begin{tikzpicture}[scale=1.39]
\draw (2,0) -- (1,0) node [below] {\tiny ${X_2}$};
\draw[tpe half arrow] (0,0) -- (1,0);
\draw (.40,.40) node {\tiny $\alpha_2$};
\draw[tpe half arrow] (2,0) -- (1.5,.866) node [right] {\tiny ${X_8}$};
\draw (1.60,.40) node {\tiny $\alpha_5$};
\draw (1,1.732) -- (1.5,.866);
\draw (1,1.732)	node [above] {\tiny $\beta_2$} -- (.5,.866) node [left] {\tiny $X_5$};
\draw[tpe half arrow] (0,0) -- (.5,.866);
\draw[tpe half arrow] (1,0.577) -- (.5,.2885);\node[anchor=south,inner sep=0pt] at (.63,.49) {\tiny ${X_4}$};
\draw (0,0) -- (.5,.2885);
\draw (1,0.577) node [below=.015cm] {\tiny $\beta'_1$} -- (1,1.1545);\node[anchor=west,inner sep=2pt] at (1,.99) {\tiny ${X}$};
\draw[tpe half arrow] (1,1.732) -- (1,1.1545);
\draw[tpe half arrow] (1,0.577) -- (1.5,.2885);\node[anchor=south,inner sep=0pt] at (1.37,.49) {\tiny ${X_7}$};
\draw (2,0) -- (1.5,.2885);
\end{tikzpicture}}
\raisebox{-1.4cm}{
\begin{tikzpicture}[scale=1.39]
\draw (2,0) -- (1,0) node [below] {\tiny ${X_3}$};
\draw[tpe half arrow] (0,0) -- (1,0);
\draw (.40,.40) node {\tiny $\alpha_3$};
\draw[tpe half arrow] (2,0) -- (1.5,.866) node [right] {\tiny ${X_9}$};
\draw (1.60,.40) node {\tiny $\alpha_6$};
\draw (1,1.732) -- (1.5,.866);
\draw (1,1.732)	node [above] {\tiny $\beta_3$} -- (.5,.866) node [left] {\tiny ${X_6}$};
\draw[tpe half arrow] (0,0) -- (.5,.866);
\draw[tpe half arrow] (1,0.577) -- (.5,.2885);\node[anchor=south,inner sep=0pt] at (.63,.49) {\tiny ${X_5}$};
\draw (0,0) -- (.5,.2885);
\draw (1,0.577) node [below=.015cm] {\tiny $\beta'_2$} -- (1,1.1545);\node[anchor=west,inner sep=2pt] at (1,.99) {\tiny ${X}$};
\draw[tpe half arrow] (1,1.732) -- (1,1.1545);
\draw[tpe half arrow] (1,0.577) -- (1.5,.2885);\node[anchor=south,inner sep=0pt] at (1.37,.49) {\tiny ${X_8}$};
\draw (2,0) -- (1.5,.2885);
\end{tikzpicture}}
$$
where $d_X = \dim(X)$; $\mathcal{O}$ is the set of simple subobjects of all three objects $X_4 \otimes X_7$, $X_5 \otimes X_8$ and $X_6 \otimes X_9$ (up to isomorphism); $B_i$ is a basis of $\hc(\one,Z_i)$ with $Z_0=X_3 \otimes X_2 \otimes X_1$, $Z_1=X^*_4 \otimes X \otimes X^*_7$, $Z_2 = X^*_5 \otimes X \otimes X^*_8$ and $Z_3 = X^*_6 \otimes X \otimes X^*_9$. Let $B'_i$ be the basis of $\hc(\one,Z_i^*)$ dual to $B_i$ under the pairing of Lemma~\ref{lem:bili}. For $\beta_i\in B_i$, let $\beta'_i$ denote the corresponding element of $B'_i$. For $i=1,2,3$, write $Z_i(X),B_i(X),B'_i(X)$ when making their dependence on the summation variable explicit; $B_0$ is fixed. The set $\mathcal O$ contains one representative of each common simple constituent.
\end{theorem}
\begin{proof}
Substitute $(X_1^*,X_2^*,X_3^*,X_4,X_5,X_6,X_7^*,X_8^*,X_9^*)$
in Theorem~\ref{thm:sTPE}. Strict pivotality gives $\sigma=\rho^3=\id$.
To see how the dual bases change, let $v:\one\to A\otimes B\otimes C$
and $u:\one\to C^*\otimes B^*\otimes A^*$. The two drawings below
evaluate the same spherical diagram, first with $C$ outermost and then
with $A$ outermost:
\newcommand{\pairedcaps}[5]{\vcenter{\hbox{\begin{tikzpicture}[x=.77cm,y=.73cm,font=\scriptsize]
 \foreach \l/\r/\h/\obj in {1.3/2.1/.55/{#3},.65/2.75/1.18/{#4},0/3.4/1.81/{#5}}{
   \draw[tpe reverse] (\l,.8)--(\l,0)
    .. controls (\l,-\h) and (\r,-\h) .. (\r,0)--(\r,.8);
   \node[right,inner sep=2pt] at (\r,.27) {$\obj$};}
 \node[draw,fill=white,minimum width=1.30cm,inner sep=2pt] at (.65,.8) {$#1$};
 \node[draw,fill=white,minimum width=1.30cm,inner sep=2pt] at (2.75,.8) {$#2$};
 \end{tikzpicture}}}}
\[
 \pairedcaps{u}{v}{A}{B}{C}
 \;=\;
 \pairedcaps{\rho^{-1}(u)}{\rho(v)}{B}{C}{A}.
\]
The left scalar is $\ev_{A\otimes B\otimes C}\circ(u\otimes v)$, the pairing
of Lemma~\ref{lem:bili}. Thus rotating a basis by $\rho$ rotates its dual
by $\rho^{-1}$. At the middle cut, \eqref{eq:dual-transposition} allows
the two bases to be interchanged. In terms of the bases of the present theorem,
those of Theorem~\ref{thm:sTPE} are therefore
\[
\begin{array}{c|cccc}
 j&0&1&2&3\\ \hline
 \text{basis element}&\beta'_0&\rho(\beta_3)&\rho(\beta_1)&\rho(\beta_2)\\
 \text{dual element}&\beta_0&\rho^{-1}(\beta'_3)&\rho^{-1}(\beta'_1)&\rho^{-1}(\beta'_2).
\end{array}
\]
For instance, the second factor on the right of Theorem~\ref{thm:sTPE}
has top label $\rho^2(\rho(\beta_1))=\beta_1$ and central label
$\rho(\rho^{-1}(\beta'_3))=\beta'_3$. Its upward $X^*$ edge is the
downward $X$ edge of the first factor here, since $\ev_{X,-}=\ev_{X^*,+}$.
The same substitutions give the two left factors in order and the right
factors in order $(2,3,1)$, with the same categorical dimension $d_X$.
\end{proof}

\begin{lemma}[Rotation eigenvalue] \label{lem:rot}
In the strictly pivotal realization of Remark~\ref{rk:StrPiv}, suppose $\hc(\one,X^{\otimes 3})$ is one-dimensional with generator $\theta$. Then
$$
\raisebox{-.6cm}{
\begin{tikzpicture}
\draw[tpe half arrow] (0,0)  node [above] {$\theta$} -- (0.433,0.25) node [below=1.5pt] {\tiny $X$};
\draw (0.433,0.25) -- (0.866,0.5);
\draw[tpe half arrow] (0,0) -- (-0.433,0.25) node [below=1.5pt] {\tiny $X$};
\draw (-0.433,0.25)--(-0.866,0.5);
\draw[tpe half arrow] (0,0)--(0,-.5) node [right] {\tiny $X$};
\draw (0,-.5) -- (0,-1);
\end{tikzpicture}}
= \omega_X
\raisebox{-.6cm}{
\begin{tikzpicture}
\node[anchor=north east,inner sep=0pt] at (-.08,-.10) {$\theta$};
\draw[tpe half arrow] (0,0) -- (0.433,0.25) node [above=1.5pt] {\tiny $X$};
\draw (0.433,0.25) -- (0.866,0.5);
\draw[tpe half arrow] (0,0) -- (-0.433,0.25) node [above=1.5pt] {\tiny $X$};
\draw (-0.433,0.25)--(-0.866,0.5);
\draw[tpe half arrow] (0,0)--(0,-.5) node [right] {\tiny $X$};
\draw (0,-.5) -- (0,-1);
\end{tikzpicture}},
$$
where $\omega_X := \nu_3(X)$, so that $\omega^3_X = 1$.
\end{lemma}
\begin{proof}
Since $\sigma=\id$, Remark~\ref{rk:FS3} gives $\rho(\theta)=\nu_3(X)\theta$. Apply \eqref{R1}, substitute this identity, and reverse the dual-labelled edge using $a_X=\id_X$ and Lemma~\ref{lem:biFS}:
\[
\raisebox{-.6cm}{
\begin{tikzpicture}
\draw[tpe half arrow] (0,0)  node [above] {$\theta$} -- (0.433,0.25) node [below=1.5pt] {\tiny $X$};
\draw (0.433,0.25) -- (0.866,0.5);
\draw[tpe half arrow] (0,0) -- (-0.433,0.25) node [below=1.5pt] {\tiny $X$};
\draw (-0.433,0.25)--(-0.866,0.5);
\draw[tpe half arrow] (0,0)--(0,-.5) node [right] {\tiny $X$};
\draw (0,-.5) -- (0,-1);
\end{tikzpicture}}
=
\raisebox{-.6cm}{
\begin{tikzpicture}
\node[anchor=north east,inner sep=0pt] at (-.08,-.10) {$\rho(\theta)$};
\draw[tpe half arrow] (0,0) -- (0.433,0.25) node [above=1.5pt] {\tiny $X$};
\draw (0.866,0.5) -- (0.433,0.25);
\draw (0,0) -- (-0.433,0.25) node [above=1.5pt] {\tiny $X^*$};
\draw[tpe half arrow] (-0.866,0.5)--(-0.433,0.25);
\draw[tpe half arrow] (0,0)--(0,-.5) node [right] {\tiny $X$};
\draw (0,-.5) -- (0,-1);
\end{tikzpicture}}
=  \nu_3(X)
\raisebox{-.6cm}{
\begin{tikzpicture}
\draw (-.35,-.2)  node {$\theta$};
\draw[tpe half arrow] (0,0) -- (0.433,0.25) node [above=1.5pt] {\tiny $X$};
\draw (0.866,0.5) -- (0.433,0.25);
\draw (0,0) -- (-0.433,0.25) node [above=1.5pt] {\tiny $X^*$};
\draw[tpe half arrow] (-0.866,0.5)--(-0.433,0.25);
\draw[tpe half arrow] (0,0)--(0,-.5) node [right] {\tiny $X$};
\draw (0,-.5) -- (0,-1);
\end{tikzpicture}}
= \nu_3(X)
\raisebox{-.6cm}{
\begin{tikzpicture}
\node[anchor=north east,inner sep=0pt] at (-.08,-.10) {$\theta$};
\draw[tpe half arrow] (0,0) -- (0.433,0.25) node [above=1.5pt] {\tiny $X$};
\draw (0.433,0.25) -- (0.866,0.5);
\draw[tpe half arrow] (0,0) -- (-0.433,0.25) node [above=1.5pt] {\tiny $X$};
\draw (-0.433,0.25)--(-0.866,0.5);
\draw[tpe half arrow] (0,0)--(0,-.5) node [right] {\tiny $X$};
\draw (0,-.5) -- (0,-1);
\end{tikzpicture}}.\qedhere
\]
\end{proof}

\begin{remark}[More simplifications] \label{rk:simplified}
Work in the realization of Remark~\ref{rk:StrPiv}. Suppose that every external edge label and every internal label occurring in the sum
is a self-dual simple object of second indicator $+1$.
Then any chosen isomorphism $\kappa_i:X_i\to X_i^*$ satisfies
$\kappa_i^*=\kappa_i$ by \eqref{eq:indicator-transpose}. Transport the cap and cup:
\[
 e_i=
\begin{tikzpicture}[baseline=(current bounding box.center),scale=.8]
\draw (0,0) node[above] {$X_i$} -- (0,-.8)
 .. controls (0,-1.45) and (1.5,-1.45) .. (1.5,-.8)
 --(1.5,0) node[above] {$X_i$};
\node[draw,fill=white,inner sep=2pt] at (0,-.45) {$\kappa_i$};
\node[left,font=\scriptsize] at (0,-.95) {$X_i^*$};
\end{tikzpicture}
=\ev_{X_i}\circ(\kappa_i\otimes\id_{X_i}),
\qquad
 c_i=
\begin{tikzpicture}[baseline=(current bounding box.center),scale=.8]
\draw (0,0) node[below] {$X_i$} -- (0,.8)
 .. controls (0,1.45) and (1.5,1.45) .. (1.5,.8)
 --(1.5,0) node[below] {$X_i$};
\node[draw,fill=white,inner sep=2pt] at (1.5,.45) {$\kappa_i^{-1}$};
\node[right,font=\scriptsize] at (1.5,1.05) {$X_i^*$};
\end{tikzpicture}
=(\id_{X_i}\otimes\kappa_i^{-1})\circ\coev_{X_i}.
\]
\emph{Omit arrows.} These maps satisfy the zigzag identities since $\kappa_i^{-1}\circ\kappa_i=\id_{X_i}$,
and $\kappa_i^*=\kappa_i$ identifies the two orientations of each cap or cup.
Using the same isomorphisms at the vertices and in the evaluation pairings,
Theorem~\ref{thm:sTPE2} reads
$$
\sum_{\beta_0 \in B_0} \hspace{-.5cm}
\raisebox{-1.4cm}{
\begin{tikzpicture}[scale=1.43]
\draw (2,0) -- (1,0) node [below] {\tiny ${X_6}$};
\draw (0,0) node [below] {\tiny $\alpha_3$} -- (1,0);
\draw (2,0) node [below] {\tiny $\alpha_1$} -- (1.5,.866) node [right] {\tiny ${X_4}$};
\draw (1.5,.866) -- (1,1.732);
\draw (1,1.732)	node [above] {\tiny $\alpha_2$} -- (.5,.866) node [left] {\tiny ${X_5}$};
\draw (.5,.866) -- (0,0);
\draw (1,0.577) -- (.5,.2885) node [above] {\tiny ${X_3}$};
\draw (0,0) -- (.5,.2885);
\draw (1,0.577) node [below=.015cm] {\tiny $\beta'_0$} -- (1,1.1545);\node[anchor=west,inner sep=2pt] at (1,.99) {\tiny ${X_2}$};
\draw (1,1.732) -- (1,1.1545);
\draw (1,0.577) -- (1.5,.2885) node [above] {\tiny ${X_1}$};
\draw (2,0) -- (1.5,.2885);
\end{tikzpicture}} \hspace{-.2cm}
\raisebox{-1.4cm}{
\begin{tikzpicture}[scale=1.43]
\draw (2,0) -- (1,0) node [below] {\tiny ${X_9}$};
\draw (0,0) node [below] {\tiny $\alpha_4$} -- (1,0);
\draw (2,0) node [below] {\tiny $\alpha_6$} -- (1.5,.866) node [right] {\tiny ${X_8}$};
\draw (1,1.732) -- (1.5,.866);
\draw (1,1.732)	node [above] {\tiny $\alpha_5$} -- (.5,.866) node [left] {\tiny ${X_7}$};
\draw (0,0) -- (.5,.866);
\draw (1,0.577) -- (.5,.2885) node [above] {\tiny ${X_1}$};
\draw (0,0) -- (.5,.2885);
\draw (1,0.577) node [below=.015cm] {\tiny $\beta_0$} -- (1,1.1545);\node[anchor=west,inner sep=2pt] at (1,.99) {\tiny ${X_2}$};
\draw (1,1.732) -- (1,1.1545);
\draw (1,0.577) -- (1.5,.2885) node [above] {\tiny ${X_3}$};
\draw (2,0) -- (1.5,.2885);
\end{tikzpicture}}
\hspace{-.4cm} = \hspace{-.1cm} \sum_{\substack{X \in \mathcal{O}, \\ \beta_i \in B_i, \\ i=1,2,3}} \hspace{-.1cm} d_X \hspace{-.3cm}
\raisebox{-1.4cm}{
\begin{tikzpicture}[scale=1.43]
\draw (2,0) -- (1,0) node [below] {\tiny ${X_1}$};
\draw (0,0) -- (1,0);
\draw (.40,.40) node {\tiny $\alpha_1$};
\draw (2,0) -- (1.5,.866) node [right] {\tiny ${X_7}$};
\draw (1.60,.40) node {\tiny $\alpha_4$};
\draw (1,1.732) -- (1.5,.866);
\draw (1,1.732)	node [above] {\tiny $\beta_1$} -- (.5,.866) node [left] {\tiny ${X_4}$};
\draw (0,0) -- (.5,.866);
\draw (1,0.577) -- (.5,.2885);\node[anchor=south,inner sep=0pt] at (.63,.49) {\tiny ${X_6}$};
\draw (0,0) -- (.5,.2885);
\draw (1,0.577) node [below=.015cm] {\tiny $\beta'_3$} -- (1,1.1545);\node[anchor=west,inner sep=2pt] at (1,.99) {\tiny ${X}$};
\draw (1,1.732) -- (1,1.1545);
\draw (1,0.577) -- (1.5,.2885);\node[anchor=south,inner sep=0pt] at (1.37,.49) {\tiny ${X_9}$};
\draw (2,0) -- (1.5,.2885);
\end{tikzpicture}}
\raisebox{-1.4cm}{
\begin{tikzpicture}[scale=1.43]
\draw (2,0) -- (1,0) node [below] {\tiny ${X_2}$};
\draw (0,0) -- (1,0);
\draw (.40,.40) node {\tiny $\alpha_2$};
\draw (2,0) -- (1.5,.866) node [right] {\tiny ${X_8}$};
\draw (1.60,.40) node {\tiny $\alpha_5$};
\draw (1,1.732) -- (1.5,.866);
\draw (1,1.732)	node [above] {\tiny $\beta_2$} -- (.5,.866) node [left] {\tiny $X_5$};
\draw (0,0) -- (.5,.866);
\draw (1,0.577) -- (.5,.2885);\node[anchor=south,inner sep=0pt] at (.63,.49) {\tiny ${X_4}$};
\draw (0,0) -- (.5,.2885);
\draw (1,0.577) node [below=.015cm] {\tiny $\beta'_1$} -- (1,1.1545);\node[anchor=west,inner sep=2pt] at (1,.99) {\tiny ${X}$};
\draw (1,1.732) -- (1,1.1545);
\draw (1,0.577) -- (1.5,.2885);\node[anchor=south,inner sep=0pt] at (1.37,.49) {\tiny ${X_7}$};
\draw (2,0) -- (1.5,.2885);
\end{tikzpicture}}
\raisebox{-1.4cm}{
\begin{tikzpicture}[scale=1.43]
\draw (2,0) -- (1,0) node [below] {\tiny ${X_3}$};
\draw (0,0) -- (1,0);
\draw (.40,.40) node {\tiny $\alpha_3$};
\draw (2,0) -- (1.5,.866) node [right] {\tiny ${X_9}$};
\draw (1.60,.40) node {\tiny $\alpha_6$};
\draw (1,1.732) -- (1.5,.866);
\draw (1,1.732)	node [above] {\tiny $\beta_3$} -- (.5,.866) node [left] {\tiny ${X_6}$};
\draw (0,0) -- (.5,.866);
\draw (1,0.577) -- (.5,.2885);\node[anchor=south,inner sep=0pt] at (.63,.49) {\tiny ${X_5}$};
\draw (0,0) -- (.5,.2885);
\draw (1,0.577) node [below=.015cm] {\tiny $\beta'_2$} -- (1,1.1545);\node[anchor=west,inner sep=2pt] at (1,.99) {\tiny ${X}$};
\draw (1,1.732) -- (1,1.1545);
\draw (1,0.577) -- (1.5,.2885);\node[anchor=south,inner sep=0pt] at (1.37,.49) {\tiny ${X_8}$};
\draw (2,0) -- (1.5,.2885);
\end{tikzpicture}}
$$

\emph{Omit basis indices.} In the oriented calculus, suppose each nonzero
vertex space is one-dimensional. Choose basis vectors at the external vertices.
Each internal basis sum has at most one term. We set tetrahedral coefficients to zero when a vertex is inadmissible. The symbol $\bullet$ denotes the chosen
vector or its dual for the evaluation pairing, in the indicated corner:
$$
\raisebox{-1.4cm}{
\begin{tikzpicture}[scale=1.43]
\draw (2,0) -- (1,0) node [below] {\tiny ${X_6}$};
\draw[tpe half arrow] (0,0) node [below] {\tiny $\bullet$} -- (1,0);
\draw[tpe half arrow] (2,0) node [below] {\tiny $\bullet$} -- (1.5,.866) node [right] {\tiny ${X_4}$};
\draw (1.5,.866) -- (1,1.732);
\draw[tpe half arrow] (1,1.732)	node [above] {\tiny $\bullet$} -- (.5,.866) node [left] {\tiny ${X_5}$};
\draw (.5,.866) -- (0,0);
\draw (1,0.577) -- (.5,.2885) node [above] {\tiny ${X_3}$};
\draw[tpe half arrow] (0,0) -- (.5,.2885);
\draw (1,0.577) node [below=.015cm] {\tiny $\bullet$} -- (1,1.1545);\node[anchor=west,inner sep=2pt] at (1,.99) {\tiny ${X_2}$};
\draw[tpe half arrow] (1,1.732) -- (1,1.1545);
\draw (1,0.577) -- (1.5,.2885) node [above] {\tiny ${X_1}$};
\draw[tpe half arrow] (2,0) -- (1.5,.2885);
\end{tikzpicture}} \hspace{-.2cm}
\raisebox{-1.4cm}{
\begin{tikzpicture}[scale=1.43]
\draw[tpe half arrow] (2,0) -- (1,0) node [below] {\tiny ${X_9}$};
\draw (0,0) node [below] {\tiny $\bullet$} -- (1,0);
\draw (2,0) node [below] {\tiny $\bullet$} -- (1.5,.866) node [right] {\tiny ${X_8}$};
\draw[tpe half arrow] (1,1.732) -- (1.5,.866);
\draw (1,1.732)	node [above] {\tiny $\bullet$} -- (.5,.866) node [left] {\tiny ${X_7}$};
\draw[tpe half arrow] (0,0) -- (.5,.866);
\draw[tpe half arrow] (1,0.577) -- (.5,.2885) node [above] {\tiny ${X_1}$};
\draw (0,0) -- (.5,.2885);
\draw[tpe half arrow] (1,0.577) node [below=.015cm] {\tiny $\bullet$} -- (1,1.1545);\node[anchor=west,inner sep=2pt] at (1,.99) {\tiny ${X_2}$};
\draw (1,1.732) -- (1,1.1545);
\draw[tpe half arrow] (1,0.577) -- (1.5,.2885) node [above] {\tiny ${X_3}$};
\draw (2,0) -- (1.5,.2885);
\end{tikzpicture}}
\hspace{-.4cm} = \hspace{-.1cm} \sum_{X \in \mathcal{O}} \hspace{-.1cm} d_X \hspace{-.3cm}
\raisebox{-1.4cm}{
\begin{tikzpicture}[scale=1.43]
\draw (2,0) -- (1,0) node [below] {\tiny ${X_1}$};
\draw[tpe half arrow] (0,0) -- (1,0);
\draw (.40,.40) node {\tiny $\bullet$};
\draw[tpe half arrow] (2,0) -- (1.5,.866) node [right] {\tiny ${X_7}$};
\draw (1.60,.40) node {\tiny $\bullet$};
\draw (1,1.732) -- (1.5,.866);
\draw (1,1.732)	node [above] {\tiny $\bullet$} -- (.5,.866) node [left] {\tiny ${X_4}$};
\draw[tpe half arrow] (0,0) -- (.5,.866);
\draw[tpe half arrow] (1,0.577) -- (.5,.2885);\node[anchor=south,inner sep=0pt] at (.63,.49) {\tiny ${X_6}$};
\draw (0,0) -- (.5,.2885);
\draw (1,0.577) node [below=.015cm] {\tiny $\bullet$} -- (1,1.1545);\node[anchor=west,inner sep=2pt] at (1,.99) {\tiny ${X}$};
\draw[tpe half arrow] (1,1.732) -- (1,1.1545);
\draw[tpe half arrow] (1,0.577) -- (1.5,.2885);\node[anchor=south,inner sep=0pt] at (1.37,.49) {\tiny ${X_9}$};
\draw (2,0) -- (1.5,.2885);
\end{tikzpicture}}
\raisebox{-1.4cm}{
\begin{tikzpicture}[scale=1.43]
\draw (2,0) -- (1,0) node [below] {\tiny ${X_2}$};
\draw[tpe half arrow] (0,0) -- (1,0);
\draw (.40,.40) node {\tiny $\bullet$};
\draw[tpe half arrow] (2,0) -- (1.5,.866) node [right] {\tiny ${X_8}$};
\draw (1.60,.40) node {\tiny $\bullet$};
\draw (1,1.732) -- (1.5,.866);
\draw (1,1.732)	node [above] {\tiny $\bullet$} -- (.5,.866) node [left] {\tiny $X_5$};
\draw[tpe half arrow] (0,0) -- (.5,.866);
\draw[tpe half arrow] (1,0.577) -- (.5,.2885);\node[anchor=south,inner sep=0pt] at (.63,.49) {\tiny ${X_4}$};
\draw (0,0) -- (.5,.2885);
\draw (1,0.577) node [below=.015cm] {\tiny $\bullet$} -- (1,1.1545);\node[anchor=west,inner sep=2pt] at (1,.99) {\tiny ${X}$};
\draw[tpe half arrow] (1,1.732) -- (1,1.1545);
\draw[tpe half arrow] (1,0.577) -- (1.5,.2885);\node[anchor=south,inner sep=0pt] at (1.37,.49) {\tiny ${X_7}$};
\draw (2,0) -- (1.5,.2885);
\end{tikzpicture}}
\raisebox{-1.4cm}{
\begin{tikzpicture}[scale=1.43]
\draw (2,0) -- (1,0) node [below] {\tiny ${X_3}$};
\draw[tpe half arrow] (0,0) -- (1,0);
\draw (.40,.40) node {\tiny $\bullet$};
\draw[tpe half arrow] (2,0) -- (1.5,.866) node [right] {\tiny ${X_9}$};
\draw (1.60,.40) node {\tiny $\bullet$};
\draw (1,1.732) -- (1.5,.866);
\draw (1,1.732)	node [above] {\tiny $\bullet$} -- (.5,.866) node [left] {\tiny ${X_6}$};
\draw[tpe half arrow] (0,0) -- (.5,.866);
\draw[tpe half arrow] (1,0.577) -- (.5,.2885);\node[anchor=south,inner sep=0pt] at (.63,.49) {\tiny ${X_5}$};
\draw (0,0) -- (.5,.2885);
\draw (1,0.577) node [below=.015cm] {\tiny $\bullet$} -- (1,1.1545);\node[anchor=west,inner sep=2pt] at (1,.99) {\tiny ${X}$};
\draw[tpe half arrow] (1,1.732) -- (1,1.1545);
\draw[tpe half arrow] (1,0.577) -- (1.5,.2885);\node[anchor=south,inner sep=0pt] at (1.37,.49) {\tiny ${X_8}$};
\draw (2,0) -- (1.5,.2885);
\end{tikzpicture}}
$$

\emph{Omit arrows and corner marks.} Use the symmetric $\kappa_i$ above.
Assume that the fusion rules are
multiplicity-free and $\nu_3(X)=1$ whenever $\hc(\one,X^{\otimes3})\ne0$.
Choose one basis vector in each cyclic orbit of vertex spaces and rotate it
to the other corners. The only triples with nontrivial stabilizer are
$(X,X,X)$, where rotation acts trivially by Lemma~\ref{lem:rot}. For two
distinct orbits exchanged by reversing the order of the labels, use dual
bases. For an orbit fixed by reversal, the pairing of a vector with its
corresponding rotation is a nonzero scalar $c$. Rescaling by $c^{-1/2}$
makes the two vectors dual; the standing assumption that $\field$ is
algebraically closed suffices here. The corner marks may thus be dropped:
$$
\raisebox{-1.4cm}{
\begin{tikzpicture}[scale=1.43]
\draw (2,0) -- (1,0) node [below] {\tiny ${X_6}$};
\draw (0,0) node [below] { } -- (1,0);
\draw (2,0) node [below] { } -- (1.5,.866) node [right] {\tiny ${X_4}$};
\draw (1.5,.866) -- (1,1.732);
\draw (1,1.732)	node [above] { } -- (.5,.866) node [left] {\tiny ${X_5}$};
\draw (.5,.866) -- (0,0);
\draw (1,0.577) -- (.5,.2885) node [above] {\tiny ${X_3}$};
\draw (0,0) -- (.5,.2885);
\draw (1,0.577) node [below=.015cm] { } -- (1,1.1545);\node[anchor=west,inner sep=2pt] at (1,.99) {\tiny ${X_2}$};
\draw (1,1.732) -- (1,1.1545);
\draw (1,0.577) -- (1.5,.2885) node [above] {\tiny ${X_1}$};
\draw (2,0) -- (1.5,.2885);
\end{tikzpicture}} \hspace{-.2cm}
\raisebox{-1.4cm}{
\begin{tikzpicture}[scale=1.43]
\draw (2,0) -- (1,0) node [below] {\tiny ${X_9}$};
\draw (0,0) node [below] { } -- (1,0);
\draw (2,0) node [below] { } -- (1.5,.866) node [right] {\tiny ${X_8}$};
\draw (1,1.732) -- (1.5,.866);
\draw (1,1.732)	node [above] { } -- (.5,.866) node [left] {\tiny ${X_7}$};
\draw (0,0) -- (.5,.866);
\draw (1,0.577) -- (.5,.2885) node [above] {\tiny ${X_1}$};
\draw (0,0) -- (.5,.2885);
\draw (1,0.577) node [below=.015cm] { } -- (1,1.1545);\node[anchor=west,inner sep=2pt] at (1,.99) {\tiny ${X_2}$};
\draw (1,1.732) -- (1,1.1545);
\draw (1,0.577) -- (1.5,.2885) node [above] {\tiny ${X_3}$};
\draw (2,0) -- (1.5,.2885);
\end{tikzpicture}}
\hspace{-.4cm} = \hspace{-.1cm} \sum_{X \in \mathcal{O}} \hspace{-.1cm} d_X \hspace{-.3cm}
\raisebox{-1.4cm}{
\begin{tikzpicture}[scale=1.43]
\draw (2,0) -- (1,0) node [below] {\tiny ${X_1}$};
\draw (0,0) -- (1,0);
\draw (.40,.40) node { };
\draw (2,0) -- (1.5,.866) node [right] {\tiny ${X_7}$};
\draw (1.60,.40) node { };
\draw (1,1.732) -- (1.5,.866);
\draw (1,1.732)	node [above] { } -- (.5,.866) node [left] {\tiny ${X_4}$};
\draw (0,0) -- (.5,.866);
\draw (1,0.577) -- (.5,.2885);\node[anchor=south,inner sep=0pt] at (.63,.49) {\tiny ${X_6}$};
\draw (0,0) -- (.5,.2885);
\draw (1,0.577) node [below=.015cm] { } -- (1,1.1545);\node[anchor=west,inner sep=2pt] at (1,.99) {\tiny ${X}$};
\draw (1,1.732) -- (1,1.1545);
\draw (1,0.577) -- (1.5,.2885);\node[anchor=south,inner sep=0pt] at (1.37,.49) {\tiny ${X_9}$};
\draw (2,0) -- (1.5,.2885);
\end{tikzpicture}}
\raisebox{-1.4cm}{
\begin{tikzpicture}[scale=1.43]
\draw (2,0) -- (1,0) node [below] {\tiny ${X_2}$};
\draw (0,0) -- (1,0);
\draw (.40,.40) node { };
\draw (2,0) -- (1.5,.866) node [right] {\tiny ${X_8}$};
\draw (1.60,.40) node { };
\draw (1,1.732) -- (1.5,.866);
\draw (1,1.732)	node [above] { } -- (.5,.866) node [left] {\tiny $X_5$};
\draw (0,0) -- (.5,.866);
\draw (1,0.577) -- (.5,.2885);\node[anchor=south,inner sep=0pt] at (.63,.49) {\tiny ${X_4}$};
\draw (0,0) -- (.5,.2885);
\draw (1,0.577) node [below=.015cm] { } -- (1,1.1545);\node[anchor=west,inner sep=2pt] at (1,.99) {\tiny ${X}$};
\draw (1,1.732) -- (1,1.1545);
\draw (1,0.577) -- (1.5,.2885);\node[anchor=south,inner sep=0pt] at (1.37,.49) {\tiny ${X_7}$};
\draw (2,0) -- (1.5,.2885);
\end{tikzpicture}}
\raisebox{-1.4cm}{
\begin{tikzpicture}[scale=1.43]
\draw (2,0) -- (1,0) node [below] {\tiny ${X_3}$};
\draw (0,0) -- (1,0);
\draw (.40,.40) node { };
\draw (2,0) -- (1.5,.866) node [right] {\tiny ${X_9}$};
\draw (1.60,.40) node { };
\draw (1,1.732) -- (1.5,.866);
\draw (1,1.732)	node [above] { } -- (.5,.866) node [left] {\tiny ${X_6}$};
\draw (0,0) -- (.5,.866);
\draw (1,0.577) -- (.5,.2885);\node[anchor=south,inner sep=0pt] at (.63,.49) {\tiny ${X_5}$};
\draw (0,0) -- (.5,.2885);
\draw (1,0.577) node [below=.015cm] { } -- (1,1.1545);\node[anchor=west,inner sep=2pt] at (1,.99) {\tiny ${X}$};
\draw (1,1.732) -- (1,1.1545);
\draw (1,0.577) -- (1.5,.2885);\node[anchor=south,inner sep=0pt] at (1.37,.49) {\tiny ${X_8}$};
\draw (2,0) -- (1.5,.2885);
\end{tikzpicture}}
$$
\end{remark}

\begin{remark} \label{rk:nonplanar}
Let us conclude this section with some remarks on non-planar trivalent graphs and braiding. Besides the tetrahedral and triangular prism graphs, there exists a third connected simple trivalent graph with no more than six vertices, which is \emph{non-planar}: the complete bipartite graph $K_{3,3}$. Below are three representations of this graph:
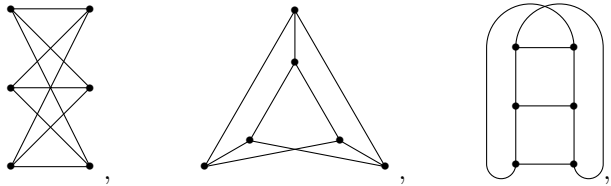
$$
\begin{tikzpicture}[scale=1.05]
\draw (0,0) node {\tiny $\bullet$} -- (1,0) node {\tiny $\bullet$} -- (0,-1) -- (1,-2);
\draw (0,-1) node  {\tiny $\bullet$}-- (1,-1) node {\tiny $\bullet$} -- (0,-2) -- (1,0);
\draw (0,-2)  node  {\tiny $\bullet$}-- (1,-2)  node  {\tiny $\bullet$}  -- (0,0) -- (1,-1);
\end{tikzpicture}
, \hspace{1cm}
\begin{tikzpicture}[scale=1.2]
\draw (0,0) node {\tiny $\bullet$} -- (1.5,.2885) node {\tiny $\bullet$};
\draw (2,0) --(1,1.732) node {\tiny $\bullet$}--(0,0);
\draw (0,0)--(.5,.2885) node {\tiny $\bullet$};
\draw (1,1.732)--(1,1.1545) node {\tiny $\bullet$};
\draw (2,0)--(1.5,.2885)--(1,1.1545)--(.5,.2885)-- (2,0) node {\tiny $\bullet$};
\end{tikzpicture}
, \hspace{1cm}
\begin{tikzpicture}[scale=.775]
\draw (0,-1)-- (0,0) node {\tiny $\bullet$} -- (1,0) node {\tiny $\bullet$} -- (1,-1);
\draw (0,-2)--(0,-1) node  {\tiny $\bullet$}-- (1,-1) node {\tiny $\bullet$} -- (1,-2);
\draw (0,-2) node  {\tiny $\bullet$}-- (1,-2)  node  {\tiny $\bullet$};
\draw (0,-2) arc (0:-180:.25) --++ (0,2) arc (180:0:.75);
\draw (1,-2) arc (-180:0:.25) --++ (0,2) arc (0:180:.75);
\end{tikzpicture},
$$
The first representation is the conventional one, the second is crafted to resemble the typical depiction of the triangular prism graph, and the third is used to derive the following monoidal category version, assuming the presence of a \emph{braiding}: 
$$
\begin{tikzpicture}
\draw (-.7,0) --++ (0,-.4) --++ (0,-.9) arc (-180:0:.7 and .3);
\draw (0,0) --++ (0,-.4) arc (-180:0:1.307 and .5);
\draw (.7,0) --++ (0,-.4) arc (-180:0: .25) arc (180:140:.35); 
\draw (-.7,-2) --++ (0,-.4) --++ (0,-.9) arc (-180:0:.7 and .3);
\draw (0,-2) --++ (0,-.4) arc (-180:0:1.307 and .5);
\draw (.7,-2) --++ (0,-.4) arc (-180:0:.25) --++ (0,1.1) arc (0:180:.25);

\draw (-.7,-4) --++ (0,-.4) arc (0:-180:.25) --++ (0,4.51);
\draw (0,-4) --++ (0,-.4) arc (-180:0:1.307 and .5);
\draw (.7,-4) --++ (0,-.4) arc (-180:0:.25) --++ (0,1.1) arc (0:180:.25);

\draw (3-.385-.7,0) --++ (0,-.4) arc (0:-180: .25) arc (0:46:.35) arc (226:180:.35) arc (0:180:1.2 and .5);%
\draw (3-.385,0) --++ (0,-.4) ;
\draw (3-.385+.7,0) --++ (0,-1.3) arc (0:-180:.7 and .3);

\draw (3-.385-.7,-2) --++ (0,-.4) arc (0:-180:.25) --++ (0,1.1) arc (180:0:.25);
\draw (3-.385,-2) --++ (0,-.4) ;
\draw (3-.385+.7,-2) --++ (0,-1.3) arc (0:-180:.7 and .3);

\draw (3-.385-.7,-4) --++ (0,-.4) arc (0:-180:.25) --++ (0,1.1) arc (180:0:.25);
\draw (3-.385,-4) --++ (0,-.4) ;
\draw (3-.385+.7,-4) --++ (0,-.4) arc (-180:0:.25) --++ (0,4.5) arc (0:180:1.2 and .5) arc (0:-40:.35);

\draw (0,0)  node [draw,fill=white,minimum width=2cm,minimum height=.4cm] { };
\draw (0,-2)  node [draw,fill=white,minimum width=2cm,minimum height=.4cm] { };
\draw (0,-4)  node [draw,fill=white,minimum width=2cm,minimum height=.4cm] { };
\draw (3-.385,0)  node [draw,fill=white,minimum width=2cm,minimum height=.4cm] { };
\draw (3-.385,-2)  node [draw,fill=white,minimum width=2cm,minimum height=.4cm] { };
\draw (3-.385,-4)  node [draw,fill=white,minimum width=2cm,minimum height=.4cm] { };
\end{tikzpicture}
$$
This illustration closely resembles the one in Definition \ref{def:TPCat}, with the primary difference being the inclusion of braiding. It suggests a \emph{braided version} of Theorem \ref{thm:TPE} (TPE). In general, we might derive such equations from any trivalent graph.
\end{remark}

\section{Pentagon and triangular prism equations}

\begin{notation} \label{not:1}
Let $(X_i)_{i\in I}$ represent the isomorphism classes of simple objects of a fusion category $\mC$, with $X_1=\one$. Write $N_{i,j}^k=\dim_{\field}\hc(X_i\otimes X_j,X_k)$ and abbreviate the edge label $X_i$ to $i$. In the pivotal setting we write $d_i=\dim(X_i)=\dim_a(X_i)$ for the dimension defined by the pivotal structure $a$ \cite[Definition~4.7.11]{EGNO15}. Selfduality means $X\simeq X^*$.
\end{notation}

\subsection{Pentagon equations} \label{sec:PE}
Choose a basis $\mB(i,j;k)$ of $\hc(X_i\otimes X_j,X_k)$. We draw $\alpha\in\mB(i,j;k)$ as a trivalent vertex:
\begin{equation}\label{Fig: alpha}
\raisebox{-.8cm}{
\begin{tikzpicture}
\begin{scope}[scale=.8]
\draw (0,2)--(.5,1.5);
\draw (1,1)--(.5,1.5);
\draw (2,2)--(1.5,1.5);
\draw (1,1)--(1.5,1.5);
\draw (1,1)--(1,.5);
\draw (1,0)--(1,.5);
\node at (1,1.3) {$\alpha$};
\node at (.2,1.5) {${i}$};
\node at (1.9,1.5) {${j}$};
\node at (.8,.5) {${k}$};
\end{scope}
\end{tikzpicture}}.
\end{equation}

For fixed $\mu_1\in\mB(i_1,i_2;i_3)$ and $\mu_2\in\mB(i_3,i_4;i_5)$, the $F$-symbols are the coefficients of the change between the two bases of $\hc(X_{i_1}\otimes X_{i_2}\otimes X_{i_4},X_{i_5})$ obtained by first decomposing the left or the right pair of factors \cite{Bon07,DaHaWa,wang}. With the associativity constraint understood, they are defined by
\[
\mu_2\circ(\mu_1 \otimes \id_{X_{i_4}})
=
\sum_{i_6}\sum_{\mu_3,\mu_4}
\left(
\begin{array}{ccc | cc}
i_1& i_2 & i_3 & \mu_1 & \mu_2 \\
i_4& i_5 & i_6 & \mu_3 & \mu_4
\end{array}
\right)
\mu_4\circ(\id_{X_{i_1}} \otimes \mu_3),
\]
where $i_6\in I$, $\mu_3\in\mB(i_2,i_4;i_6)$ and $\mu_4\in\mB(i_1,i_6;i_5)$. In pictures,
\begin{equation} \label{Equ: F-symbols}
\raisebox{-1.25cm}{
\begin{tikzpicture}
\begin{scope}[scale=1.25]
\node[above left,inner sep=.7pt] at (0,0) {$i_1$};
\node[above left,inner sep=.7pt] at (1,0) {$i_2$};
\node[above left,inner sep=.7pt] at (2,0) {$i_4$};
\node at (.75-.2,-.75) {$i_3$};
\node at (1-.2,-1.5) {$i_5$};
\node at (.5,-.5+.29) {\scriptsize $\textcolor{orange!65!black}{\mu_1}$};
\node at (1,-1+.29) {\scriptsize $\textcolor{orange!65!black}{\mu_2}$};
\draw (0,0)--++(1,-1);
\draw (1,0)--++(-.5,-.5);
\draw (2,0)--++(-1,-1)--++(0,-.5);
\end{scope}
\end{tikzpicture}}
=
\sum_{i_6}\sum_{\mu_3,\mu_4}
\left(
\begin{array}{ccc | cc}
i_1& i_2 & i_3 & \mu_1 & \mu_2 \\
i_4& i_5 & i_6 & \mu_3 & \mu_4
\end{array}
\right)
\raisebox{-1.25cm}{
\begin{tikzpicture}
\begin{scope}[scale=1.25]
\node[above left,inner sep=.7pt] at (0,0) {$i_1$};
\node[above left,inner sep=.7pt] at (1,0) {$i_2$};
\node[above left,inner sep=.7pt] at (2,0) {$i_4$};
\node at (1.25+.2,-.75) {$i_6$};
\node at (1-.2,-1.5) {$i_5$};
\node at (1.5,-.5+.29) {\scriptsize $\textcolor{orange!65!black}{\mu_3}$};
\node at (1,-1+.29) {\scriptsize $\textcolor{orange!65!black}{\mu_4}$};
\draw (0,0)--++(1,-1);
\draw (1,0)--++(.5,-.5);
\draw (2,0)--++(-1,-1)--++(0,-.5);
\end{scope}
\end{tikzpicture}}.
\end{equation}
Expanding the two routes of the pentagon in these bases gives the \emph{pentagon equations} (PE)
\begin{align}\label{Equ: Pentagon Equation}
&
\sum_{\mu_0} \left(
\begin{array}{ccc | cc}
i_2& i_7 & i_8 & \mu_5 & \mu_6 \\
i_9& i_3 & i_1 & \mu_4 & \mu_0
\end{array}
\right)
\left(
\begin{array}{ccc | cc}
i_5& i_4 & i_2 & \mu_2 & \mu_0 \\
i_1& i_3 & i_6 & \mu_1 & \mu_3
\end{array}
\right)
\\* \nonumber
=&
\sum_{\spec}
\sum_{\mu_7,\mu_8,\mu_9}
\left(
\begin{array}{ccc | cc}
i_5& i_4 & i_2 & \mu_2 & \mu_5 \\
i_7& i_8 & \spec & \mu_7 & \mu_8
\end{array}
\right)
\left(
\begin{array}{ccc | cc}
i_5& \spec & i_8 & \mu_8 & \mu_6 \\
i_9& i_3 & i_6 & \mu_9 & \mu_3
\end{array}
\right)
\left(
\begin{array}{ccc | cc}
i_4& i_7 & \spec & \mu_7 & \mu_9 \\
i_9& i_6 & i_1 & \mu_4 & \mu_1
\end{array}
\right).
\end{align}
Here $i_1,\ldots,i_9\in I$, and the basis labels have the following types:
$\mu_1 \in\mB(i_4,i_1;i_6)$, $\mu_2 \in\mB(i_5,i_4;i_2)$, $\mu_3 \in\mB(i_5,i_6;i_3)$, $\mu_4 \in\mB(i_7,i_9;i_1)$,  $\mu_5 \in\mB(i_2,i_7;i_8)$, $\mu_6 \in\mB(i_8,i_9;i_3)$,  $\mu_7 \in\mB(i_4,i_7;i_0)$,  $\mu_8 \in\mB(i_5,i_0;i_8)$,  $\mu_9 \in\mB(i_0,i_9;i_6)$ and $\mu_0 \in\mB(i_2,i_1;i_3)$,
where $i_0\in I$ is summed on the right. To read the identity, fix the bottom and upper-right trees of Figure~\ref{Fig: Pentagon Equation}. Expand the first in the basis containing the second by successive applications of \eqref{Equ: F-symbols}. Along the two-move route the only free label is $\mu_0$. Along the three-move route it is necessary to sum over $i_0,\mu_7,\mu_8,\mu_9$. The pentagon axiom equates the two coefficients.

\begin{figure}[ht]
\begin{tikzpicture}
\begin{scope}[scale=1.3]
\begin{scope}
\node[above left,inner sep=.7pt] at (0,0) {$i_5$};
\node[above left,inner sep=.7pt] at (1,0) {$i_4$};
\node[above left,inner sep=.7pt] at (2,0) {$i_7$};
\node[above left,inner sep=.7pt] at (3,0) {$i_9$};
\node at (1.75-.2,-.75) {$\spec$};
\node at (1.75+.2,-1.25) {$i_6$};
\node at (1.5-.2,-1.75) {$i_3$};
\node at (1.5,-.5+.29) {\scriptsize $\textcolor{orange!65!black}{\mu_7}$};
\node at (2,-1+.29) {\scriptsize $\textcolor{orange!65!black}{\mu_9}$};
\node at (1.5,-1.5+.29) {\scriptsize $\textcolor{orange!65!black}{\mu_3}$};
\draw (0,0) --++ (1.5,-1.5)--++(0,-.5);
\draw (1,0) --++ (1,-1);
\draw (2,0) --++ (-.5,-.5);
\draw (3,0) --++ (-1.5,-1.5);
\end{scope}
\begin{scope}[shift={(5,0)}]
\node[above left,inner sep=.7pt] at (0,0) {$i_5$};
\node[above left,inner sep=.7pt] at (1,0) {$i_4$};
\node[above left,inner sep=.7pt] at (2,0) {$i_7$};
\node[above left,inner sep=.7pt] at (3,0) {$i_9$};
\node at (2.25+.2,-.75) {$i_1$};
\node at (1.75+.2,-1.25) {$i_6$};
\node at (1.5-.2,-1.75) {$i_3$};
\node at (2.5,-.5+.29) {\scriptsize $\textcolor{orange!65!black}{\mu_4}$};
\node at (2,-1+.29) {\scriptsize $\textcolor{orange!65!black}{\mu_1}$};
\node at (1.5,-1.5+.29) {\scriptsize $\textcolor{orange!65!black}{\mu_3}$};
\draw (0,0) --++ (1.5,-1.5)--++(0,-.5);
\draw (1,0) --++ (1,-1);
\draw (2,0) --++ (.5,-.5);
\draw (3,0) --++ (-1.5,-1.5);
\end{scope}
\begin{scope}[shift={(-2.5,-3)}]
\node[above left,inner sep=.7pt] at (0,0) {$i_5$};
\node[above left,inner sep=.7pt] at (1,0) {$i_4$};
\node[above left,inner sep=.7pt] at (2,0) {$i_7$};
\node[above left,inner sep=.7pt] at (3,0) {$i_9$};
\node at (1.25+.2,-.75) {$\spec$};
\node at (1.25-.2,-1.25) {$i_8$};
\node at (1.5-.2,-1.75) {$i_3$};
\node at (1.5,-.5+.29) {\scriptsize $\textcolor{orange!65!black}{\mu_7}$};
\node at (1,-1+.29) {\scriptsize $\textcolor{orange!65!black}{\mu_8}$};
\node at (1.5,-1.5+.29) {\scriptsize $\textcolor{orange!65!black}{\mu_6}$};
\draw (0,0) --++ (1.5,-1.5)--++(0,-.5);
\draw (1,0) --++ (.5,-.5);
\draw (2,0) --++ (-1,-1);
\draw (3,0) --++ (-1.5,-1.5);
\end{scope}
\begin{scope}[shift={(7.5,-3)}]
\node[above left,inner sep=.7pt] at (0,0) {$i_5$};
\node[above left,inner sep=.7pt] at (1,0) {$i_4$};
\node[above left,inner sep=.7pt] at (2,0) {$i_7$};
\node[above left,inner sep=.7pt] at (3,0) {$i_9$};
\node at (.75-.2,-.75) {$i_2$};
\node at (2.25-.2,-.75) {$i_1$};
\node at (1.5-.2,-1.75) {$i_3$};
\node at (.5,-.5+.29) {\scriptsize $\textcolor{orange!65!black}{\mu_2}$};
\node at (2.5,-.5+.29) {\scriptsize $\textcolor{orange!65!black}{\mu_4}$};
\node at (1.5,-1.5+.29) {\scriptsize $\textcolor{orange!65!black}{\mu_0}$};
\draw (0,0) --++ (1.5,-1.5)--++(0,-.5);
\draw (1,0) --++ (-.5,-.5);
\draw (2,0) --++ (.5,-.5);
\draw (3,0) --++ (-1.5,-1.5);
\end{scope}
\begin{scope}[shift={(2.5,-5)}]
\node[above left,inner sep=.7pt] at (0,0) {$i_5$};
\node[above left,inner sep=.7pt] at (1,0) {$i_4$};
\node[above left,inner sep=.7pt] at (2,0) {$i_7$};
\node[above left,inner sep=.7pt] at (3,0) {$i_9$};
\node at (.75-.2,-.75) {$i_2$};
\node at (1.25-.2,-1.25) {$i_8$};
\node at (1.5-.2,-1.75) {$i_3$};
\node at (.5,-.5+.29) {\scriptsize $\textcolor{orange!65!black}{\mu_2}$};
\node at (1,-1+.29) {\scriptsize $\textcolor{orange!65!black}{\mu_5}$};
\node at (1.5,-1.5+.29) {\scriptsize $\textcolor{orange!65!black}{\mu_6}$};
\draw (0,0) --++ (1.5,-1.5)--++(0,-.5);
\draw (1,0) --++ (-.5,-.5);
\draw (2,0) --++ (-1,-1);
\draw (3,0) --++ (-1.5,-1.5);
\end{scope}
\draw [red, thick,tpe forward] (2.4,-6.5)--(-1,-5.2) node[pos=.57,below=3pt,text=black,font=\small] {$G_1$};
\draw [red, thick,tpe forward] (-1,-2.5)--(1,-1.5) node[pos=.57,above=3pt,text=black,font=\small] {$G_2$};
\draw [red, thick,tpe forward] (3.2,-.75)--(5,-.75) node[pos=.57,above=3pt,text=black,font=\small] {$G_3$};
\draw [blue, thick,tpe forward] (9,-2.5)--(7,-1.5) node[pos=.57,above=3pt,text=black,font=\small] {$F_2$};
\draw [blue, thick,tpe forward] (5.6,-6.5)--(9,-5.2) node[pos=.57,below=3pt,text=black,font=\small] {$F_1$};
\end{scope}
\end{tikzpicture}
\caption[Pentagon equation]{The two routes in the pentagon \eqref{Equ: Pentagon Equation}. Each arrow applies \eqref{Equ: F-symbols} once. Theorem~\ref{Thm: Ob-0} uses the blue route $F_1,F_2$, and Theorem~\ref{Thm: Ob-1} uses the red route $G_1,G_2,G_3$.}
\label{Fig: Pentagon Equation}
\end{figure}
\subsection{TPE versus PE} \label{sec:TPEvsPE}
We first bend a fusion vertex into a three-legged invariant tensor. We then close the two sides of an $F$-move and follow the resulting tetrahedra through the pentagon.

\begin{definition}\label{def:213}
Let $\mC$ be a monoidal category with chosen left duals. Consider the functors
$G,H:\mC^{\mathrm{op}}\times\mC^{\mathrm{op}}\times\mC\to\mathrm{Set}$ given by
\[
 G(X,Y,Z)=\hc(X\otimes Y,Z),\qquad
 H(X,Y,Z)=\hc(\one,Z\otimes Y^*\otimes X^*).
\]
For $\mu:X\otimes Y\to Z$, the following formulas give the natural isomorphism $G\simeq H$ and its inverse, with $(X\otimes Y)^*=Y^*\otimes X^*$ as in \S\ref{sec:pre}:
\begin{align}
 \widetilde\mu
 &= \vcenter{\hbox{%
\begin{tikzpicture}[every node/.style={font=\scriptsize,inner sep=1.5pt}]
\draw (-.2,0)--(-.2,.37) arc (180:0:.74 and .37) --(1.28,-.44) node[below] {$X^*$};
\draw (.2,0)--(.2,.37) arc (180:0:.24 and .12) --(.68,-.44) node[below] {$Y^*$};
\draw (0,0)--(0,-.44) node[below] {$Z$};
\node[draw,fill=white,minimum width=.8cm,font=\small] at (0,0) {$\mu$};
\end{tikzpicture}}} =
   (\mu\otimes\id_{(X\otimes Y)^*})\circ\coev_{X\otimes Y},
   \label{eq:tilde-algebraic}\\*
 \mu
 &= \vcenter{\hbox{%
\begin{tikzpicture}[every node/.style={font=\scriptsize,inner sep=1.5pt}]
\draw (.4,0)--(.4,-.42) arc (-180:0:.26 and .13) --(.92,.28) node[above] {$X$};
\draw (0,0)--(0,-.42) arc (-180:0:.74 and .37) --(1.48,.28) node[above] {$Y$};
\draw (-.4,0)--(-.4,-.42) node[below] {$Z$};
\node[draw,fill=white,minimum width=1cm,font=\small] at (0,0) {$\widetilde\mu$};
\end{tikzpicture}}} =
   (\id_Z\otimes\ev_{X\otimes Y})
      \circ(\widetilde\mu\otimes\id_{X\otimes Y}).\notag
\end{align}
The zigzag identities make these assignments inverse to each other \cite[Proposition~2.10.8]{EGNO15}.
\end{definition}

For the remainder of \S\ref{sec:TPEvsPE}, let $\mC$ be a spherical fusion category in the strictly pivotal realization of Remark~\ref{rk:StrPiv}. Fix simple objects $A,B,C$. For each $\mu$ in a basis of $\hc(A\otimes B,C)$, let $q_\mu\in\hc(C,A\otimes B)$ be its dual for the trace pairing of \S\ref{sub:reso}, characterized by $\tr(\nu\circ q_\mu)=\delta_{\nu,\mu}$ for every $\nu$ in the basis. Bending $q_\mu$ gives
\begin{equation}\label{eq:right-tilde-pivotal}
 {}'\widetilde\mu=(q_\mu\otimes\id_{C^*})\circ\coev_C
 \ \in\hc(\one,A\otimes B\otimes C^*).
\end{equation}
Set $Z=C\otimes B^*\otimes A^*$. Closing the two tensors using the evaluation pairing $b_Z$ of Lemma~\ref{lem:bili} gives
\[
 b_Z({}'\widetilde\mu,\widetilde\nu)
 =\tr(\nu\circ q_\mu)
 =b_{Z^*}(\widetilde\nu,{}'\widetilde\mu)=\delta_{\nu,\mu}.
\]
Thus the primed basis in $\hc(\one,Z^*)$ is dual to the bent basis in $\hc(\one,Z)$ for $b_Z$. The categorical dual of the morphism $\widetilde\mu$ itself would have type $Z^*\to\one$.

\begin{theorem}[PE--TPE, change of basis]\label{thm:PE-TPE}
The triangular prism equation of Theorem~\ref{thm:sTPE2} is the pentagon equation~\eqref{Equ: Pentagon Equation} under the following invertible change of coordinates. Write $V_i$ for the simple representative indexed by $i$ in Notation~\ref{not:1}, to distinguish it from the numbered edges of the prism. For given pentagon labels, the prism edges are $X=V_{i_0}$ and
\[
 (X_1,\ldots,X_9)=
 (V_{i_3}^*,V_{i_1},V_{i_2},V_{i_6},V_{i_4},V_{i_5}^*,
  V_{i_9}^*,V_{i_7},V_{i_8}).
\]
The table gives the vertex morphisms, all with source $\one$.
\begin{center}\normalfont
\renewcommand{\arraystretch}{1.15}
\begin{tabular}{c|c|c}
Vertex & Morphism & Target\\\hline
$\alpha_1$ & $\rho({}'\widetilde\mu_3)$ & $X_4\otimes X_1\otimes X_6^*$\\
$\alpha_2$ & ${}'\widetilde\mu_1$ & $X_5\otimes X_2\otimes X_4^*$\\
$\alpha_3$ & $\rho^{-1}(\widetilde\mu_2)$ & $X_6\otimes X_3\otimes X_5^*$\\
$\alpha_4$ & $\rho^{-1}(\widetilde\mu_6)$ & $X_9^*\otimes X_1^*\otimes X_7$\\
$\alpha_5$ & $\rho({}'\widetilde\mu_4)$ & $X_7^*\otimes X_2^*\otimes X_8$\\
$\alpha_6$ & $\rho(\widetilde\mu_5)$ & $X_8^*\otimes X_3^*\otimes X_9$\\
$\beta_0$ & ${}'\widetilde\mu_0$ & $X_3\otimes X_2\otimes X_1$\\
$\beta_1$ & $\rho^{-1}({}'\widetilde\mu_9)$ & $X_4^*\otimes X\otimes X_7^*$\\
$\beta_2$ & $\rho^{-1}(\widetilde\mu_7)$ & $X_5^*\otimes X\otimes X_8^*$\\
$\beta_3$ & ${}'\widetilde\mu_8$ & $X_6^*\otimes X\otimes X_9^*$.
\end{tabular}
\end{center}
\end{theorem}
\begin{proof}
\begingroup
\newcommand{\peportbox}[2]{\node[draw,fill=white,anchor=south,
 minimum width=.99cm,minimum height=.41cm,inner sep=1pt,font=\small]
 at (#1,0) {$#2$};}
\newcommand{\pecap}[3]{\draw (#1,0)--(#1,-.08)
 .. controls (#1,-#3) and (#2,-#3) .. (#2,-.08)--(#2,0);}
\newcommand{\pebridge}[4]{\draw (#1,0)--(#1,-.08)
 .. controls (#1,-.35) and (#2,-.35) .. (#2,-.08)--(#2,.53)
 .. controls (#2,.95) and (#3,.95) .. (#3,.53)--(#3,-.08)
 .. controls (#3,-.35) and (#4,-.35) .. (#4,-.08)--(#4,0);}
\newcommand{\peinput}[2]{\draw (#1,0)--(#1,.69)
 node[above,font=\scriptsize,inner sep=1pt] {$#2$};}
\tikzset{pecalc/.style={x=.85cm,y=.85cm,baseline=(current bounding box.center)}}

\emph{Extract the coefficient.}
Fix $\mu_1,\mu_2$ in \eqref{Equ: F-symbols} and write
$F^{j}_{\nu_3,\nu_4}$ for its coefficient with intermediate label $j$
and lower pair $\nu_3,\nu_4$. Bending the two input legs at each vertex
by Definition~\ref{def:213} gives the same recoupling:
\[
\begin{tikzpicture}[pecalc]
 \peportbox{.44}{\widetilde\mu_2}\peportbox{1.99}{\widetilde\mu_1}
 \draw (0,0)--(0,-.29) node[left,font=\scriptsize,inner sep=1pt] {$i_5$};
 \pecap{.88}{1.55}{.39}\pecap{2.43}{3.08}{.39}
 \pecap{1.99}{3.59}{.72}\pecap{.44}{4.10}{1.05}
 \node[right,font=\scriptsize,inner sep=1pt] at (1.55,-.14) {$i_3$};
 \peinput{3.08}{i_1}\peinput{3.59}{i_2}\peinput{4.10}{i_4}
\end{tikzpicture}
\quad=\quad
\sum_{j,\nu_3,\nu_4}F^{j}_{\nu_3,\nu_4}\,
\begin{tikzpicture}[pecalc]
 \peportbox{.44}{\widetilde\nu_4}\peportbox{2.59}{\widetilde\nu_3}
 \draw (0,0)--(0,-.29) node[left,font=\scriptsize,inner sep=1pt] {$i_5$};
 \pecap{.88}{1.43}{.39}\pecap{.44}{2.15}{.72}
 \pecap{3.03}{3.57}{.39}\pecap{2.59}{4.10}{.72}
 \node[right,font=\scriptsize,inner sep=1pt] at (2.15,-.14) {$j$};
 \peinput{1.43}{i_1}\peinput{3.57}{i_2}\peinput{4.10}{i_4}
\end{tikzpicture}.
\]

Choose the lower pair $\mu_3,\mu_4$ with intermediate label $i_6$ and
let $q_3,q_4$ be its trace-duals. Attach $ {}'\widetilde\mu_3$ and
$ {}'\widetilde\mu_4$ to the open legs and close the $i_5$ strand:
\[
c\;:=\;\begin{tikzpicture}[pecalc]
 \peportbox{.44}{\widetilde\mu_2}\peportbox{1.99}{\widetilde\mu_1}
 \peportbox{3.54}{{}'\widetilde\mu_3}\peportbox{5.09}{{}'\widetilde\mu_4}
 \pecap{.88}{1.55}{.40}
 \pecap{1.99}{3.10}{.66}\pecap{.44}{3.54}{1.09}
 \pebridge{2.43}{2.77}{4.32}{4.65}
 \pecap{3.98}{5.09}{.66}\pecap{0}{5.53}{1.54}
\end{tikzpicture}
\quad=\quad
\sum_{j,\nu_3,\nu_4}F^{j}_{\nu_3,\nu_4}\,
\begin{tikzpicture}[pecalc]
 \peportbox{.44}{\widetilde\nu_4}\peportbox{1.99}{\widetilde\nu_3}
 \peportbox{3.54}{{}'\widetilde\mu_3}\peportbox{5.09}{{}'\widetilde\mu_4}
 \pecap{.44}{1.55}{.66}
 \pecap{2.43}{3.10}{.40}\pecap{1.99}{3.54}{.81}
 \pebridge{.88}{1.21}{4.32}{4.65}
 \pecap{3.98}{5.09}{.66}\pecap{0}{5.53}{1.54}
\end{tikzpicture}.
\]
On the right, straighten the two legs joining the middle coupons.
The resulting map $\nu_3\circ q_3:X_{i_6}\to X_j$ vanishes unless
$j=i_6$. For $j=i_6$, Schur's lemma and the trace-dual normalization give
\[
\begin{tikzpicture}[pecalc]
 \draw (0,1.10)--(0,.63);
 \draw (-.20,.63)--(-.20,-.15) (.20,.63)--(.20,-.15);
 \draw (0,-.15)--(0,-.60);
 \node[draw,fill=white,inner sep=1.2pt,minimum width=.68cm,font=\small] at (0,.63) {$q_3$};
 \node[draw,fill=white,inner sep=1.2pt,minimum width=.68cm,font=\small] at (0,-.15) {$\nu_3$};
 \node[left,font=\scriptsize,inner sep=1pt] at (0,1.01) {$i_6$};
 \node[left,font=\scriptsize,inner sep=1pt] at (-.20,.23) {$i_2$};
 \node[right,font=\scriptsize,inner sep=1pt] at (.20,.23) {$i_4$};
\end{tikzpicture}
=
\frac{\delta_{\nu_3,\mu_3}}{d_{i_6}}\,
\begin{tikzpicture}[pecalc]
 \draw (0,0)--(0,1.0) node[midway,right,font=\scriptsize] {$i_6$};
\end{tikzpicture},
\qquad
\begin{tikzpicture}[pecalc]
 \peportbox{.44}{\widetilde\nu_4}\peportbox{1.99}{{}'\widetilde\mu_4}
 \pecap{.88}{1.55}{.40}\pecap{.44}{1.99}{.71}\pecap{0}{2.43}{1.02}
\end{tikzpicture}=\tr(\nu_4\circ q_4)
=\delta_{\nu_4,\mu_4}.
\]
Thus the closure selects precisely the requested coefficient:
\begin{equation}\label{eq:F-extraction}
 F^{i_6}_{\mu_3,\mu_4}=d_{i_6}\,c,\qquad
 c=\tr\!\left(
 \mu_2\circ(\mu_1\otimes\id_{X_{i_4}})\circ
 (\id_{X_{i_1}}\otimes q_3)\circ q_4\right).
\end{equation}

\emph{Read the same closure as a tetrahedron.}
At $ {}'\widetilde\mu_3$ start with its last leg $X_{i_6}^*$; this
changes the marked corner by $\rho^{-1}$. The other three coupons of
$c$ already have the orders prescribed by Definition~\ref{def:TetraCat}.
In the reference pattern, the result is
\begin{equation}\label{eq:F-T}
 F^{i_6}_{\mu_3,\mu_4}
 =d_{i_6}\,
 \begin{tikzpicture}[baseline=(current bounding box.center),x=1.16cm,y=1.16cm]
  \coordinate (L) at (0,0); \coordinate (R) at (2.15,0);
  \coordinate (U) at (1.075,1.86); \coordinate (M) at (1.075,.76);
  \draw[tpe forward] (R)--(L);
  \draw[tpe forward] (R)--(U);
  \draw[tpe forward] (U)--(L);
  \draw[tpe forward] (M)--(L);
  \draw[tpe forward] (M)--(U);
  \draw[tpe forward] (M)--(R);
  \node[below,font=\scriptsize,inner sep=1.5pt] at (L) {$\widetilde\mu_2$};
  \node[below,font=\scriptsize,inner sep=1.5pt] at (R) {$ {}'\widetilde\mu_4$};
  \node[above,font=\scriptsize,inner sep=1.5pt] at (U) {$\widetilde\mu_1$};
  \node[font=\scriptsize,inner sep=.8pt] at (1.075,.25) {$\rho^{-1}({}'\widetilde\mu_3)$};
 \end{tikzpicture}
 =d_{i_6}\,
 T\bigl(\rho^{-1}({}'\widetilde\mu_3),\widetilde\mu_2,
                    \widetilde\mu_1,{}'\widetilde\mu_4\bigr).
\end{equation}
\endgroup

\smallskip
\noindent\emph{Substitute in the pentagon.}
By \eqref{eq:F-T} the two factors on the left of \eqref{Equ: Pentagon Equation} contribute $d_{i_1}d_{i_6}$, and the three factors on the right contribute $d_{i_0}d_{i_6}d_{i_1}$. Cancelling the common nonzero factor leaves $d_{i_0}$.

\begingroup
\newcommand{\PEtet}[3][]{\vcenter{\hbox{%
\begin{tikzpicture}[x=1.09cm,y=1.09cm,line width=.36pt,>=stealth,
                    every node/.style={font=\scriptsize,inner sep=.45pt}]
\path[use as bounding box] (-.17,-.31) rectangle (2.17,2.04);
\coordinate (A) at (1,.577); \coordinate (B) at (0,0);
\coordinate (C) at (1,1.732); \coordinate (D) at (2,0);
\coordinate (As) at (1,.335); \coordinate (Anw) at (.78,.74);
\coordinate (Ane) at (1.22,.74);
\coordinate (Bo) at (0,-.18); \coordinate (Blo) at (.54,.13);
\coordinate (Bhi) at (.65,.66);
\coordinate (Co) at (1,1.90); \coordinate (Cl) at (.78,1.05);
\coordinate (Cr) at (1.22,1.05);
\coordinate (Do) at (2,-.18); \coordinate (Dlo) at (1.46,.13);
\coordinate (Dhi) at (1.35,.66);
\foreach \PEfrom/\PEto in {#2}{%
  \draw[tpe forward] (\PEfrom)--(\PEto);}
\foreach \PEpos/\PEtext in {#3}{\node at (\PEpos) {$\PEtext$};}
#1\end{tikzpicture}}}}
\newcommand{\PEref}[4]{\PEtet{A/D,A/C,A/B,D/B,C/B,D/C}%
                              {As/{#1},Bo/{#2},Co/{#3},Do/{#4}}}
\newcommand{\PEprod}[5]{\sum_{\mu_0}\,#1\,#2
   =\sum_{i_0}\sum_{\mu_7,\mu_8,\mu_9}\!d_{i_0}\,#3\,#4\,#5}
Write $u_j=\widetilde\mu_j$, $v_j={}'\widetilde\mu_j$, and
$w^\pm=\rho^{\pm1}(w)$ for either kind of vertex label. The vertex types determine the edge labels. The final display marks the three edges labelled $i_0$.
In the reference pattern of Definition~\ref{def:TetraPiv}, substitution in the
pentagon gives
\begin{equation}\label{eq:PE-as-tetrahedra}
\PEprod{\PEref{v_4^-}{u_6}{u_5}{v_0}}
       {\PEref{v_1^-}{u_0}{u_2}{v_3}}
       {\PEref{v_7^-}{u_5}{u_2}{v_8}}
       {\PEref{v_9^-}{u_6}{u_8}{v_3}}
       {\PEref{v_4^-}{u_9}{u_7}{v_1}}.
\end{equation}
In each factor, move the central corner once by \eqref{R1}. The first left
factor illustrates the move: its centre-to-right edge $i_1^*$ becomes a
right-to-centre edge $i_1$, and its central label changes from $v_4^-$ to
$v_4$. All other corners stay fixed:
\[
 \PEtet[\node at (1.53,.52) {$i_1^*$};]
       {A/D,A/C,A/B,D/B,C/B,D/C}{As/{v_4^-},Bo/{u_6},Co/{u_5},Do/{v_0}}
 \;=\;
 \PEtet[\node at (1.53,.52) {$i_1$};]
       {D/A,A/C,A/B,D/B,C/B,D/C}{Ane/{v_4},Bo/{u_6},Co/{u_5},Do/{v_0}}.
\]
Now rotate the five tetrahedra by the spherical symmetries of
Proposition~\ref{prop:A4Sym}. Reorder the two left factors as $(2,1)$ and the
three right factors as $(2,3,1)$. Rotating the arrows and marked corners with each tetrahedron gives
\[
\PEprod{\PEtet{D/C,C/B,C/A,D/A,B/A,D/B}{Co/{v_1},As/{u_0},Blo/{u_2},Dlo/{v_3}}}
       {\PEtet{A/C,C/D,C/B,A/B,D/B,A/D}{Cr/{v_4},Blo/{u_6},Dlo/{u_5},As/{v_0}}}
       {\PEtet{B/C,C/A,C/D,B/D,A/D,B/A}{Cl/{v_9},Dlo/{u_6},As/{u_8},Blo/{v_3}}}
       {\PEtet[\coordinate (Bhi) at (.30,.30);]{B/D,D/C,D/A,B/A,C/A,B/C}{Do/{v_4},Anw/{u_9},Cl/{u_7},Bhi/{v_1}}}
       {\PEtet{C/A,A/B,A/D,C/D,B/D,C/B}{Anw/{v_7},Do/{u_5},Bo/{u_2},Co/{v_8}}}.
\]
\smallskip
\noindent\emph{Match the corners and the paired bases.}
Apply \eqref{R1}--\eqref{R2} to place the labels at the corners of Theorem~\ref{thm:sTPE2}. With $a=\id$, the remaining edge reversals have the form
\[
 \raisebox{-.4cm}{%
\begin{tikzpicture}[scale=.5]
\draw[tpe forward] (0,0)--(2,2) node[midway,above left] {$i$};
\end{tikzpicture}}
=
\raisebox{-.4cm}{%
\begin{tikzpicture}[scale=.5]
\draw[tpe reverse] (0,0)--(2,2) node[midway,above left] {$i^*$};
\end{tikzpicture}}\,.
\]

The fourth and fifth external vertices illustrate the corner convention: $\rho^{-1}(\widetilde\mu)$ has ordered legs $(A^*,C,B^*)$ and $\rho({}'\widetilde\mu)$ has ordered legs $(B,C^*,A)$. Taking $\mu=\mu_6$ and $\mu=\mu_4$, respectively, gives the entries $\alpha_4$ and $\alpha_5$ of the table. Figure~\ref{fig:PEconfig} gives the corresponding prism labels.

At an internal cut, the basis and its dual change together. Indeed, evaluating $(f^{-1})^*\circ\eta$ against $f\circ\theta$ gives the original pairing, for $f:Z\xrightarrow{\sim}Z$:
\tikzset{pe pairing/.style={baseline=(current bounding box.center),x=1pt,y=1pt},
         pe coupon/.style={draw,fill=white,inner sep=2pt}}
\newcommand{\pepairingcap}{%
  \draw (0,0)--(0,-50) .. controls (0,-66) and (42,-66) .. (42,-50)--(42,0);
  \node[pe coupon] (eta) at (0,0) {$\eta$};
  \node[pe coupon] (theta) at (42,0) {$\theta$};}
\[
\begin{tikzpicture}[pe pairing]
 \pepairingcap
 \node[pe coupon] at (0,-26) {$(f^{-1})^*$};
 \node[pe coupon] at (42,-26) {$f$};
 \node[left,font=\scriptsize] at (0,-45) {$Z^*$};
 \node[right,font=\scriptsize] at (42,-45) {$Z$};
\end{tikzpicture}
\;=\;
\begin{tikzpicture}[pe pairing]
 \pepairingcap
 \node[pe coupon,anchor=north] (f) at ([yshift=-6pt]theta.south) {$f$};
 \node[pe coupon,anchor=north] at ([yshift=-6pt]f.south) {$f^{-1}$};
\end{tikzpicture}
\;=\;
\begin{tikzpicture}[pe pairing]
 \pepairingcap
\end{tikzpicture}.
\]
For the rotations in the table one uses the cyclic pairing identity from the proof of Theorem~\ref{thm:sTPE2}:
\[
 b_{B\otimes C\otimes A}\bigl(\rho^{-1}(u),\rho(v)\bigr)=b_{A\otimes B\otimes C}(u,v).
\]

The resulting marked tetrahedra are
\[
\PEprod{\PEtet{B/D,D/C,C/B,B/A,C/A,D/A}{As/{u_0},Bo/{u_2^-},Co/{v_1},Do/{v_3^+}}}
       {\PEtet{D/B,C/D,B/C,A/B,A/C,A/D}{As/{v_0},Bo/{u_6^-},Co/{v_4^+},Do/{u_5^+}}}
       {\PEtet[\node at (1.23,1.04) {$i_0$};]{B/D,D/C,B/C,A/B,C/A,A/D}{As/{u_8},Bhi/{v_3^+},Co/{v_9^-},Dhi/{u_6^-}}}
       {\PEtet[\node at (1.23,1.04) {$i_0$};]{B/D,D/C,B/C,A/B,C/A,A/D}{As/{u_9^+},Bhi/{v_1},Co/{u_7^-},Dhi/{v_4^+}}}
       {\PEtet[\node at (1.23,1.04) {$i_0$};]{B/D,D/C,B/C,A/B,C/A,A/D}{As/{v_7^+},Bhi/{u_2^-},Co/{v_8},Dhi/{u_5^+}}}.
\]
The centre of the first left factor is $u_0=\beta'_0$, that of the second is
$v_0=\beta_0$. On the right, the three top labels are
$(v_9^-,u_7^-,v_8)=(\beta_1,\beta_2,\beta_3)$, and the three centres are
$(u_8,u_9^+,v_7^+)=(\beta'_3,\beta'_1,\beta'_2)$. This is exactly the
factor order and pairing in Theorem~\ref{thm:sTPE2}.
\endgroup

Multilinearity gives the identity for arbitrary vertex morphisms on simple edges. Resolving the identity on each external edge into projections and inclusions of simple summands then gives the identity for arbitrary edge objects. Conversely, the bends, rotations and paired basis changes are invertible, and the simple categorical dimensions are nonzero. Reading the calculation backwards recovers \eqref{Equ: Pentagon Equation}.
\end{proof}
\begin{figure}[ht]
\centering
\begingroup
\newcommand{\pentprismvertices}{%
 \coordinate (P) at (0,2.4); \coordinate (Q) at (1.12,1.2);
 \coordinate (R) at (0,0); \coordinate (S) at (4.2,2.4);
 \coordinate (T) at (3.08,1.2); \coordinate (U) at (4.2,0);}
\newcommand{\pentprismedge}[5]{%
 \draw[#1] (#2)--(#3) node[midway,#4,inner sep=2.5pt] {$#5$};}
\hfill
\begin{tikzpicture}[baseline=(R),x=1.2cm,y=1.2cm,font=\scriptsize]
\pentprismvertices
\pentprismedge{tpe forward}{P}{S}{above}{X_1}
\pentprismedge{tpe forward}{Q}{T}{above}{X_2}
\pentprismedge{tpe forward}{R}{U}{above}{X_3}
\pentprismedge{tpe forward}{P}{Q}{above right}{X_4}
\pentprismedge{tpe forward}{Q}{R}{below right}{X_5}
\pentprismedge{tpe forward}{R}{P}{left}{X_6}
\pentprismedge{tpe forward}{S}{T}{above left}{X_7}
\pentprismedge{tpe forward}{T}{U}{below left}{X_8}
\pentprismedge{tpe forward}{U}{S}{right}{X_9}
\node at (.20,1.98) {$\alpha_1$};
\node at (.78,1.20) {$\alpha_2$};
\node at (.20,.42) {$\alpha_3$};
\node at (4.00,1.98) {$\alpha_4$};
\node at (3.42,1.20) {$\alpha_5$};
\node at (4.00,.42) {$\alpha_6$};
\end{tikzpicture}
\hfill
\begin{tikzpicture}[baseline=(R),x=1.2cm,y=1.2cm,font=\scriptsize]
\pentprismvertices
\pentprismedge{tpe reverse}{P}{S}{above}{i_1^*}
\pentprismedge{tpe forward}{Q}{T}{above}{i_2}
\pentprismedge{tpe reverse}{R}{U}{above}{i_3}
\pentprismedge{tpe reverse}{P}{Q}{above right}{i_4^*}
\pentprismedge{tpe forward}{Q}{R}{below right}{i_5^*}
\pentprismedge{tpe forward}{R}{P}{left}{i_6}
\pentprismedge{tpe reverse}{S}{T}{above left}{i_7^*}
\pentprismedge{tpe reverse}{T}{U}{below left}{i_8^*}
\pentprismedge{tpe reverse}{U}{S}{right}{i_9}
\node at (.21,1.94) {$ {}'\widetilde\mu_1$};
\node at (1.25,.97) {$\widetilde\mu_2$};
\node[left,inner sep=2pt] at (R) {$ {}'\widetilde\mu_3$};
\node at (3.70,2.23) {$ {}'\widetilde\mu_4$};
\node at (2.95,.97) {$\widetilde\mu_5$};
\node[right,inner sep=2pt] at (U) {$\widetilde\mu_6$};
\end{tikzpicture}
\hfill\mbox{}
\endgroup
\caption{Left, the edge and vertex labels of the triangular prism equation (Figure~\ref{fig:TPEconfig}). Right, the simple indices and bent basis morphisms of the pentagon equation \eqref{Equ: Pentagon Equation}. To pass between them, cyclically rotate each triangular end and apply the changes of marked corner in Theorem~\ref{thm:PE-TPE}. For instance, $\alpha_1$ at the upper left corresponds to $ {}'\widetilde\mu_3$ at the lower left by $\alpha_1=\rho({}'\widetilde\mu_3)$.}
\label{fig:PEconfig}
\end{figure}

Theorem~\ref{thm:PE-TPE} compares two expressions for associativity in a given spherical fusion category. Reconstruction from abstract scalars also requires invertible associativity matrices, unit normalization and compatible spherical duality data. The spherical localization arguments use these equations as necessary conditions. Their dimension factors $d_i$ are categorical dimensions. They define a ring character \cite[Proposition~4.7.12]{EGNO15} and satisfy $d_i=d_{i^*}$ by sphericality.

\section{Localization}

\subsection{Localization strategy} \label{sec:loc}
Tetrahedral rotations identify coefficients, and rotations of the triangular prism identify equations among them. The groups are $A_4$ and $D_3\cong S_3$, respectively. A rotation moves the orientations, dual labels and marked corners along with the graph. It therefore identifies decorated evaluations for which the ordinary pentagon would require explicit changes of basis.

The following pictures give a geometric interpretation of localization. Resolving $A\otimes B$ and $C\otimes D$ into simples gives a matrix-unit basis of $\hc(A\otimes B,C\otimes D)$. Its diagrams have the shape of an I. A quarter-turn turns the I into an H and gives the \emph{string Fourier transform} \cite{JafLiu18}. For $f:A\otimes B\to C\otimes D$ it is the rigidity isomorphism
\[
 \hc(A\otimes B,C\otimes D)\simeq\hc(C^*\otimes A,D\otimes B^*),
 \qquad f\longmapsto\mathcal F(f)=
\vcenter{\hbox{\begin{tikzpicture}[scale=.51,font=\scriptsize]
\draw (-1,1.6) node[above] {$C^*$} -- (-1,-.8)
  .. controls (-1,-1.35) and (.2,-1.35) .. (.2,-.8) -- (.2,-.25);
\draw (.2,1.6) node[above] {$A$} -- (.2,.25);
\draw (1.4,.25) -- (1.4,.85)
  .. controls (1.4,1.35) and (2.7,1.35) .. (2.7,.85)
  -- (2.7,-1.6) node[below] {$B^*$};
\draw (1.4,-.25) -- (1.4,-1.6) node[below] {$D$};
\node[left] at (1.4,.65) {$B$};
\node[right] at (.2,-.65) {$C$};
\draw[fill=white] (-.3,-.32) rectangle (1.9,.32);
\node at (.8,0) {$f$};
\end{tikzpicture}}}.
\]
Here $\mathcal F(f)=(\ev_C\otimes\id_{D\otimes B^*})\circ
(\id_{C^*}\otimes f\otimes\id_{B^*})\circ
(\id_{C^*\otimes A}\otimes\coev_B)$. Suppressing the vertex bases and duality decorations of \S\S\ref{sec:pre}--\ref{sec:pre2}, we obtain
\[
\begin{tikzpicture}[baseline=-.1cm,scale=.65]
\draw (-.6,1)--(0,.45)--(.6,1);
\draw (0,.45)--(0,-.45) node[midway,right] {$a$};
\draw (-.6,-1)--(0,-.45)--(.6,-1);
\fill (0,.45) circle(2pt);\fill(0,-.45) circle(2pt);
\end{tikzpicture}
\quad\xrightarrow{\text{quarter-turn}}\quad
\begin{tikzpicture}[baseline=-.1cm,scale=.65]
\draw (-1,.6)--(-.45,0)--(-1,-.6);
\draw (-.45,0)--(.45,0) node[midway,above] {$a$};
\draw (1,.6)--(.45,0)--(1,-.6);
\fill (-.45,0) circle(2pt);\fill(.45,0) circle(2pt);
\end{tikzpicture}
\]
Closing a pair of such diagrams gives a tetrahedron. Closing three gives a triangular prism. The dashed lines mark the cuts.
\[
\underbrace{\begin{tikzpicture}[baseline=(current bounding box.center),x=.61cm,y=.56cm,font=\scriptsize,inner sep=1pt]
\draw (0,1.6)--(0,.6) node[midway,right=2pt] {$a$};
\draw (0,.6)--(-.8,-.1)--(-.8,-.8);
\draw (0,.6)--(.8,-.1)--(.8,-.8);
\draw (-.8,-.8)--(.8,-.8) node[midway,below=2pt] {$b$};
\draw (0,1.6)--(-.8,2.3) .. controls (-2,3) and (-2,-2.2) .. (-.8,-1.5)--(-.8,-.8);
\draw (0,1.6)--(.8,2.3) .. controls (2,3) and (2,-2.2) .. (.8,-1.5)--(.8,-.8);
\draw[densely dashed,gray] (-1.91,-.1)--(1.91,-.1);
\foreach \p in {(0,1.6),(0,.6),(-.8,-.8),(.8,-.8)} \fill \p circle(1.4pt);
\end{tikzpicture}}_{\text{closure of an I/H pair}}
\;=\;
\begin{tikzpicture}[baseline=(current bounding box.center),x=.82cm,y=.66cm,font=\scriptsize,inner sep=1pt]
\draw (0,1.8)--(-1.35,-1)--(1.35,-1)--cycle;
\draw (0,1.8)--(0,-.18) node[pos=.72,right=2pt] {$a$};
\draw (0,-.18)--(-1.35,-1);\draw (0,-.18)--(1.35,-1);
\node[below=2pt] at (0,-1) {$b$};
\foreach \p in {(0,1.8),(-1.35,-1),(1.35,-1),(0,-.18)} \fill \p circle(1.4pt);
\end{tikzpicture},
\qquad
\underbrace{\begin{tikzpicture}[baseline=(current bounding box.center),x=.62cm,y=.62cm,font=\scriptsize,inner sep=1pt]
\draw (-.8,1.5)--(-.8,-1.5);\draw (.8,1.5)--(.8,-1.5);
\draw (-.8,1.5) .. controls (-2,2.8) and (-2,-2.8) .. (-.8,-1.5);
\draw (.8,1.5) .. controls (2,2.8) and (2,-2.8) .. (.8,-1.5);
\draw (-.8,1.5)--(.8,1.5) node[midway,above=2pt] {$a$};
\draw (-.8,0)--(.8,0) node[midway,above=2pt] {$b$};
\draw (-.8,-1.5)--(.8,-1.5) node[midway,below=2pt] {$c$};
\draw[densely dashed,gray] (-1.91,.80)--(1.91,.80);
\draw[densely dashed,gray] (-1.91,-.80)--(1.91,-.80);
\foreach \x in {-.8,.8} \foreach \y in {-1.5,0,1.5} \fill (\x,\y) circle(1.4pt);
\end{tikzpicture}}_{\text{closure of three H-diagrams}}
\;=\;
\begin{tikzpicture}[baseline=(current bounding box.center),x=.82cm,y=.66cm,font=\scriptsize,inner sep=1pt]
\draw (-1.27,1.5)--(-.54,0)--(-1.27,-1.5)--cycle;
\draw (1.27,1.5)--(.54,0)--(1.27,-1.5)--cycle;
\draw (-1.27,1.5)--(1.27,1.5) node[midway,above=2pt] {$a$};
\draw (-.54,0)--(.54,0) node[midway,above=2pt] {$b$};
\draw (-1.27,-1.5)--(1.27,-1.5) node[midway,below=2pt] {$c$};
\foreach \p in {(-1.27,1.5),(-.54,0),(-1.27,-1.5),(1.27,1.5),(.54,0),(1.27,-1.5)} \fill \p circle(1.4pt);
\end{tikzpicture}.
\]
With the trace normalizations above, the tetrahedral coefficients are pairings of the quarter-turned basis vectors with the trace-normalized basis. Evaluating the closed prism before and after resolving its cuts gives the two sides of TPE. Repeating edge labels lets us group equations involving the same coefficients.

For a quantitative picture, suppose $X$ is self-dual and $X\otimes X$ is
multiplicity-free, with $n$ simple summands. Then
$\dim_{\field}\End_{\mC}(X\otimes X)=n$. After the boundary identifications have
been fixed, its quarter-turn map has $n^2$ matrix entries. Labelling the six edges of the two triangular ends by $X$ leaves $n^3$ possible triples of simple labels on the three connecting edges. At the level of these triples, cyclic permutation has
$(n^3+2n)/3$ orbits. This count describes the available configurations;
it does not assert that all resulting equations are independent, or that
their other tetrahedral contractions are already known. Unit labels and
inadmissible vertices give explicit values for some contractions, while the
others are retained as variables. The functions $x$ and $y$ in
Theorem~\ref{thm:loc} record precisely these contractions.

The general localization procedure is iterative:
\begin{enumerate}
\item Choose a measure of complexity for the unknown coefficients. For
example, count the distinct simple labels, up to duality, occurring in a
tetrahedron, and assign the lowest complexity to coefficients already known.
\item Select a small collection of coefficients and equations involving
them. Generate these collections incrementally; repeated-label prisms and
their rotations supply systematic choices.
\item Compute the consequences of the selected subsystem, using linear or
binomial elimination when available and a Gr\"obner basis when needed.
\item Substitute the resulting relations into further subsystems and add
compatibility equations between overlapping choices. Retain every solution
branch not ruled out by a necessary condition. An inconsistent subsystem
already obstructs categorification; otherwise the process continues.
\end{enumerate}
Useful subsystems can be much smaller than the full associator system. The
multiplicity-free calculations in \cite{gvfuscat}, for example, exploit
binomial equations, equations linear in selected variables, and low-degree
relations. The choice of subsystem remains important: a subsystem with many
algebraic solution branches may be expensive to combine with the rest.
The geometric construction supplies several families of related subsystems
that can be compared, without assuming a uniform complexity bound for their
solution. For the near-group fusion rules $G+|G|$, the two cuts give a sum of products of two coefficients on one side of the essential equation. He, Liu, Lu, Palcoux and Ren use this form to construct solutions for cyclic $G$ under the spherical and indicator hypotheses recalled in the introduction \cite[\S2]{LLPRnear}.

In the specializations used below, the first prism of Theorem~\ref{thm:loc} has $k$ on all six edges of the triangular ends. Give the connecting edges labels $(a,b,1)$ or $(a,b,b)$. The second prism has $(b,k,k)$ on each end and $(k,k,a)$ on the connecting edges. On cutting these prisms, the same single tetrahedra and the same pairs recur. The braces in~\eqref{equ:I}--\eqref{equ:II} show which evaluations to record as $y$ and which contractions to record as $x$. Corollary~\ref{cor:loc} restricts $a,b$ to $S'_k$ while retaining the full internal sum over $S_k$. To compare two such choices for $\mathcal F_{210}$, Theorem~\ref{thm:locextra} uses a prism whose nine edges carry just $k$ and $l$.

The small-spectrum criteria in \S\ref{sec:SpeCrit} have a different scope. They are derived directly from the pentagon equations and do not assume a pivotal or spherical structure. Thus the exclusion of $\mathcal F_{660}$ does not depend on the spherical reformulation.

\subsection{Categorification criteria from localization} \label{sec:FirstLoc}

We use Notation~\ref{not:1}. All categorical dimensions below belong to the algebraically closed ground field $\field$; the dimensions of simple objects are nonzero. The equations are necessary conditions for spherical categorification.

\begin{theorem} \label{thm:loc}
Let $\mC$ be a spherical fusion category. Consider $X_k$, a self-dual simple object within $\mC$, such that for every simple object $X_a$, the condition $N_{k,k}^{a} \le 1$ holds. Assume that if $N_{k,k}^{a} = 1$, then $X_a$ is self-dual. Define $S_k$ as the set of such indices $a$.
Assume further that there exists $b$ such that $N_{b,b^*}^k$ is odd.
There exist functions $x:S_k^3\to\field$ and $y:S_k^2\to\field$ satisfying, for every $a,b,c\in S_k$, the following equations:
\begin{align}
\label{equ:Ibis} x(a,b,c)=  & \sum_{i \in S_k} d_i y(i,a)y(i,b)y(i,c), \\
\label{equ:IIbis} y(a,b)y(a,c) = & \sum_{i \in S_k} d_i y(i,a)x(i,b,c),
\end{align}
\noindent with $x(a,k,k) = y(a,k)^2$ whenever $k\in S_k$, $x(a,b,c) = 0$ if $N_{b,c}^a  =0$, $y(a,b) = y(b,a)$, $y(1,b) = d_k^{-1}$, $x(1,b,c) = \delta_{b,c} (d_bd_k)^{-1}$.
\end{theorem}
\begin{proof}
Work in the strictly pivotal realization of Remark~\ref{rk:StrPiv}.
By Theorem~\ref{thm:FrobSchur}, every $X_i$ with $i\in S_k$ has
indicator one (its multiplicity in $X_k\otimes X_k^*$ is one);
the odd-multiplicity hypothesis gives the same conclusion for $X_k$.
Fix symmetric self-duality maps $\kappa_i$ as in Remark~\ref{rk:simplified},
with $\kappa_1=\id_{\one}$, and use its unarrowed calculus throughout.

For each $i\in S_k$, choose
$\theta_i\in\hc(\one,X_k\otimes X_i\otimes X_k)$ with
$\omega(\theta_i,\theta_i)=1$, using the pairing of
Lemma~\ref{prop:binu} under the fixed self-dualities. Its cyclic transports
at the marked corners are denoted by $\bullet$. For $i,b,c\in S_k$, let $\beta$
run through a basis $B$ of $\hc(\one,X_b\otimes X_i\otimes X_c)$,
with dual vector $\beta'$. Define
\[
\underbrace{\sum_{\beta\in B}\vcenter{\hbox{\scalebox{.77}{\begin{tikzpicture}[scale=1.5]
\draw (0,0)  -- (1,0) node [below] {$k$}; \draw (1,0) -- (2,0);
\draw (2,0) -- (1.5,.866) node [right] {$c$}; \draw (1.5,.866) -- (1,1.732);
\draw (1,1.732) node [above] {$\beta$} -- (.5,.866) node [left] {$b$}; \draw (.5,.866) -- (0,0);
\draw (0,0) -- (.75,.433) node [above] {$k$}; \draw (.75,.433) -- (1,0.577);
\draw (1,0.577) node [below] {$\bullet$} -- (1,1.1545) node [left] {$i$}; \draw (1,1.1545) -- (1,1.732);
\draw (1,0.577) -- (1.25,.433) node [above] {$k$}; \draw (1.25,.433) -- (2,0);
\node at (.36,.35) {$\bullet$};
\node at (2-.36,.35) {$\bullet$};
\end{tikzpicture}}}}\,\vcenter{\hbox{\scalebox{.77}{\begin{tikzpicture}[scale=1.5]
\draw (0,0) -- (1,0) node [below] {$k$}; \draw (1,0) -- (2,0);
\draw (2,0) -- (1.5,.866) node [right] {$k$}; \draw (1.5,.866) -- (1,1.732);
\draw (1,1.732) node [above] {$\bullet$} -- (.5,.866) node [left] {$k$}; \draw (.5,.866) -- (0,0);
\draw (0,0) -- (.75,.433) node [above] {$b$}; \draw (.75,.433) -- (1,0.577);
\draw (1,0.577) node [below] {$\beta'$} -- (1,1.1545) node [left] {$i$}; \draw (1,1.1545) -- (1,1.732);
\draw (1,0.577) -- (1.25,.433) node [above] {$c$}; \draw (1.25,.433) -- (2,0);
\node at (.36,.35) {$\bullet$};
\node at (2-.36,.35) {$\bullet$};
\end{tikzpicture}}}}}_{x(i,b,c)}\,,\qquad\underbrace{\vcenter{\hbox{\scalebox{.77}{\begin{tikzpicture}[scale=1.5]
\draw (0,0) -- (1,0) node [below] {$b$}; \draw (1,0) -- (2,0);
\draw (2,0) -- (1.5,.866) node [right] {$k$}; \draw (1.5,.866) -- (1,1.732);
\draw (1,1.732) node [above] {$\bullet$} -- (.5,.866) node [left] {$k$}; \draw (.5,.866) -- (0,0);
\draw (0,0) -- (.75,.433) node [above] {$k$}; \draw (.75,.433) -- (1,0.577);
\draw (1,0.577) node [below] {$\bullet$} -- (1,1.1545) node [left] {$i$}; \draw (1,1.1545) -- (1,1.732);
\draw (1,0.577) -- (1.25,.433) node [above] {$k$}; \draw (1.25,.433) -- (2,0);
\node at (.36,.35) {$\bullet$};
\node at (2-.36,.35) {$\bullet$};
\end{tikzpicture}}}}}_{y(i,b)}\,.
\]
Apply Theorem~\ref{thm:sTPE2} to the following prisms:
\[
\underbrace{\vcenter{\hbox{\scalebox{.74}{\begin{tikzpicture}[scale=2]
\draw (0,2)--(3,2);
\draw (1,1)--(2,1);
\draw (0,0)--(3,0);
\draw (0,0) -- (0,2) -- (1,1) -- (0,0);
\draw (3,0)--(3,2)--(2,1)--(3,0);
\draw (0,2) -- (1.5,2) node [above] {$a$};
\draw (1,1) -- (1.5,1) node [above] {$b$};
\draw (0,0) -- (1.5,0) node [above] {$c$};
\draw (0,2) --++ (.5,-.5) node [above right=1pt] {$k$};
\draw (1,1) --++ (-.5,-.5) node [below right=1pt] {$k$};
\draw (0,0) -- (0,1) node [left] {$k$};
\draw (3,2) --++ (-.5,-.5) node [above left=1pt] {$k$};
\draw (2,1) --++ (.5,-.5) node [below left=1pt] {$k$};
\draw (3,0) -- (3,1) node [right] {$k$};
\node at (0+.15,2-.3) {$\bullet$};
\node at (1-.2,1) {$\bullet$};
\node at (0+.15,0+.3) {$\bullet$};
\node at (3-.15,2-.3) {$\bullet$};
\node at (2+.2,1) {$\bullet$};
\node at (3-.15,0+.3) {$\bullet$};
\end{tikzpicture}}}}}_{\eqref{equ:I}}\qquad\qquad\underbrace{\vcenter{\hbox{\scalebox{.74}{\begin{tikzpicture}[scale=2]
\draw (0,2)--(3,2);
\draw (1,1)--(2,1);
\draw (0,0)--(3,0);
\draw (0,0) -- (0,2) -- (1,1) -- (0,0);
\draw (3,0)--(3,2)--(2,1)--(3,0);
\draw (0,2) -- (1.5,2) node [above] {$k$};
\draw (1,1) -- (1.5,1) node [above] {$k$};
\draw (0,0) -- (1.5,0) node [above] {$a$};
\draw (0,2) --++ (.5,-.5) node [above right=1pt] {$b$};
\draw (1,1) --++ (-.5,-.5) node [below right=1pt] {$k$};
\draw (0,0) -- (0,1) node [left] {$k$};
\draw (3,2) --++ (-.5,-.5) node [above left=1pt] {$c$};
\draw (2,1) --++ (.5,-.5) node [below left=1pt] {$k$};
\draw (3,0) -- (3,1) node [right] {$k$};
\node at (0+.15,2-.3) {$\bullet$};
\node at (1-.2,1) {$\bullet$};
\node at (0+.15,0+.3) {$\bullet$};
\node at (3-.15,2-.3) {$\bullet$};
\node at (2+.2,1) {$\bullet$};
\node at (3-.15,0+.3) {$\bullet$};
\end{tikzpicture}}}}}_{\eqref{equ:II}}\,.
\]
The three side cuts in the first prism have boundary object $X_k^{\otimes2}$. In the
second they have boundary objects $X_b\otimes X_c$ and two copies of $X_k^{\otimes2}$.
Thus we can sum over $S_k$ in both equations: where the corresponding simple summand is absent
from $X_b\otimes X_c$, the basis $B$ is empty.
\begin{equation}\label{equ:I}
\underbrace{\sum_{\beta_0\in B_0}\vcenter{\hbox{\scalebox{.75}{\begin{tikzpicture}[scale=1.5]
\draw (0,0) node [below] {$\bullet$} -- (1,0)  node [below] 	{$k$}; 		\draw (1,0) -- (2,0);
\draw (2,0) node [below] {$\bullet$} -- (1.5,.866) node [right] {$k$}; 		\draw (1.5,.866)--(1,1.732);
\draw (1,1.732)	node [above] {$\bullet$} -- (.5,.866) 	node [left] {$k$}; 	\draw (.5,.866) -- (0,0);
\draw (0,0) -- (.5,.2885) 	node [above] {$c$}; 							\draw (.5,.2885) -- (1,0.577);
\draw (1,0.577) node [below] {$\beta_0'$} -- (1,1.1545) node [right] {$b$}; 	\draw (1,1.1545)-- (1,1.732);
\draw (1,0.577) -- (1.5,.2885) 	node [above] {$a$}; 						\draw (1.5,.2885) -- (2,0);
\end{tikzpicture}}}}\,\vcenter{\hbox{\scalebox{.75}{\begin{tikzpicture}[scale=1.5]
\draw (0,0) node [below] {$\bullet$} -- (1,0) node [below] {$k$}; 			\draw (1,0) -- (2,0);
\draw (2,0) node [below] {$\bullet$} -- (1.5,.866) 	node [right] {$k$};		\draw (1.5,.866) -- (1,1.732);
\draw (1,1.732)	node [above] {$\bullet$} -- (.5,.866) 	node [left] {$k$};	\draw (.5,.866) -- (0,0);
\draw (0,0) -- (.5,.2885) node [above] {$a$};								\draw (.5,.2885) 	-- (1,0.577);
\draw (1,0.577) node [below] {$\beta_0$} -- (1,1.1545) 	node [right] {$b$};	\draw (1,1.1545) -- (1,1.732);
\draw (1,0.577) -- (1.5,.2885) 	node [above] {$c$};							\draw (1.5,.2885) -- (2,0);
\end{tikzpicture}}}}}_{x(a,b,c)}
=\sum_{i\in S_k}d_i\,\underbrace{\vcenter{\hbox{\scalebox{.75}{\begin{tikzpicture}[scale=1.5]
\draw (0,0)  	-- (1,0) 		node [below] 	{$a$}; \draw (1,0) 		-- (2,0);
\draw (2,0) 	-- (1.5,.866) 	node [right] 	{$k$}; \draw (1.5,.866) 	-- (1,1.732);
\draw (1,1.732)	node [above] {$\bullet$} 	-- (.5,.866) 	node [left] 	{$k$}; \draw (.5,.866) 	-- (0,0);
\draw (0,0) 	-- (.75,.433) 	node [above] 	{$k$}; \draw (.75,.433) -- (1,0.577);
\draw (1,0.577) node [below] {$\bullet$} -- (1,1.1545) 	node [left]	{$i$}; \draw (1,1.1545) 	-- (1,1.732);
\draw (1,0.577) -- (1.25,.433) 	node [above] 	{$k$}; \draw (1.25,.433) -- (2,0);
\node at (.36,.35) {$\bullet$};
\node at (2-.36,.35) {$\bullet$};
\end{tikzpicture}}}}}_{y(i,a)}\,\underbrace{\vcenter{\hbox{\scalebox{.75}{\begin{tikzpicture}[scale=1.5]
\draw (0,0) 	-- (1,0) 		node [below] 	{$b$}; \draw (1,0) 		-- (2,0);
\draw (2,0) 	-- (1.5,.866) 	node [right] 	{$k$}; \draw (1.5,.866) 	-- (1,1.732);
\draw (1,1.732)	node [above] {$\bullet$} 	-- (.5,.866) 	node [left] 	{$k$}; \draw (.5,.866) 	-- (0,0);
\draw (0,0) 	-- (.75,.433) 	node [above] 	{$k$}; \draw (.75,.433) -- (1,0.577);
\draw (1,0.577) node [below] {$\bullet$} -- (1,1.1545) 	node [left]	{$i$}; \draw (1,1.1545) 	-- (1,1.732);
\draw (1,0.577) -- (1.25,.433) 	node [above] 	{$k$}; \draw (1.25,.433) -- (2,0);
\node at (.36,.35) {$\bullet$};
\node at (2-.36,.35) {$\bullet$};
\end{tikzpicture}}}}}_{y(i,b)}\,\underbrace{\vcenter{\hbox{\scalebox{.75}{\begin{tikzpicture}[scale=1.5]
\draw (0,0) -- (1,0) 		node [below] 	{$c$}; \draw (1,0) 		-- (2,0);
\draw (2,0) -- (1.5,.866) 	node [right] 	{$k$}; \draw (1.5,.866) 	-- (1,1.732);
\draw (1,1.732)	node [above] {$\bullet$} 	-- (.5,.866) 	node [left] 	{$k$}; \draw (.5,.866) 	-- (0,0);
\draw (0,0) 	-- (.75,.433) 	node [above] 	{$k$}; \draw (.75,.433) -- (1,0.577);
\draw (1,0.577) node [below] {$\bullet$} -- (1,1.1545) 	node [left]	{$i$}; \draw (1,1.1545) 	-- (1,1.732);
\draw (1,0.577) -- (1.25,.433) 	node [above] 	{$k$}; \draw (1.25,.433) -- (2,0);
\node at (.36,.35) {$\bullet$};
\node at (2-.36,.35) {$\bullet$};
\end{tikzpicture}}}}}_{y(i,c)}
\end{equation}
\begin{equation}\label{equ:II}
\underbrace{\vcenter{\hbox{\scalebox{.75}{\begin{tikzpicture}[scale=1.5]
\draw (0,0) node [below] {$\bullet$} -- (1,0) node [below] {$k$}; \draw (1,0) -- (2,0);
\draw (2,0) node [below] {$\bullet$} -- (1.5,.866) node [right] {$b$}; \draw (1.5,.866) -- (1,1.732);
\draw (1,1.732) node [above] {$\bullet$} -- (.5,.866) node [left] {$k$}; \draw (.5,.866) -- (0,0);
\draw (0,0) -- (.5,.2885) node [above] {$a$}; \draw (.5,.2885) -- (1,0.577);
\draw (1,0.577) node [below] {$\bullet$} -- (1,1.1545) node [right] {$k$}; \draw (1,1.1545) -- (1,1.732);
\draw (1,0.577) -- (1.5,.2885) node [above] {$k$}; \draw (1.5,.2885) -- (2,0);
\end{tikzpicture}}}}\,\vcenter{\hbox{\scalebox{.75}{\begin{tikzpicture}[scale=1.5]
\draw (0,0) node [below] {$\bullet$} -- (1,0) node [below] {$k$}; \draw (1,0) -- (2,0);
\draw (2,0) node [below] {$\bullet$} -- (1.5,.866) node [right] {$k$}; \draw (1.5,.866) -- (1,1.732);
\draw (1,1.732) node [above] {$\bullet$} -- (.5,.866) node [left] {$c$}; \draw (.5,.866) -- (0,0);
\draw (0,0) -- (.5,.2885) node [above] {$k$}; \draw (.5,.2885) -- (1,0.577);
\draw (1,0.577) node [below] {$\bullet$} -- (1,1.1545) node [right] {$k$}; \draw (1,1.1545) -- (1,1.732);
\draw (1,0.577) -- (1.5,.2885) node [above] {$a$}; \draw (1.5,.2885) -- (2,0);
\end{tikzpicture}}}}}_{y(a,b)y(a,c)}
=\sum_{i\in S_k}d_i\,\underbrace{\sum_{\beta\in B}\vcenter{\hbox{\scalebox{.75}{\begin{tikzpicture}[scale=1.5]
\draw (0,0)  -- (1,0) node [below] {$k$}; \draw (1,0) -- (2,0);
\draw (2,0) -- (1.5,.866) node [right] {$c$}; \draw (1.5,.866) -- (1,1.732);
\draw (1,1.732) node [above] {$\beta$} -- (.5,.866) node [left] {$b$}; \draw (.5,.866) -- (0,0);
\draw (0,0) -- (.75,.433) node [above] {$k$}; \draw (.75,.433) -- (1,0.577);
\draw (1,0.577) node [below] {$\bullet$} -- (1,1.1545) node [left] {$i$}; \draw (1,1.1545) -- (1,1.732);
\draw (1,0.577) -- (1.25,.433) node [above] {$k$}; \draw (1.25,.433) -- (2,0);
\node at (.36,.35) {$\bullet$};
\node at (2-.36,.35) {$\bullet$};
\end{tikzpicture}}}}\,\vcenter{\hbox{\scalebox{.75}{\begin{tikzpicture}[scale=1.5]
\draw (0,0) -- (1,0) node [below] {$k$}; \draw (1,0) -- (2,0);
\draw (2,0) -- (1.5,.866) node [right] {$k$}; \draw (1.5,.866) -- (1,1.732);
\draw (1,1.732) node [above] {$\bullet$} -- (.5,.866) node [left] {$k$}; \draw (.5,.866) -- (0,0);
\draw (0,0) -- (.75,.433) node [above] {$b$}; \draw (.75,.433) -- (1,0.577);
\draw (1,0.577) node [below] {$\beta'$} -- (1,1.1545) node [left] {$i$}; \draw (1,1.1545) -- (1,1.732);
\draw (1,0.577) -- (1.25,.433) node [above] {$c$}; \draw (1.25,.433) -- (2,0);
\node at (.36,.35) {$\bullet$};
\node at (2-.36,.35) {$\bullet$};
\end{tikzpicture}}}}}_{x(i,b,c)}\,\underbrace{\vcenter{\hbox{\scalebox{.75}{\begin{tikzpicture}[scale=1.5]
\draw (0,0) -- (1,0) node [below] {$a$}; \draw (1,0) -- (2,0);
\draw (2,0) -- (1.5,.866) node [right] {$k$}; \draw (1.5,.866) -- (1,1.732);
\draw (1,1.732) node [above] {$\bullet$} -- (.5,.866) node [left] {$k$}; \draw (.5,.866) -- (0,0);
\draw (0,0) -- (.75,.433) node [above] {$k$}; \draw (.75,.433) -- (1,0.577);
\draw (1,0.577) node [below] {$\bullet$} -- (1,1.1545) node [left] {$i$}; \draw (1,1.1545) -- (1,1.732);
\draw (1,0.577) -- (1.25,.433) node [above] {$k$}; \draw (1.25,.433) -- (2,0);
\node at (.36,.35) {$\bullet$};
\node at (2-.36,.35) {$\bullet$};
\end{tikzpicture}}}}}_{y(i,a)}
\end{equation}

For the first left-hand side, write $T,L,R,O$ for the top, left, right and
interior vertices. Rotate the two factors defining $x(a,b,c)$ on the sphere
and interchange them. Before restoring their marked corners this gives
\[
x(a,b,c)=\sum_{\beta\in B}
\underbrace{\vcenter{\hbox{\scalebox{.89}{\begin{tikzpicture}[scale=1.40,font=\small,inner sep=1.5pt]
\coordinate (L) at (0,0); \coordinate (R) at (2.6,0);
\coordinate (T) at (1.3,2.25); \coordinate (O) at (1.3,.75);
\draw (L)--node[midway,below] {$k$}(R)--node[midway,right=2pt] {$k$}(T)
 --node[midway,left=2pt] {$k$}cycle;
\draw (O)--node[pos=.48,above left=-1pt] {$c$}(L);
\draw (O)--node[pos=.55,right] {$b$}(T);
\draw (O)--node[pos=.48,above right=-1pt] {$a$}(R);

\node at (1.40,1.91) {$\bullet$}; \node[below right] at (R) {$\bullet$};
\node at (.32,.08) {$\bullet$}; \node at (1.05,.91) {$\beta'$};\end{tikzpicture}}}}}_{\text{second factor: }(T\,R\,L)}\qquad\underbrace{\vcenter{\hbox{\scalebox{.89}{\begin{tikzpicture}[scale=1.40,font=\small,inner sep=1.5pt]
\coordinate (L) at (0,0); \coordinate (R) at (2.6,0);
\coordinate (T) at (1.3,2.25); \coordinate (O) at (1.3,.75);
\draw (L)--node[midway,below] {$k$}(R)--node[midway,right=2pt] {$k$}(T)
 --node[midway,left=2pt] {$k$}cycle;
\draw (O)--node[pos=.48,above left=-1pt] {$a$}(L);
\draw (O)--node[pos=.55,right] {$b$}(T);
\draw (O)--node[pos=.48,above right=-1pt] {$c$}(R);

\node at (1.20,1.91) {$\bullet$}; \node[below left] at (L) {$\bullet$};
\node at (2.28,.08) {$\bullet$}; \node at (1.54,.91) {$\beta$};\end{tikzpicture}}}}}_{\text{first factor: }(T\,O\,L)}\,.
\]
Moving the central marks to the lower face by~\eqref{R1} gives
$\beta'_0=\rho(\beta')$ and $\beta_0=\rho^{-1}(\beta)$,
a pair of vectors dual under the evaluation pairing. By Lemma~\ref{lem:rot} the remaining
corner factors multiply to
\[
 \underbrace{\omega_k^{\delta_{k,b}}\omega_k^{2\delta_{k,b}}}_{1}
 \;\underbrace{\omega_k^{2\delta_{k,c}}\omega_k^{\delta_{k,c}}}_{1}
 \qquad (\omega_k^3=1\text{ if }k\in S_k;
             \ \omega_k:=1\text{ otherwise}).
\]
For the left side of~\eqref{equ:II}, rotate $y(a,b)$ by $(T\,O\,L)$
and $y(a,c)$ by $(T\,O\,R)$. The two $b$-corner shifts are opposite,
as are the two $c$-shifts; the $a$-shifts cancel between the factors.
The underbraces now give~\eqref{equ:Ibis} and~\eqref{equ:IIbis}.

The half-turn below exchanges $TO$ and $LR$, with their marked corners:
\[
y(a,b)=
\begin{tikzpicture}[baseline=.6cm,font=\scriptsize,scale=1.05]
\draw (0,0)--(1,0) node[below] {$b$} --(2,0);
\draw (0,0)--(.5,.866) node[left] {$k$} --(1,1.732);
\draw (2,0)--(1.5,.866) node[right] {$k$} --(1,1.732);
\draw (0,0)--(.5,.2885) node[above] {$k$} --(1,.577);
\draw (2,0)--(1.5,.2885) node[above] {$k$} --(1,.577);
\draw (1,.577)--(1,1.1545) node[right] {$a$} --(1,1.732);
\node[above] at (1,1.732) {$\bullet$};
\node[below] at (1,.577) {$\bullet$};
\node at (.36,.35) {$\bullet$};\node at (1.64,.35) {$\bullet$};
\end{tikzpicture}
\ \underset{(T\,L)(O\,R)}{=}\ 
\begin{tikzpicture}[baseline=.6cm,font=\scriptsize,scale=1.05]
\draw (0,0)--(1,0) node[below] {$a$} --(2,0);
\draw (0,0)--(.5,.866) node[left] {$k$} --(1,1.732);
\draw (2,0)--(1.5,.866) node[right] {$k$} --(1,1.732);
\draw (0,0)--(.5,.2885) node[above] {$k$} --(1,.577);
\draw (2,0)--(1.5,.2885) node[above] {$k$} --(1,.577);
\draw (1,.577)--(1,1.1545) node[right] {$b$} --(1,1.732);
\node[above] at (1,1.732) {$\bullet$};
\node[below] at (1,.577) {$\bullet$};
\node at (.36,.35) {$\bullet$};\node at (1.64,.35) {$\bullet$};
\end{tikzpicture}
=y(b,a).
\]
If $k\in S_k$, then $N_{k,k}^a=1$: in the definition of $x(a,k,k)$ the single pair is two copies of $y(a,k)$, with the same vector and its dual. Thus $x(a,k,k)=y(a,k)^2$.

Finally, an inadmissible vertex gives $x(a,b,c)=0$ when $N_{b,c}^a=0$.
For the unit normalizations, the unit edges are dashed below; the paired
unit vertices cancel by the dual-basis and zigzag identities:
\[
y(1,b)=
\begin{tikzpicture}[baseline=.6cm,font=\scriptsize,scale=.95]
\draw (0,0)--(1,0) node[below] {$b$} --(2,0);
\draw (0,0)--(.5,.866) node[left] {$k$} --(1,1.732);
\draw (2,0)--(1.5,.866) node[right] {$k$} --(1,1.732);
\draw (0,0)--(.5,.2885) node[above] {$k$} --(1,.577);
\draw (2,0)--(1.5,.2885) node[above] {$k$} --(1,.577);
\draw[dashed] (1,.577)--(1,1.035) node[right] {$\one$} --(1,1.732);
\node[above] at (1,1.732) {$\bullet$};
\node[below] at (1,.577) {$\bullet$};
\node at (.36,.35) {$\bullet$};\node at (1.64,.35) {$\bullet$};
\end{tikzpicture}
=\frac1{d_k},
\qquad
x(1,b,c)=\sum_{\beta\in B}
\begin{tikzpicture}[baseline=.6cm,font=\scriptsize,scale=.95]
\draw (0,0)--(1,0) node[below] {$k$} --(2,0);
\draw (0,0)--(.5,.866) node[left] {$b$} --(1,1.732);
\draw (2,0)--(1.5,.866) node[right] {$c$} --(1,1.732);
\draw (0,0)--(.5,.2885) node[above] {$k$} --(1,.577);
\draw (2,0)--(1.5,.2885) node[above] {$k$} --(1,.577);
\draw[dashed] (1,.577)--(1,1.035) node[right] {$\one$} --(1,1.732);
\node[above] at (1,1.732) {$\beta$};
\node[below] at (1,.577) {$\bullet$};
\node at (.36,.35) {$\bullet$};\node at (1.64,.35) {$\bullet$};
\end{tikzpicture}
\begin{tikzpicture}[baseline=.6cm,font=\scriptsize,scale=.95]
\draw (0,0)--(1,0) node[below] {$k$} --(2,0);
\draw (0,0)--(.5,.866) node[left] {$k$} --(1,1.732);
\draw (2,0)--(1.5,.866) node[right] {$k$} --(1,1.732);
\draw (0,0)--(.5,.2885) node[above] {$b$} --(1,.577);
\draw (2,0)--(1.5,.2885) node[above] {$c$} --(1,.577);
\draw[dashed] (1,.577)--(1,1.035) node[right] {$\one$} --(1,1.732);
\node[above] at (1,1.732) {$\bullet$};
\node[below] at (1,.577) {$\beta'$};
\node at (.36,.35) {$\bullet$};\node at (1.64,.35) {$\bullet$};
\end{tikzpicture}
=\frac{\delta_{b,c}}{d_bd_k}.
\]
Here $B$ is a basis of $\hc(\one,X_b\otimes X_c)$, and $\beta'$ denotes
the corresponding dual vector. This space vanishes unless $b=c$. For $b=c$, Schur's lemma converts
the normalized dual-basis pairing into $d_b^{-1}\id_{X_b}$;
the paired $k$-unit vertices contribute $d_k^{-1}$.
\end{proof}

\begin{corollary} \label{cor:loc}
Let $S'_k\subseteq S_k$. Restrict the functions of Theorem~\ref{thm:loc}, and write $x(i,b)=x(i,b,b)$. There are functions $x(i,b):S_k\times S_k' \to \field $ and $y(i,b):S_k\times S_k' \to \field$ such that for all $a,b \in S'_k$
\begin{align}
\label{equ:Ia}   \delta_{a,b} = d_b & \sum_{i \in S_k} d_i y(i,a)y(i,b), \\
\label{equ:Ib}   x(a,b) = & \sum_{i \in S_k} d_i y(i,a)y(i,b)^2,         \\
\label{equ:IIa} y(a,b)^2 = & \sum_{i \in S_k} d_i y(i,a)x(i,b),
\end{align}
with $x(a,k) = y(a,k)^2$ when $k\in S'_k$; $x(a,b) = 0$ if $N_{b,b}^a = 0$; $y(a,b) = y(b,a)$; $y(1,b) = d_k^{-1}$; $x(1,b) = (d_bd_k)^{-1}$.
\end{corollary}
\begin{proof}
Since $X_k$ is self-dual, $N_{k,k}^{1}=1$ and hence $1\in S_k$. The functions in Theorem~\ref{thm:loc} are defined on all of $S_k$, regardless of whether $1\in S'_k$. Equation~\eqref{equ:Ibis} shows that $x(a,b,c)$ is symmetric in its three arguments. Set $c=1$ and use $y(i,1)=d_k^{-1}$ and $x(a,b,1)=\delta_{a,b}(d_b d_k)^{-1}$ to obtain \eqref{equ:Ia}. Setting $c=b$ in \eqref{equ:Ibis} and \eqref{equ:IIbis} gives \eqref{equ:Ib} and \eqref{equ:IIa}. Restricting the second arguments to $S'_k$ gives the stated subsystem.
\end{proof}

\begin{remark}
For an obstruction it suffices to retain any subset of these necessary equations. Fix an order on $S'_k$. One useful selection retains \eqref{equ:Ia} for $a\ne1$ and $a\ge b$, \eqref{equ:Ib} for $a,b\ne1$, and \eqref{equ:IIa} for $b\notin\{1,k\}$. This is a selection of a necessary subsystem; no equivalence to the full system is needed for an obstruction. Symmetry avoids duplicating \eqref{equ:Ia}, and equations involving the unit often simplify after the prescribed values are substituted. In the application below, we state the retained equations explicitly. The \href{https://github.com/sebastienpalcoux/triangular-prism-categorification/blob/f816c697fde088b261e1fe1ce703980e5dbdab51/localization.py}{general generator} described in \href{https://github.com/sebastienpalcoux/triangular-prism-categorification/blob/f816c697fde088b261e1fe1ce703980e5dbdab51/GUIDE.md}{\texttt{GUIDE.md}} constructs all instances and removes identically zero or scalar-proportional duplicates.
\end{remark}

\begin{theorem} \label{thm:locextra}
Suppose $X_k$ and $X_l$ both satisfy the hypotheses of Theorem~\ref{thm:loc}, with $k\ne l$. Choose subsets $S'_k\subseteq S_k$ and $S'_l\subseteq S_l$ such that $l\in S'_k\cap S'_l$, and form the subsystems $E_k,E_l$ of Corollary~\ref{cor:loc} using compatible vertex bases in the same category. Then their variables satisfy the additional equation

$$ x_k(l,l) = \sum_{i \in S_k \cap S_l} d_i y_l(i,l)x_k(i,l)$$
\end{theorem}

\begin{proof}
Use the same strictly pivotal realization, self-duality maps and marked
vertex bases in the two localizations. Transport each basis together with
its dual for the evaluation pairing, as in Theorem~\ref{thm:loc}. The mixed prism is
$$
\scalebox{.75}{
\begin{tikzpicture}[scale=2]
\draw (0,2)--(3,2);
\draw (1,1)--(2,1);
\draw (0,0)--(3,0);
\draw (0,0) -- (0,2) -- (1,1) -- (0,0);
\draw (3,0)--(3,2)--(2,1)--(3,0);
\draw (0,2) -- (1.5,2) node [above] {$k$};
\draw (1,1) -- (1.5,1) node [above] {$k$};
\draw (0,0) -- (1.5,0) node [above] {$l$};
\draw (0,2) --++ (.5,-.5) node [above right=1pt] {$k$};
\draw (1,1) --++ (-.5,-.5) node [below right=1pt] {$l$};
\draw (0,0) -- (0,1) node [left] {$l$};
\draw (3,2) --++ (-.5,-.5) node [above left=1pt] {$k$};
\draw (2,1) --++ (.5,-.5) node [below left=1pt] {$l$};
\draw (3,0) -- (3,1) node [right] {$l$};
\node at (0+.15,2-.3) {$\bullet$};
\node at (1-.2,1) {$\bullet$};
\node at (0+.15,0+.3) {$\bullet$};
\node at (3-.15,2-.3) {$\bullet$};
\node at (2+.2,1) {$\bullet$};
\node at (3-.15,0+.3) {$\bullet$};
\end{tikzpicture}} $$

The three side cuts carry $X_k^{\otimes2},X_l^{\otimes2},X_l^{\otimes2}$.
Hence $i\in S_k\cap S_l$ and $N_{k,k}^i=N_{l,l}^i=1$; also
$N_{k,k}^l=N_{l,l}^l=1$ since $l\in S_k\cap S_l$. With
$x_k(i,l)=x_k(i,l,l)$ the TPE reads
\begin{equation}
\underbrace{\vcenter{\hbox{\scalebox{.75}{\begin{tikzpicture}[scale=1.5]
\draw (0,0) node [below] {$\bullet$} -- (1,0)  node [below] {$l$}; \draw (1,0)  -- (2,0);
\draw (2,0) node [below] {$\bullet$} -- (1.5,.866) node [right] {$k$}; \draw (1.5,.866) -- (1,1.732);
\draw (1,1.732) node [above] {$\bullet$} -- (.5,.866) node [left] {$l$}; \draw (.5,.866) -- (0,0);
\draw (0,0) -- (.5,.2885) node [above] {$l$}; \draw (.5,.2885) -- (1,0.577);
\draw (1,0.577) node [below] {$\bullet$} -- (1,1.1545) node [right] {$k$}; \draw (1,1.1545) -- (1,1.732);
\draw (1,0.577) -- (1.5,.2885) node [above] {$k$}; \draw (1.5,.2885) -- (2,0);
\end{tikzpicture}}}}\,\vcenter{\hbox{\scalebox{.75}{\begin{tikzpicture}[scale=1.5]
\draw (0,0) node [below] {$\bullet$} -- (1,0)  node [below] {$l$}; \draw (1,0)  -- (2,0);
\draw (2,0) node [below] {$\bullet$} -- (1.5,.866) node [right] {$l$}; \draw (1.5,.866) -- (1,1.732);
\draw (1,1.732) node [above] {$\bullet$} -- (.5,.866) node [left] {$k$}; \draw (.5,.866) -- (0,0);
\draw (0,0) -- (.5,.2885) node [above] {$k$}; \draw (.5,.2885) -- (1,0.577);
\draw (1,0.577) node [below] {$\bullet$} -- (1,1.1545) node [right] {$k$}; \draw (1,1.1545) -- (1,1.732);
\draw (1,0.577) -- (1.5,.2885) node [above] {$l$}; \draw (1.5,.2885) -- (2,0);
\end{tikzpicture}}}}}_{x_k(l,l)}=\sum_{i\in S_k\cap S_l}d_i\,\underbrace{\vcenter{\hbox{\scalebox{.75}{\begin{tikzpicture}[scale=1.5]
\draw (0,0)  -- (1,0)  node [below] {$k$}; \draw (1,0)  -- (2,0);
\draw (2,0) -- (1.5,.866) node [right] {$k$}; \draw (1.5,.866) -- (1,1.732);
\draw (1,1.732) node [above] {$\bullet$} -- (.5,.866) node [left] {$k$}; \draw (.5,.866) -- (0,0);
\draw (0,0) -- (.75,.433) node [above] {$l$}; \draw (.75,.433) -- (1,0.577);
\draw (1,0.577) node [below] {$\bullet$} -- (1,1.1545) node [left] {$i$}; \draw (1,1.1545) -- (1,1.732);
\draw (1,0.577) -- (1.25,.433) node [above] {$l$}; \draw (1.25,.433) -- (2,0);
\node at (.3,.3) {$\bullet$};
\node at (2-.3,.3) {$\bullet$};
\end{tikzpicture}}}}\,\vcenter{\hbox{\scalebox{.75}{\begin{tikzpicture}[scale=1.5]
\draw (0,0) -- (1,0)  node [below] {$k$}; \draw (1,0)  -- (2,0);
\draw (2,0) -- (1.5,.866) node [right] {$l$}; \draw (1.5,.866) -- (1,1.732);
\draw (1,1.732) node [above] {$\bullet$} -- (.5,.866) node [left] {$l$}; \draw (.5,.866) -- (0,0);
\draw (0,0) -- (.75,.433) node [above] {$k$}; \draw (.75,.433) -- (1,0.577);
\draw (1,0.577) node [below] {$\bullet$} -- (1,1.1545) node [left] {$i$}; \draw (1,1.1545) -- (1,1.732);
\draw (1,0.577) -- (1.25,.433) node [above] {$k$}; \draw (1.25,.433) -- (2,0);
\node at (.3,.3) {$\bullet$};
\node at (2-.3,.3) {$\bullet$};
\end{tikzpicture}}}}}_{x_k(i,l)}\,\underbrace{\vcenter{\hbox{\scalebox{.75}{\begin{tikzpicture}[scale=1.5]
\draw (0,0) -- (1,0)  node [below] {$l$}; \draw (1,0)  -- (2,0);
\draw (2,0) -- (1.5,.866) node [right] {$l$}; \draw (1.5,.866) -- (1,1.732);
\draw (1,1.732) node [above] {$\bullet$} -- (.5,.866) node [left] {$l$}; \draw (.5,.866) -- (0,0);
\draw (0,0) -- (.75,.433) node [above] {$l$}; \draw (.75,.433) -- (1,0.577);
\draw (1,0.577) node [below] {$\bullet$} -- (1,1.1545) node [left] {$i$}; \draw (1,1.1545) -- (1,1.732);
\draw (1,0.577) -- (1.25,.433) node [above] {$l$}; \draw (1.25,.433) -- (2,0);
\node at (.3,.3) {$\bullet$};
\node at (2-.3,.3) {$\bullet$};
\end{tikzpicture}}}}}_{y_l(i,l)}.
\end{equation}
The marks on the left are restored by spherical rotation as above;
the two transports at the paired $l,l,l$ vertices have inverse scalars.
\end{proof}
We apply these equations in \S\ref{sub:F210}.

\paragraph{Applying localization to another fusion ring.}
Choose $k$ and $S'_k$ as in Corollary~\ref{cor:loc}. The \href{https://github.com/sebastienpalcoux/triangular-prism-categorification/blob/f816c697fde088b261e1fe1ce703980e5dbdab51/criteria.py#L26-L30}{\texttt{localize} command} of \href{https://github.com/sebastienpalcoux/triangular-prism-categorification/blob/f816c697fde088b261e1fe1ce703980e5dbdab51/criteria.py}{\texttt{criteria.py}} takes the ordered fusion coefficients, allows noncommutative fusion rings, and records the source of each polynomial. It accepts a rational categorical dimension character or treats its values as variables subject to
\[
 d_0=1,\qquad d_{i^*}=d_i,\qquad d_i d_j=\sum_aN_{ij}^a d_a,
 \qquad d_i r_i=1
\]
(the program numbers the unit $0$). Inconsistency over $\overline{\mathbb Q}$ with symbolic dimensions excludes all spherical categorifications in characteristic zero. With prescribed dimensions it excludes only that character. For an integral fusion ring the Frobenius--Perron character suffices in characteristic zero by pseudo-unitarity \cite[Proposition~8.24]{ENO05}.

Use compatible vertex bases when coupling two systems by Theorem~\ref{thm:locextra}. Each distinguished label has its own variables, and an ordered pair $(k,l)$ imposes the displayed compatibility equation. The file \href{https://github.com/sebastienpalcoux/triangular-prism-categorification/blob/f816c697fde088b261e1fe1ce703980e5dbdab51/GUIDE.md}{\texttt{GUIDE.md}} gives the input format and worked examples.

\subsection{Zero and one spectrum criteria} \label{sec:SpeCrit}
The index $i^*$ represents the common isomorphism class of the left and right duals of a simple. We use neither a pivotal structure nor categorical dimensions. Semisimplicity is essential to the resolution into simples, so the criteria do not address categorification by a nonsemisimple finite tensor category.

\begin{remark}[Scalar extension]\label{rk:scalar-extension}
Every categorification $\mC$ of a fusion ring over an arbitrary field $\field$ is split.
Indeed, since $\End_{\mC}(\one)=\field$ by definition, rigidity and the unit multiplicity give
\[
 \dim_{\field}\End_{\mC}(X)
 =\dim_{\field}\hc(\one,X\otimes X^*)=[X\otimes X^*:\one]=1
\]
for every simple $X$. Thus $\mC$ is split: $\End_{\mC}(X)=\field$.
Scalar extension to $\overline{\field}$ preserves the simple objects, fusion rules,
and any pivotal or spherical structure \cite[\S4.16]{EGNO15}.
Hence noncategorifiability may be checked over algebraically closed fields
in each characteristic.
\end{remark}

A summand on the right-hand side of \eqref{Equ: Pentagon Equation} can occur only for
\[
 \spec\in\SpecS=\{i\in I:N_{i_4,i_7}^{i}N_{i_5^*,i_8}^{i}N_{i_6,i_9^*}^{i}>0\}.
\]
We call this set the \emph{spectrum} of the equation. An empty spectrum makes the three-factor side vanish. The first criterion forces the two-factor side to be nonzero. The second forces a unique nonzero term on the three-factor side and an empty sum on the two-factor side. All numerical conditions are tested in the nonnegative integers, in every characteristic.

\begin{notation} \label{not:2}
For the basis $\mB(i,j;k)$ fixed in \S\ref{sec:PE}, choose the dual basis in $\hc(X_k,X_i\otimes X_j)$ for the composition pairing: $\alpha\circ\beta'=\delta_{\alpha,\beta}\id_{X_k}$. Equivalently,
\[
\begin{tikzpicture}[baseline=(current bounding box.center),scale=.65]
\draw (0,1.7) node[above] {$k$} --(0,1)
 --(-.75,.3)--(0,-.4)--(0,-1.1) node[below] {$k$};
\draw (0,1)--(.75,.3)--(0,-.4);
\node[left] at (-.75,.3) {$i$};\node[right] at (.75,.3) {$j$};
\node[above right] at (0,1) {$\beta'$};\node[below right] at (0,-.4) {$\alpha$};
\end{tikzpicture}
=\delta_{\alpha,\beta}
\begin{tikzpicture}[baseline=(current bounding box.center),scale=.65]
\draw (0,1.7) node[above] {$k$} -- (0,-1.1) node[below] {$k$};
\end{tikzpicture}.
\]
The pairing is perfect: the projections onto the copies of $X_k$ in $X_i\otimes X_j$ and their inclusions pair to the identity matrix. We write
$\cup_i=\ev_{X_i}$ and $\cap_i=\coev_{X_i}$; the standard rigid
adjunctions are used with their stated sources and targets.
\end{notation}

\begin{lemma}[One-dimensional trick]\label{Lem: one-dim}
Choose basis morphisms $\mu_1 \in \mB(i_1,i_2;i_3)$, $\mu_2 \in \mB(i_3,i_4;i_5)$, $\mu_3 \in \mB(i_2,i_4;i_6)$ and $\mu_4 \in \mB(i_1,i_6;i_5)$. If
$$
\sum_{k\in I} N_{i_1,i_6}^{k} N_{i_3,i_4}^{k} = 1
~\text{or}~
\sum_{k \in I} N_{i_1,i_2}^{k} N_{i_5,i_4^*}^{k} =1
~\text{or}~
\sum_{k\in I} N_{i_3,i_2^*}^{k} N_{i_5,i_6^*}^{k} = 1,
$$
then
$$
\left(
\begin{array}{ccc | cc}
i_1& i_2 & i_3 & \mu_1 & \mu_2 \\
i_4& i_5 & i_6 & \mu_3 & \mu_4
\end{array}
\right) \neq 0.
$$
\end{lemma}
\begin{proof}
Write $(A,B,C,D,E,S)=(X_{i_1},\ldots,X_{i_6})$ and denote the displayed coefficient by $F$. By semisimplicity and rigidity, the three sums are, in order, the dimensions of
\[
 H_1=\hc(A\otimes S,C\otimes D),\quad H_2=\hc(A\otimes B\otimes D,E),\quad H_3=\hc(E,C\otimes B^*\otimes S).
\]
Primes denote the composition-duals of Notation~\ref{not:2}.

\emph{First cut.} Precompose the recoupling \eqref{Equ: F-symbols} with the tree splitting $E$ by $\mu'_4$ and then $S$ by $\mu'_3$. This gives the first equality. If $\dim H_1=1$, the boxed maps are proportional:
\newcommand{\caseone}{\vcenter{\hbox{\begin{tikzpicture}[x=.82cm,y=.65cm,font=\scriptsize,inner sep=1pt]
\draw[dashed,gray,rounded corners=1pt] (-2.16,.42) rectangle (2.12,2.65);
\coordinate (a) at (0,3.55); \coordinate (b) at (.79,2.12);
\coordinate (c) at (-.79,.91); \coordinate (d) at (0,-.52);
\draw (0,4.03) node[above] {$E$} --(a)--(-1.4,2.6)--(-1.4,1.37)--(c)--(d)--(0,-1.00) node[below] {$E$};
\draw (a)--(b)--(c); \draw (b)--(1.4,1.65)--(1.4,.45)--(d);
\foreach \pt in {a,b,c,d} \fill (\pt) circle(.95pt);
\node[above left=2pt] at (a) {$\mu'_4$};
\node[anchor=west] at (.95,2.28) {$\mu'_3$};
\node[left=2pt] at (c) {$\mu_1$};
\node[below right=2pt] at (d) {$\mu_2$};
\node[left=2pt] at (-1.4,1.93) {$A$};
\node[right=2pt] at (.34,2.95) {$S$};
\node at (.28,1.30) {$B$};
\node[right=2pt] at (1.4,1.07) {$D$};
\node[left=2pt] at (-.34,-.02) {$C$};
\end{tikzpicture}}}}
\newcommand{\twobubbles}{\vcenter{\hbox{\begin{tikzpicture}[x=.82cm,y=.65cm,font=\scriptsize,inner sep=1pt]
\draw[dashed,gray,rounded corners=1pt] (-1.34,.22) rectangle (1.34,2.80);
\draw (0,4.03) node[above] {$E$} --(0,3.55);
\draw (0,3.55) .. controls (-1.52,3.38) and (-1.52,2.30) .. (0,2.12)
 .. controls (1.52,2.30) and (1.52,3.38) .. (0,3.55);
\draw (0,2.12)--(0,.91);
\draw (0,.91) .. controls (-1.52,.73) and (-1.52,-.34) .. (0,-.52)
 .. controls (1.52,-.34) and (1.52,.73) .. (0,.91);
\draw (0,-.52)--(0,-1.00) node[below] {$E$};
\foreach \y in {3.55,2.12,.91,-.52} \fill (0,\y) circle(.95pt);
\node[above right=2pt] at (0,3.55) {$\mu'_4$};
\node[below right=2pt] at (0,2.12) {$\mu_4$};
\node[above right=2pt] at (0,.91) {$\mu'_2$};
\node[below right=2pt] at (0,-.52) {$\mu_2$};
\node at (-1.26,3.25) {$A$}; \node at (1.26,3.25) {$S$};
\node at (-1.26,-.20) {$C$}; \node at (1.26,-.20) {$D$};
\end{tikzpicture}}}}

\[
 F\,\id_E=\caseone\;=\;\lambda_1\,\twobubbles\;=\;\lambda_1\,\id_E .
\]
Both boxes are nonzero. The right one factors through the simple $E$ by a split epimorphism followed by a split monomorphism. For the left one, bend its $C$- and $S$-legs by rigidity \cite[Proposition~2.10.8]{EGNO15}. Its mate factors through the simple $B^*$ in the same way:
\[
\widehat\mu_1=
\begin{tikzpicture}[baseline=(current bounding box.center),x=.65cm,y=.62cm,font=\scriptsize]
\draw[dashed,gray] (-.50,-.32) rectangle (2.45,1.52);
\node[anchor=south west,inner sep=1.5pt] at (2.45,1.52) {$p$};
\draw (-1.2,1.6) node[above] {$C^*$} --(-1.2,-.6)
 .. controls (-1.2,-1.15) and (.5,-1.15) .. (.5,-.6)--(.5,.05);
\draw (0,1.6) node[above] {$A$} --(0,.45);
\draw (1,.45)--(1,1.05)
 .. controls (1,1.55) and (2.2,1.55) .. (2.2,1.05)
 --(2.2,-1.35) node[below] {$B^*$};
\node[left,font=\scriptsize] at (1,.85) {$B$};
\node[right,font=\scriptsize] at (.5,-.64) {$C$};
\draw[fill=white] (-.3,-.035) rectangle (1.3,.535);
\node at (.5,.25) {$\mu_1$};
\end{tikzpicture},
\qquad
\widehat{\mu'_3}=
\begin{tikzpicture}[baseline=(current bounding box.center),x=.65cm,y=.62cm,font=\scriptsize]
\draw[dashed,gray] (-1.49,-1.16) rectangle (1.60,.85);
\node[anchor=north east,inner sep=1.5pt] at (-1.49,-1.16) {$z$};
\draw (-1.2,1.6) node[above] {$B^*$} --(-1.2,-.6)
 .. controls (-1.2,-1.15) and (0,-1.15) .. (0,-.6)--(0,.05);
\draw (.5,.45)--(.5,1.05)
 .. controls (.5,1.55) and (2.2,1.55) .. (2.2,1.05)
 --(2.2,-1.35) node[below] {$S^*$};
\draw (1,.05)--(1,-1.35) node[below] {$D$};
\node[left,font=\scriptsize] at (.5,1.15) {$S$};
\node[right,font=\scriptsize] at (0,-.45) {$B$};
\draw[fill=white] (-.3,-.035) rectangle (1.3,.535);
\node at (.5,.25) {$\mu'_3$};
\end{tikzpicture},
\qquad
\begin{tikzpicture}[baseline=(current bounding box.center),x=.78cm,y=.66cm,font=\scriptsize,inner sep=1pt]
\draw (-.9,1.82) node[above] {$C^*$} --(0,.76);
\draw (.9,1.82) node[above] {$A$} --(0,.76)--(0,-.76);
\draw (0,-.76)--(-.9,-1.82) node[below] {$D$};
\draw (0,-.76)--(.9,-1.82) node[below] {$S^*$};
\node[right=2pt] at (0,0) {$B^*$};
\node[anchor=east] at (-.29,.61) {$\widehat\mu_1$};
\node[anchor=east] at (-.29,-.61) {$\widehat{\mu'_3}$};
\end{tikzpicture}
\ne0.
\]
The dashed boxes here also define the one-bend mates $p:A\to C\otimes B^*$ and $z:B^*\otimes S\to D$. In particular $\lambda_1\ne0$.

\newcommand{\oneDimTree}[9]{\vcenter{\hbox{\begin{tikzpicture}[x=.73cm,y=.62cm,yscale=#1,xscale=#2,font=\scriptsize,inner sep=1pt]
\draw (0,2.45)--(0,1.89)--(-1.07,.90)--(-1.57,0);
\draw (-1.07,.90)--(-.28,0);\draw (0,1.89)--(1.32,.80)--(1.32,0);
\node at (0,2.71){$#3$};\node at (-1.57,-.28){$#4$};
\node at (-.28,-.28){$#5$};\node at (1.32,-.28){$#6$};
\node at (-.55,2.06){$#7$};\node at (-1.64,.98){$#8$};
\node at (-.24,1.22){$#9$};
\end{tikzpicture}}}}

\emph{Second cut.} If $\dim H_2=1$, the two fusion trees are nonzero split epimorphisms, so
\[
 \oneDimTree{-1}{1}{E}{A}{B}{D}{\mu_2}{\mu_1}{C}
 \;=\;\lambda_2\,
 \oneDimTree{-1}{-1}{E}{D}{B}{A}{\mu_4}{\mu_3}{S}
 \qquad (\lambda_2\ne0).
\]
Precompose with $(\id_A\otimes\mu'_3)\circ\mu'_4$. The left tree gives $F\,\id_E$ and the right tree straightens to $\id_E$, so $F=\lambda_2$.

\emph{Third cut.} Choose a section $z'$ of the nonzero map $z$ onto $D$. If $\dim H_3=1$, the two splitting trees are nonzero split monomorphisms, so
\[
 \oneDimTree{1}{1}{E}{C}{B^*}{S}{\mu'_4}{p}{A}
 \;=\;\lambda_3\,
 \oneDimTree{1}{-1}{E}{S}{B^*}{C}{\mu'_2}{z'}{D}
 \qquad (\lambda_3\ne0).
\]
Postcompose with $\mu_2\circ(\id_C\otimes z)$. The zigzag between $p$ and $z$ straightens to the first network; on the right $z\circ z'=\id_D$ leaves $\mu_2\circ\mu'_2=\id_E$. Thus $F=\lambda_3$.
\end{proof}

\newcommand{\spectrumtree}[1]{\vcenter{\hbox{%
\begin{tikzpicture}[x=.30cm,y=.31cm,line width=.55pt,line cap=round,line join=round]
\coordinate (a) at (0,2.42); \coordinate (b) at (1,2.42);
\coordinate (c) at (2,2.42); \coordinate (d) at (3,2.42);
\coordinate (o) at (1.5,0); \draw (o)--(1.5,-.47);
\ifcase#1\relax
 \coordinate (v) at (.5,1.70); \coordinate (w) at (1,.88);
 \draw (a)--(v)--(b); \draw (v)--(w)--(c); \draw (w)--(o)--(d);
\or
 \coordinate (v) at (.5,1.45); \coordinate (w) at (2.5,1.45);
 \draw (a)--(v)--(b); \draw (c)--(w)--(d); \draw (v)--(o)--(w);
\or
 \coordinate (v) at (2.5,1.70); \coordinate (w) at (2,.88);
 \draw (c)--(v)--(d); \draw (b)--(w)--(v); \draw (a)--(o)--(w);
\or
 \coordinate (v) at (1.5,1.70); \coordinate (w) at (1,.88);
 \draw (b)--(v)--(c); \draw (a)--(w)--(v); \draw (w)--(o)--(d);
\or
 \coordinate (v) at (1.5,1.70); \coordinate (w) at (2,.88);
 \draw (b)--(v)--(c); \draw (v)--(w)--(d); \draw (a)--(o)--(w);
\fi
\end{tikzpicture}}}}

\begin{theorem}[Zero spectrum criterion]\label{Thm: Ob-0}
For a fusion ring $R$, if there exist indices $i_j \in I$ for $1 \leq j \leq 9$, such that the fusion coefficients $N_{i_4,i_1}^{i_6}$, $N_{i_5,i_4}^{i_2}$, $N_{i_5,i_6}^{i_3}$, $N_{i_7,i_9}^{i_1}$, $N_{i_2,i_7}^{i_8}$, and $N_{i_8,i_9}^{i_3}$  are non-zero, and if the following conditions hold:
\begin{align} \label{Equ: 01}
\sum_{k} N_{i_4,i_7}^{k} N_{i_5^*,i_8}^{k} N_{i_6,i_9^*}^{k}&=0,\\  \label{Equ: 02}
N_{i_2,i_1}^{i_{3}}&=1, \\ \label{Equ: 03}
\sum_{k} N_{i_5,i_6}^{k} N_{i_2,i_1}^{k} =1 
~\text{or}~
\sum_{k \in I} N_{i_5,i_4}^{k} N_{i_3,i_1^*}^{k} &=1 
~\text{or}~
\sum_{k} N_{i_2,i_4^*}^{k} N_{i_3,i_6^*}^{k} =1, 
\\
\label{Equ: 04}
\sum_{k} N_{i_2,i_1}^{k} N_{i_8,i_9}^{k} =1 
~\text{or}~
\sum_{k \in I} N_{i_2,i_7}^{k} N_{i_3,i_9^*}^{k} &=1 
~\text{or}~
\sum_{k} N_{i_8,i_7^*}^{k} N_{i_3,i_1^*}^{k} =1, 
\end{align}
then $R$ admits no fusion categorification in any characteristic.
\end{theorem}

\begin{proof}
Choose the six external basis vectors. In the trees of Figure~\ref{Fig: Pentagon Equation} the input labels are $i_5,i_4,i_7,i_9$, in that order, and the output is $i_3$. Read each row as the product of its change-of-basis coefficients. By \eqref{Equ: 02} the two-move route has just one choice of $\mu_0$. Lemma~\ref{Lem: one-dim} makes each coefficient nonzero by the condition below its arrow. The three-move route has empty spectrum by \eqref{Equ: 01}. Thus the pentagon would give
\[
 \underbrace{\spectrumtree{0}\;\xrightarrow[\eqref{Equ: 04}]{F_1}\;\spectrumtree{1}
 \;\xrightarrow[\eqref{Equ: 03}]{F_2}\;\spectrumtree{2}}_{F_1F_2\ne0}
 \;=\;\underbrace{0}_{\SpecS=\varnothing}.\qedhere
\]
\end{proof}

\begin{theorem}[One spectrum criterion]\label{Thm: Ob-1}
For a fusion ring $R$,
if there are indices $i_j \in I$, $0\leq j \leq 9$, such that the fusion coefficients
$N_{i_4,i_1}^{i_6}$, $N_{i_5,i_4}^{i_2}$, $N_{i_5,i_6}^{i_3}$, $N_{i_7,i_9}^{i_1}$, $N_{i_2,i_7}^{i_8}$, $N_{i_8,i_9}^{i_3}$ are non-zero, and
\begin{align}\label{Equ: 11}
\sum_{k} N_{i_4,i_7}^{k} N_{i_5^*,i_8}^{k} N_{i_6,i_9^*}^{k}&=1, \\ \label{Equ: 11'}
N_{i_4,i_7}^{\spec}=N_{i_5^*,i_8}^{\spec}=N_{i_6,i_9^*}^{\spec}&=1,
\\ \label{Equ: 12}
N_{i_2,i_1}^{i_{3}}&=0, \\ \label{Equ: 13}
\sum_{k} N_{i_5,\spec}^{k} N_{i_2,i_7}^{k} =1 
~\text{or}~
\sum_{k \in I} N_{i_5,i_4}^{k} N_{i_8,i_7^*}^{k} &=1 
~\text{or}~
\sum_{k} N_{i_2,i_4^*}^{k} N_{i_8,\spec^*}^{k} =1, 
\\ \label{Equ: 14}
\sum_{k} N_{i_5,i_6}^{k} N_{i_8,i_9}^{k} =1 
~\text{or}~
\sum_{k \in I} N_{i_5,\spec}^{k} N_{i_3,i_9^*}^{k} &=1 
~\text{or}~
\sum_{k} N_{i_8,\spec^*}^{k} N_{i_3,i_6^*}^{k} =1, 
\\ \label{Equ: 15}
\sum_{k} N_{i_4,i_1}^{k} N_{\spec,i_9}^{k} =1 
~\text{or}~
\sum_{k \in I} N_{i_4,i_7}^{k} N_{i_6,i_9^*}^{k} &=1 
~\text{or}~
\sum_{k} N_{\spec,i_7^*}^{k} N_{i_6,i_1^*}^{k} =1, 
\end{align}
then $R$ admits no fusion categorification in any characteristic.
\end{theorem}

\begin{proof}
Fix the external basis vectors as above. By \eqref{Equ: 12} the two-move route has no $\mu_0$. Equations~\eqref{Equ: 11}--\eqref{Equ: 11'} leave just one term on the three-move route, with all three coefficients nonzero by Lemma~\ref{Lem: one-dim}. The pentagon would give
\[
 \underbrace{0}_{\mB(i_2,i_1;i_3)=\varnothing}\;=\;
 \underbrace{\spectrumtree{0}\;\xrightarrow[\eqref{Equ: 13}]{G_1}\;\spectrumtree{3}
 \;\xrightarrow[\eqref{Equ: 14}]{G_2}\;\spectrumtree{4}
 \;\xrightarrow[\eqref{Equ: 15}]{G_3}\;\spectrumtree{2}}_{G_1G_2G_3\ne0}.\qedhere
\]
\end{proof}

\paragraph{Applying the spectrum criteria.}
The \href{https://github.com/sebastienpalcoux/triangular-prism-categorification/blob/f816c697fde088b261e1fe1ce703980e5dbdab51/criteria.py#L18-L25}{\texttt{spectrum} command} of \href{https://github.com/sebastienpalcoux/triangular-prism-categorification/blob/f816c697fde088b261e1fe1ce703980e5dbdab51/criteria.py}{\texttt{criteria.py}} tests the ordered fusion coefficients, including those of noncommutative fusion rings. It returns an ordered index tuple satisfying the chosen criterion, together with the integers in each hypothesis, or counts all such tuples. The \href{https://github.com/sebastienpalcoux/triangular-prism-categorification/blob/f816c697fde088b261e1fe1ce703980e5dbdab51/spectrum.py}{search program} calls each such tuple a \emph{witness}. If no tuple satisfies either criterion, these tests give no obstruction. We apply the zero spectrum criterion to $\mathcal F_{660}$ in \S\ref{sub:F660}.

\begin{proposition}\label{prop:spectrum-incomparable}
Neither spectrum criterion implies the other.
\end{proposition}
\begin{proof}
Number the basis elements in each of the following examples from $1$ to $6$, with the unit first. Consider the following two rank-$6$ fusion data drawn from the dataset presented in \cite{SliVer}: 
$$ {\begin{smallmatrix}1 & 0 & 0 & 0 & 0 & 0 \\0 & 1 & 0 & 0 & 0 & 0 \\0 & 0 & 1 & 0 & 0 & 0 \\0 & 0 & 0 & 1 & 0 & 0 \\0 & 0 & 0 & 0 & 1 & 0 \\0 & 0 & 0 & 0 & 0 & 1\end{smallmatrix} , \ \begin{smallmatrix}0 & 1 & 0 & 0 & 0 & 0 \\1 & 0 & 0 & 0 & 0 & 0 \\0 & 0 & 1 & 0 & 0 & 0 \\0 & 0 & 0 & 1 & 0 & 0 \\0 & 0 & 0 & 0 & 1 & 0 \\0 & 0 & 0 & 0 & 0 & 1\end{smallmatrix} , \ \begin{smallmatrix}0 & 0 & 1 & 0 & 0 & 0 \\0 & 0 & 1 & 0 & 0 & 0 \\1 & 1 & 2 & 1 & 0 & 1 \\0 & 0 & 1 & 2 & 2 & 1 \\0 & 0 & 0 & 2 & 1 & 0 \\0 & 0 & 1 & 1 & 0 & 0\end{smallmatrix} , \ \begin{smallmatrix}0 & 0 & 0 & 1 & 0 & 0 \\0 & 0 & 0 & 1 & 0 & 0 \\0 & 0 & 1 & 2 & 2 & 1 \\1 & 1 & 2 & 3 & 1 & 1 \\0 & 0 & 2 & 1 & 1 & 1 \\0 & 0 & 1 & 1 & 1 & 0\end{smallmatrix} , \ \begin{smallmatrix}0 & 0 & 0 & 0 & 1 & 0 \\0 & 0 & 0 & 0 & 1 & 0 \\0 & 0 & 0 & 2 & 1 & 0 \\0 & 0 & 2 & 1 & 1 & 1 \\1 & 1 & 1 & 1 & 0 & 0 \\0 & 0 & 0 & 1 & 0 & 1\end{smallmatrix} , \ \begin{smallmatrix}0 & 0 & 0 & 0 & 0 & 1 \\0 & 0 & 0 & 0 & 0 & 1 \\0 & 0 & 1 & 1 & 0 & 0 \\0 & 0 & 1 & 1 & 1 & 0 \\0 & 0 & 0 & 1 & 0 & 1 \\1 & 1 & 0 & 0 & 1 & 0\end{smallmatrix}} $$

$$ {\begin{smallmatrix}1 & 0 & 0 & 0 & 0 & 0 \\0 & 1 & 0 & 0 & 0 & 0 \\0 & 0 & 1 & 0 & 0 & 0 \\0 & 0 & 0 & 1 & 0 & 0 \\0 & 0 & 0 & 0 & 1 & 0 \\0 & 0 & 0 & 0 & 0 & 1\end{smallmatrix} , \ \begin{smallmatrix}0 & 1 & 0 & 0 & 0 & 0 \\1 & 0 & 0 & 0 & 0 & 0 \\0 & 0 & 1 & 0 & 0 & 0 \\0 & 0 & 0 & 1 & 0 & 0 \\0 & 0 & 0 & 0 & 0 & 1 \\0 & 0 & 0 & 0 & 1 & 0\end{smallmatrix} , \ \begin{smallmatrix}0 & 0 & 1 & 0 & 0 & 0 \\0 & 0 & 1 & 0 & 0 & 0 \\1 & 1 & 2 & 0 & 0 & 0 \\0 & 0 & 0 & 0 & 1 & 1 \\0 & 0 & 0 & 1 & 1 & 1 \\0 & 0 & 0 & 1 & 1 & 1\end{smallmatrix} , \ \begin{smallmatrix}0 & 0 & 0 & 1 & 0 & 0 \\0 & 0 & 0 & 1 & 0 & 0 \\0 & 0 & 0 & 0 & 1 & 1 \\1 & 1 & 0 & 2 & 0 & 0 \\0 & 0 & 1 & 0 & 1 & 1 \\0 & 0 & 1 & 0 & 1 & 1\end{smallmatrix} , \ \begin{smallmatrix}0 & 0 & 0 & 0 & 1 & 0 \\0 & 0 & 0 & 0 & 0 & 1 \\0 & 0 & 0 & 1 & 1 & 1 \\0 & 0 & 1 & 0 & 1 & 1 \\1 & 0 & 1 & 1 & 0 & 2 \\0 & 1 & 1 & 1 & 2 & 0\end{smallmatrix} , \ \begin{smallmatrix}0 & 0 & 0 & 0 & 0 & 1 \\0 & 0 & 0 & 0 & 1 & 0 \\0 & 0 & 0 & 1 & 1 & 1 \\0 & 0 & 1 & 0 & 1 & 1 \\0 & 1 & 1 & 1 & 2 & 0 \\1 & 0 & 1 & 1 & 0 & 2\end{smallmatrix}} $$
Counting the ordered tuples satisfying the zero- and one-spectrum criteria, in that order, gives $(12,0)$ for the \href{https://github.com/sebastienpalcoux/triangular-prism-categorification/blob/f816c697fde088b261e1fe1ce703980e5dbdab51/data/rank6_zero_only.json}{first fusion ring} and $(0,96)$ for the \href{https://github.com/sebastienpalcoux/triangular-prism-categorification/blob/f816c697fde088b261e1fe1ce703980e5dbdab51/data/rank6_one_only.json}{second}. For the first fusion ring, $(i_1,\ldots,i_9)=(4,5,6,3,5,6,4,5,6)$ satisfies Theorem~\ref{Thm: Ob-0}; for the second, $(i_0,\ldots,i_9)=(1,5,5,5,3,4,5,3,4,5)$ satisfies Theorem~\ref{Thm: Ob-1}.
\end{proof}

\begin{proposition}\label{prop:spectrum-rank}
The smallest rank at which either spectrum criterion gives an obstruction is four. Both criteria attain this bound on the same fusion ring.
\end{proposition}
\begin{proof}
No ordered tuple satisfies either criterion in ranks at most three, as shown in the \href{https://github.com/sebastienpalcoux/triangular-prism-categorification/blob/f816c697fde088b261e1fe1ce703980e5dbdab51/SPECTRUM_RANK.md}{supplementary proof} and its \href{https://github.com/sebastienpalcoux/triangular-prism-categorification/blob/f816c697fde088b261e1fe1ce703980e5dbdab51/verify_spectrum_rank.py}{finite check} using the rank-three tables of \cite[Proposition~3.1]{Os15}.

In rank four, number the basis from $1$ to $4$, with the unit first, and take
\[
\begin{pmatrix}1&0&0&0\\0&1&0&0\\0&0&1&0\\0&0&0&1\end{pmatrix},\quad
\begin{pmatrix}0&1&0&0\\1&1&0&1\\0&0&1&1\\0&1&1&2\end{pmatrix},\quad
\begin{pmatrix}0&0&1&0\\0&0&1&1\\1&1&1&0\\0&1&0&2\end{pmatrix},\quad
\begin{pmatrix}0&0&0&1\\0&1&1&2\\0&1&0&2\\1&2&2&3\end{pmatrix}.
\]
Theorems~\ref{Thm: Ob-0} and~\ref{Thm: Ob-1} apply to the respective tuples
\[
(i_1,\ldots,i_9)=(3,3,2,2,3,4,3,3,3),\qquad
(i_0,\ldots,i_9)=(3,4,3,3,2,3,3,3,3,2).\qedhere
\]
\end{proof}

\section{Applications}\label{sec:app}
We apply the localization equations to $\mathcal F_{210}$ and the zero-spectrum criterion to $\mathcal F_{660}$. Together with the census recalled in the introduction, these exclusions prove Theorem~\ref{thm:ClassSimpleIntegral}. The \href{https://github.com/sebastienpalcoux/triangular-prism-categorification/blob/f816c697fde088b261e1fe1ce703980e5dbdab51/GUIDE.md#4-verify-the-papers-applications}{verification guide} lists the input data, scripts and certificates for both applications.

\subsection{Fusion ring \texorpdfstring{$\mathcal{F}_{210}$}{F210}} \label{sub:F210}
The simple integral fusion ring $\mathcal{F}_{210}$ has rank $7$, $\mathrm{FPdim}$  $210$ and type $[[1,1],[5,3],[6,1],[7,2]]$, and the following \href{https://github.com/sebastienpalcoux/triangular-prism-categorification/blob/f816c697fde088b261e1fe1ce703980e5dbdab51/data/f210.json}{fusion matrices}:
$$ {\begin{smallmatrix}
1 & 0 & 0 & 0& 0& 0& 0 \\
0 & 1 & 0 & 0& 0& 0& 0 \\
0 & 0 & 1 & 0& 0& 0& 0 \\
0 & 0 & 0 & 1& 0& 0& 0 \\
0 & 0 & 0 & 0& 1& 0& 0 \\
0 & 0 & 0 & 0& 0& 1& 0 \\
0 & 0 & 0 & 0& 0& 0& 1
\end{smallmatrix} , \
\begin{smallmatrix}
0 & 1 & 0 & 0& 0& 0& 0 \\
1 & 1 & 0 & 1& 0& 1& 1 \\
0 & 0 & 1 & 0& 1& 1& 1 \\
0 & 1 & 0 & 0& 1& 1& 1 \\
0 & 0 & 1 & 1& 1& 1& 1 \\
0 & 1 & 1 & 1& 1& 1& 1 \\
0 & 1 & 1 & 1& 1& 1& 1
\end{smallmatrix} , \
\begin{smallmatrix}
0 & 0 & 1 & 0& 0& 0& 0 \\
0 & 0 & 1 & 0& 1& 1& 1 \\
1 & 1 & 1 & 0& 0& 1& 1 \\
0 & 0 & 0 & 1& 1& 1& 1 \\
0 & 1 & 0 & 1& 1& 1& 1 \\
0 & 1 & 1 & 1& 1& 1& 1 \\
0 & 1 & 1 & 1& 1& 1& 1
\end{smallmatrix} , \
\begin{smallmatrix}
0 & 0 & 0 & 1& 0& 0& 0 \\
0 & 1 & 0 & 0& 1& 1& 1 \\
0 & 0 & 0 & 1& 1& 1& 1 \\
1 & 0 & 1 & 1& 0& 1& 1 \\
0 & 1 & 1 & 0& 1& 1& 1 \\
0 & 1 & 1 & 1& 1& 1& 1 \\
0 & 1 & 1 & 1& 1& 1& 1
\end{smallmatrix} , \
\begin{smallmatrix}
0 & 0 & 0 & 0& 1& 0& 0 \\
0 & 0 & 1 & 1& 1& 1& 1 \\
0 & 1 & 0 & 1& 1& 1& 1 \\
0 & 1 & 1 & 0& 1& 1& 1 \\
1 & 1 & 1 & 1& 1& 1& 1 \\
0 & 1 & 1 & 1& 1& 2& 1 \\
0 & 1 & 1 & 1& 1& 1& 2
\end{smallmatrix} , \
\begin{smallmatrix}
0 & 0 & 0 & 0& 0& 1& 0 \\
0 & 1 & 1 & 1& 1& 1& 1 \\
0 & 1 & 1 & 1& 1& 1& 1 \\
0 & 1 & 1 & 1& 1& 1& 1 \\
0 & 1 & 1 & 1& 1& 2& 1 \\
1 & 1 & 1 & 1& 2& 1& 2 \\
0 & 1 & 1 & 1& 1& 2& 2
\end{smallmatrix} , \
\begin{smallmatrix}
0 & 0 & 0 & 0& 0& 0& 1 \\
0 & 1 & 1 & 1& 1& 1& 1 \\
0 & 1 & 1 & 1& 1& 1& 1 \\
0 & 1 & 1 & 1& 1& 1& 1 \\
0 & 1 & 1 & 1& 1& 1& 2 \\
0 & 1 & 1 & 1& 1& 2& 2 \\
1 & 1 & 1 & 1& 2& 2& 1
\end{smallmatrix}}$$

\begin{remark} \label{rk:interpolated}
The interpolation construction helps explain why arithmetic categorification criteria do not exclude $\mathcal{F}_{210}$. The character table of $\Rep(\PSL(2,q))$ can be described uniformly, with separate formulas for $q$ even and $q\equiv\pm1\pmod4$ (see, e.g., \cite[\S 5.2]{FuHa} or \cite[\S 12.5]{DiMi}). This table can be interpolated to any integer $q>1$, and applying the Schur orthogonality relations then yields an infinite family of simple integral fusion rings (simple for $q \ge 4$). When $q$ is a prime power, these coincide with the Grothendieck ring of $\Rep(\PSL(2,q))$, but for other integers they give new fusion rings with the same favorable arithmetic properties (see \cite{LPRinter}). The case $q=6$ corresponds precisely to $\mathcal{F}_{210}$. It is expected that similar interpolated simple integral fusion rings exist for every family of finite simple groups of Lie type.
\end{remark}

Let us label (and order) the simple objects by $1, 5_1, 5_2, 5_3, 6_1, 7_1, 7_2$, with FPdim $1,5,5,5,6,7,7$ (respectively). Observe that all the fusion matrices are self-adjoint, so the simple objects are self-dual.

\begin{theorem} \label{thm: F210NoCat}
The fusion ring $\mathcal{F}_{210}$ cannot be categorified in characteristic zero.
\end{theorem}
\begin{proof}
A characteristic-zero categorification $\mC$ is split by Remark~\ref{rk:scalar-extension}. Its finitely many structure constants lie in a finitely generated field that embeds into $\mathbb C$, so we may work over $\mathbb C$. Since $\mC$ is integral, it is pseudo-unitary and then spherical by \cite[Propositions 9.6.5 and 9.5.1]{EGNO15}. We can apply Corollary \ref{cor:loc} with $k=5_1$, $S_k = \{1, 5_1, 5_3, 7_1, 7_2 \}$, and $S'_k = \{1, 5_1, 5_3\}$. We get the following subsystem $E_k$ of $10$ variables and $12$ equations:
\[
\setlength{\arraycolsep}{4pt}
\begin{array}{rcl@{\qquad}rcl}
5u_0 + 7u_1 + 7u_2 - 4/25 & = & 0, &
5v_1 + 5v_2 + 7v_4 + 7v_6 + 1/5 & = & 0, \\
5v_0 + 5v_1 + 7v_3 + 7v_5 + 1/5 & = & 0, &
25v_0 v_1 + 25v_1 v_2 + 35v_3 v_4 + 35v_5 v_6 + 1/5 & = & 0, \\
25v_0^2 + 25v_1^2 + 35v_3^2 + 35v_5^2 - 4/5 & = & 0, &
5v_0^2 v_1 + 5v_1^2 v_2 + 7v_3^2 v_4 + 7v_5^2 v_6 - v_1^2 + 1/125 & = & 0, \\
5v_0^3 + 5v_1^3 + 7v_3^3 + 7v_5^3 - v_0^2 + 1/125 & = & 0, &
25v_1^2 + 25v_2^2 + 35v_4^2 + 35v_6^2 - 4/5 & = & 0, \\
5v_0 v_1^2 + 5v_1 v_2^2 + 7v_3 v_4^2 + 7v_5 v_6^2 + 1/125 & = & 0, &
5v_1^3 + 5v_2^3 + 7v_4^3 + 7v_6^3 - u_0 + 1/125 & = & 0, \\
5u_0 v_1 - v_1^2 + 7u_1 v_3 + 7u_2 v_5 + 1/125 & = & 0, &
5u_0 v_2 - v_2^2 + 7u_1 v_4 + 7u_2 v_6 + 1/125 & = & 0,
\end{array}
\]
where $u_0 = x_k(5_3,5_3), u_1=x_k(7_1,5_3), u_2=x_k(7_2,5_3), v_0=y_k(5_1,5_1), v_1=y_k(5_3,5_1), v_2=y_k(5_3,5_3), v_3=y_k(7_1,5_1), v_4=y_k(7_1,5_3), v_5=y_k(7_2,5_1), v_6=y_k(7_2,5_3)$.

For each local system, the selected prisms give $55$ marked tetrahedral coefficients in $19$ distinct nonzero equations when arrows are retained and unit edges are left unevaluated (\href{https://github.com/sebastienpalcoux/triangular-prism-categorification/blob/f816c697fde088b261e1fe1ce703980e5dbdab51/verify_f210_ordered.py}{count on GitHub}). Rotating the tetrahedra, evaluating unit edges and taking the paired contractions that define $x$ reduce the calculation to the $12$ equations in $10$ variables displayed above.

In the displayed system three linear equations eliminate $u_0,v_0,v_2$. The remaining seven variables admit an exact parametrization by the roots of a squarefree polynomial of degree $14$. Thus there are exactly $14$ complex solutions, as certified by \href{https://github.com/sebastienpalcoux/triangular-prism-categorification/blob/f816c697fde088b261e1fe1ce703980e5dbdab51/certificates/f210-solutions.json}{\nolinkurl{certificates/f210-solutions.json}} and the spanning bound below.

Next we can apply Theorem \ref{thm:locextra} with $k, S_k, S'_k$ as above, together with $l=5_3$, $S_l = \{1, 5_2, 5_3, 7_1, 7_2 \}$ and $S'_l = \{1, 5_2, 5_3\}$. We get a subsystem $E_l$ obtained from $E_k$ by $u_j\mapsto w_j$ and $(v_0,v_1,v_2,v_3,v_4,v_5,v_6)\mapsto(z_2,z_1,z_0,z_4,z_3,z_6,z_5)$,
with $w_0=x_l(5_2, 5_2), w_1=x_l(7_1,5_2), w_2=x_l(7_2,5_2), z_0=y_l(5_2,5_2), z_1=y_l(5_3,5_2), z_2=y_l(5_3,5_3), z_3=y_l(7_1,5_2), z_4=y_l(7_1,5_3), z_5=y_l(7_2,5_2), z_6=y_l(7_2,5_3)$. The common simple summands in Theorem~\ref{thm:locextra} give the extra equation
\begin{equation}\label{eq:extra}
 \underbrace{1/125}_{i=1}
 +\underbrace{5u_0z_2}_{i=5_3}
 +\underbrace{7u_1z_4}_{i=7_1}
 +\underbrace{7u_2z_6}_{i=7_2}-u_0=0 .
\end{equation}
Let $J$ be the ideal over $\mathbb Q$ generated by these two explicitly specified systems and the left-hand side of \eqref{eq:extra}. We verify $1\in J$ by the exact certificate described below, whose \href{https://github.com/sebastienpalcoux/triangular-prism-categorification/blob/f816c697fde088b261e1fe1ce703980e5dbdab51/certificates/f210-basis.txt}{reduction relations} and \href{https://github.com/sebastienpalcoux/triangular-prism-categorification/blob/f816c697fde088b261e1fe1ce703980e5dbdab51/certificates/f210-lift.txt.gz}{polynomial multipliers} are supplied in the supplement.

After rational rescaling and three linear eliminations in each system, the certificate expresses every recorded reduction relation as an explicit polynomial combination of the remaining input equations. The leading monomials of these relations leave $14$ monomials for each system, so their $196$ pairwise products span the quotient for the two systems together. Reducing the product of each spanning monomial with the rescaled left-hand side $H$ of \eqref{eq:extra} gives a $196\times196$ rational matrix $M$. All these reductions are defined over $\mathbb Z_{(101)}$, and the determinant of the resulting matrix modulo $101$ is $85$. Hence $M$ is invertible over $\mathbb Q$. It follows that $1$ is a multiple of $H$ modulo the certified reduction relations, and therefore $1\in J$. Linear independence of the spanning monomials is not needed for this implication. Thus the necessary equations have no common solution in characteristic zero.

The files \href{https://github.com/sebastienpalcoux/triangular-prism-categorification/blob/f816c697fde088b261e1fe1ce703980e5dbdab51/GUIDE.md}{\texttt{GUIDE.md}} and \href{https://github.com/sebastienpalcoux/triangular-prism-categorification/blob/f816c697fde088b261e1fe1ce703980e5dbdab51/CERTIFICATE.md}{\texttt{CERTIFICATE.md}} identify the \href{https://github.com/sebastienpalcoux/triangular-prism-categorification/blob/f816c697fde088b261e1fe1ce703980e5dbdab51/data/f210-equations.json}{input equations} and explain the certificate. The \href{https://github.com/sebastienpalcoux/triangular-prism-categorification/blob/f816c697fde088b261e1fe1ce703980e5dbdab51/verify_f210.py}{Python verifier} checks the polynomial identities, the linear eliminations, the monomial reductions, and the determinant using exact arithmetic.

\end{proof}

The character table provides another useful description of
$\mathcal F_{210}$ and explains the role of its dimension characters.
For a commutative fusion ring, the commuting normal fusion matrices can be
simultaneously diagonalized. Their joint eigenvalue columns are precisely
the unital ring homomorphisms to $\mathbb C$, called \emph{linear characters}.
This follows either from the regular representation of the commutative
semisimple algebra $\mathcal F\otimes_{\mathbb Z}\mathbb C$, or directly
from $M_iM_j=\sum_kN_{ij}^kM_k$.

Writing $\zeta_m=\exp(2\pi\mathrm i/m)$, the character table of
$\mathcal F_{210}$, in the displayed basis order, is
\[
\begin{pmatrix}
1&1&1&1&1&1&1\\
5&-1&-\zeta_7-\zeta_7^6&-\zeta_7^5-\zeta_7^2&-\zeta_7^4-\zeta_7^3&0&0\\
5&-1&-\zeta_7^5-\zeta_7^2&-\zeta_7^4-\zeta_7^3&-\zeta_7-\zeta_7^6&0&0\\
5&-1&-\zeta_7^4-\zeta_7^3&-\zeta_7-\zeta_7^6&-\zeta_7^5-\zeta_7^2&0&0\\
6&0&-1&-1&-1&1&1\\
7&1&0&0&0&\zeta_5+\zeta_5^4&\zeta_5^2+\zeta_5^3\\
7&1&0&0&0&\zeta_5^2+\zeta_5^3&\zeta_5+\zeta_5^4
\end{pmatrix}.
\]
The formal codegree of a character is
\begin{equation}\label{def:formalco}
 n_\chi=\sum_i\chi(b_i)\chi(b_i^*).
\end{equation}
These are the eigenvalues, counted with multiplicity, of multiplication by
$c_{\mathcal F}=\sum_i b_i b_i^*$ on
$\mathcal F\otimes_{\mathbb Z}\mathbb C$; see \cite[\S2.3]{Os15}.
A character is \emph{nonvanishing} if it is nonzero on every basis element.
The columns above have formal codegrees $210,6,7,7,7,5,5$, respectively.
In particular, only the Frobenius--Perron column is nonzero on every basis
element.

\label{defdualburn}A commutative fusion ring is \emph{dual-Burnside} if a linear character
is nonvanishing on its basis exactly when its formal codegree equals the
global Frobenius--Perron dimension \cite[\S12.2]{BuPa23}.
For a finite group $G$, columns of the ordinary character table are indexed
by conjugacy classes, and the column of $g$ has formal codegree $|C_G(g)|$.
Thus the character ring of a centerless group is dual-Burnside precisely
when every column except the identity column has a zero entry.
Every non-abelian, non-alternating finite simple group has this property
except $M_{22}$ and $M_{24}$; see \cite[\S12.2]{BuPa23}.

The displayed table shows that $\mathcal F_{210}$ is dual-Burnside. The \href{https://github.com/sebastienpalcoux/triangular-prism-categorification/blob/f816c697fde088b261e1fe1ce703980e5dbdab51/verify_f210_characters.py}{Python verifier} checks every character multiplication identity and all displayed formal codegrees using exact cyclotomic arithmetic.

\begin{proposition}\label{prop:unique-nonvanishing}
Let $\mathcal F$ be a commutative fusion ring whose Frobenius--Perron
character is its only nonvanishing linear character.
Every pivotal categorification $\mathcal C$ of $\mathcal F$ over an algebraically closed
field $\field$ of positive characteristic has nonzero global dimension and lifts to
characteristic zero.
\end{proposition}
\begin{proof}
The hypothesis makes $\mathcal F$ dual-Burnside, so
$D=\FPdim(\mathcal F)\in\mathbb Z$ by \cite[Theorem~1.13]{BuPa23}.
Put $B=\prod_i b_i$ and recall $c_{\mathcal F}=\sum_i b_i b_i^*$.
All linear characters except the Frobenius--Perron character annihilate $B$,
so character orthogonality gives
$B=a\sum_i\FPdim(b_i)b_i$ in $\mathcal F\otimes_{\mathbb Z}\mathbb C$,
where $a=\FPdim(B)/D$ is the unit coefficient of $B$ and hence a positive integer.
Thus
\[
 B^2=aDB,\qquad c_{\mathcal F}B=DB
\]
are identities in $\mathcal F$.
The pivotal dimension character $\chi:\mathcal F\to\field$ is nonvanishing
\cite[Propositions~4.7.12 and~4.8.4]{EGNO15}.
Applying $\chi$ and cancelling $\chi(B)\ne0$ gives
\[
 0\ne\chi(B)=aD\cdot1_{\field},\qquad
 \dim(\mathcal C)=\chi(c_{\mathcal F})=D\cdot1_{\field}\ne0.
\]
The lifting theorem \cite[Theorem~9.16.1]{EGNO15} applies.
\end{proof}

\begin{corollary} \label{cor:PosChar}
The fusion ring $\mathcal{F}_{210}$ admits no pivotal categorification in positive characteristic.
\end{corollary}
\begin{proof}
By Remark~\ref{rk:scalar-extension}, it suffices to work over an algebraically closed field.
The displayed character table shows that the Frobenius--Perron character is
the only nonvanishing linear character.
Proposition~\ref{prop:unique-nonvanishing} and
Theorem~\ref{thm: F210NoCat} give the conclusion.
\end{proof}

\begin{proposition}\label{prop:integral-nonsemisimple}
Let $\mathcal F$ be an integral fusion ring and let $\field$ be a field of characteristic $p>0$.
If $p$ divides $\FPdim(\mathcal F)$, then
$\mathcal F\otimes_{\mathbb Z}\field$ is not semisimple.
\end{proposition}
\begin{proof}
Put $d_i=\FPdim(b_i)$, $D=\sum_i d_i^2$, and $\rho=\sum_i d_i b_i$.
Frobenius reciprocity gives $b_i\rho=\rho b_i=d_i\rho$: for example,
the coefficient of $b_k$ in $b_i\rho$ is
\[
 \sum_j d_jN_{ij}^k=\sum_j d_jN_{i^*k}^j=d_i d_k,
\]
and the right-handed identity is analogous.
Hence $\rho^2=D\rho=0$ after reduction to $\field$.
Since $\rho$ has unit coefficient $1$, the two-sided ideal $\field\rho$
is nonzero and has square zero. This excludes semisimplicity.
\end{proof}

\begin{proposition}\label{prop:F210-nonsemisimple}
For every field $\field$ of characteristic $2$, $3$, $5$ or $7$,
the algebra $\mathcal F_{210}\otimes_{\mathbb Z}\field$ is not semisimple.
\end{proposition}
\begin{proof}
Apply Proposition~\ref{prop:integral-nonsemisimple}, since
$\FPdim(\mathcal F_{210})=210=2\cdot3\cdot5\cdot7$.
\end{proof}

\begin{question} 
Does $\mathcal F_{210}$ have a non-pivotal fusion categorification in positive characteristic? Separately, can it occur as the Grothendieck ring of a non-semisimple finite tensor category in positive characteristic?
\end{question}

\begin{remark}[Further localizations and an alternative obstruction]\label{rk:504}\label{rk:ind}
A further test case considered in \cite[Question~1.14 and \S6.5]{ABDP} is a rank-$9$ candidate of Frobenius--Perron
dimension $504$ with the same type as $\Rep(\PSL(2,8))$ but different
fusion rules. Its multiplicity-free squares give larger localizations of
the same kind. Another test case in the interpolated family is $q=10$:
the associated fusion ring has rank $11$, dimension $990$, and type
$[[1,1],[9,5],[10,1],[11,4]]$. Several choices of repeated lateral label
lead to overlapping subsystems. These examples indicate concrete directions
for extending the $q=6$ argument; checking the consistency and compatibility
of their subsystems is a separate problem.

An alternative approach to $\mathcal F_{210}$, suggested by discussions
with Scott Morrison and Pavel Etingof, is to seek an induction matrix for
the adjoint of the forgetful functor
$\mathcal Z(\mC)\to\mC$ (see \cite{ABDP,MW17}).
Proving that no such matrix exists would give a different obstruction.
Such an approach requires controlling possible induction data; the localized triangular prism
equations provide an alternative that does not enumerate those data.
\end{remark}

\subsection{Fusion ring \texorpdfstring{$\mathcal{F}_{660}$}{F660}}  \label{sub:F660}
The simple integral fusion ring $\mathcal F_{660}$ has rank $8$ and Frobenius--Perron dimension $660$. The basis dimensions, in the order below, are $1,5,5,10,10,11,12,12$. Its \href{https://github.com/sebastienpalcoux/triangular-prism-categorification/blob/f816c697fde088b261e1fe1ce703980e5dbdab51/data/f660.json}{fusion matrices} are
$$ {\begin{smallmatrix}1 & 0 & 0 & 0 & 0 & 0 & 0 & 0 \\ 0 & 1 & 0 & 0 & 0 & 0 & 0 & 0 \\ 0 & 0 & 1 & 0 & 0 & 0 & 0 & 0 \\ 0 & 0 & 0 & 1 & 0 & 0 & 0 & 0 \\ 0 & 0 & 0 & 0 & 1 & 0 & 0 & 0 \\ 0 & 0 & 0 & 0 & 0 & 1 & 0 & 0 \\ 0 & 0 & 0 & 0 & 0 & 0 & 1 & 0 \\ 0 & 0 & 0 & 0 & 0 & 0 & 0 & 1\end{smallmatrix},  \ \begin{smallmatrix}0 & 1 & 0 & 0 & 0 & 0 & 0 & 0 \\ 0 & 0 & 1 & 1 & 1 & 0 & 0 & 0 \\ 1 & 0 & 0 & 0 & 0 & 0 & 1 & 1 \\ 0 & 0 & 1 & 0 & 1 & 1 & 1 & 1 \\ 0 & 0 & 1 & 1 & 0 & 1 & 1 & 1 \\ 0 & 0 & 0 & 1 & 1 & 1 & 1 & 1 \\ 0 & 1 & 0 & 1 & 1 & 1 & 1 & 1 \\ 0 & 1 & 0 & 1 & 1 & 1 & 1 & 1\end{smallmatrix},  \ \begin{smallmatrix}0 & 0 & 1 & 0 & 0 & 0 & 0 & 0 \\ 1 & 0 & 0 & 0 & 0 & 0 & 1 & 1 \\ 0 & 1 & 0 & 1 & 1 & 0 & 0 & 0 \\ 0 & 1 & 0 & 0 & 1 & 1 & 1 & 1 \\ 0 & 1 & 0 & 1 & 0 & 1 & 1 & 1 \\ 0 & 0 & 0 & 1 & 1 & 1 & 1 & 1 \\ 0 & 0 & 1 & 1 & 1 & 1 & 1 & 1 \\ 0 & 0 & 1 & 1 & 1 & 1 & 1 & 1\end{smallmatrix},  \ \begin{smallmatrix}0 & 0 & 0 & 1 & 0 & 0 & 0 & 0 \\ 0 & 0 & 1 & 0 & 1 & 1 & 1 & 1 \\ 0 & 1 & 0 & 0 & 1 & 1 & 1 & 1 \\ 1 & 0 & 0 & 3 & 1 & 1 & 2 & 2 \\ 0 & 1 & 1 & 1 & 1 & 2 & 2 & 2 \\ 0 & 1 & 1 & 1 & 2 & 2 & 2 & 2 \\ 0 & 1 & 1 & 2 & 2 & 2 & 2 & 2 \\ 0 & 1 & 1 & 2 & 2 & 2 & 2 & 2\end{smallmatrix}, \begin{smallmatrix}0 & 0 & 0 & 0 & 1 & 0 & 0 & 0 \\ 0 & 0 & 1 & 1 & 0 & 1 & 1 & 1 \\ 0 & 1 & 0 & 1 & 0 & 1 & 1 & 1 \\ 0 & 1 & 1 & 1 & 1 & 2 & 2 & 2 \\ 1 & 0 & 0 & 1 & 3 & 1 & 2 & 2 \\ 0 & 1 & 1 & 2 & 1 & 2 & 2 & 2 \\ 0 & 1 & 1 & 2 & 2 & 2 & 2 & 2 \\ 0 & 1 & 1 & 2 & 2 & 2 & 2 & 2\end{smallmatrix},  \ \begin{smallmatrix}0 & 0 & 0 & 0 & 0 & 1 & 0 & 0 \\ 0 & 0 & 0 & 1 & 1 & 1 & 1 & 1 \\ 0 & 0 & 0 & 1 & 1 & 1 & 1 & 1 \\ 0 & 1 & 1 & 1 & 2 & 2 & 2 & 2 \\ 0 & 1 & 1 & 2 & 1 & 2 & 2 & 2 \\ 1 & 1 & 1 & 2 & 2 & 2 & 2 & 2 \\ 0 & 1 & 1 & 2 & 2 & 2 & 3 & 2 \\ 0 & 1 & 1 & 2 & 2 & 2 & 2 & 3\end{smallmatrix},  \ \begin{smallmatrix}0 & 0 & 0 & 0 & 0 & 0 & 1 & 0 \\ 0 & 1 & 0 & 1 & 1 & 1 & 1 & 1 \\ 0 & 0 & 1 & 1 & 1 & 1 & 1 & 1 \\ 0 & 1 & 1 & 2 & 2 & 2 & 2 & 2 \\ 0 & 1 & 1 & 2 & 2 & 2 & 2 & 2 \\ 0 & 1 & 1 & 2 & 2 & 2 & 3 & 2 \\ 1 & 1 & 1 & 2 & 2 & 3 & 2 & 3 \\ 0 & 1 & 1 & 2 & 2 & 2 & 3 & 3\end{smallmatrix},  \ \begin{smallmatrix}0 & 0 & 0 & 0 & 0 & 0 & 0 & 1 \\ 0 & 1 & 0 & 1 & 1 & 1 & 1 & 1 \\ 0 & 0 & 1 & 1 & 1 & 1 & 1 & 1 \\ 0 & 1 & 1 & 2 & 2 & 2 & 2 & 2 \\ 0 & 1 & 1 & 2 & 2 & 2 & 2 & 2 \\ 0 & 1 & 1 & 2 & 2 & 2 & 2 & 3 \\ 0 & 1 & 1 & 2 & 2 & 2 & 3 & 3 \\ 1 & 1 & 1 & 2 & 2 & 3 & 3 & 2\end{smallmatrix}} $$

\noindent Number these basis elements $b_0,\ldots,b_7$ in the displayed order, with $b_0=1$, and write $N_{ab}^c$ for their structure constants. The duality involution is $(1\ 2)$, fixing all other indices. The matrices use the convention $(M_a)_{bc}=N_{ab}^c$.

\begin{theorem}\label{thm:F660NoCat}
The fusion ring $\mathcal F_{660}$ has no fusion categorification in any characteristic.
\end{theorem}

\begin{proof}
Apply Theorem~\ref{Thm: Ob-0} to the zero-based tuple
\[
 (i_1,i_2,i_3,i_4,i_5,i_6,i_7,i_8,i_9)=(1,3,4,1,1,3,4,2,2).
\]
The six required fusion coefficients, in the order used in that theorem, are
\[
 (N_{11}^{3},N_{11}^{3},N_{13}^{4},N_{42}^{1},N_{34}^{2},N_{22}^{4})=(1,1,1,1,1,1),
 \qquad N_{31}^{4}=1.
\]
The following table lists all coefficients needed for the empty-spectrum condition and the two one-dimensionality conditions:
\[
\renewcommand{\arraystretch}{1.22}
\begin{array}{c|rrrrrrrr}
 k&0&1&2&3&4&5&6&7\\\hline
 N_{14}^{k}&0&0&1&1&0&1&1&1\\
 N_{22}^{k}&0&1&0&1&1&0&0&0\\
 N_{31}^{k}&0&0&1&0&1&1&1&1\\\hline
 N_{11}^{k}&0&0&1&1&1&0&0&0\\
 N_{42}^{k}&0&1&0&1&0&1&1&1
\end{array}
\]
The first three rows have no position at which all entries are nonzero. Consequently
\[
 \sum_kN_{14}^kN_{22}^kN_{31}^k=0.
\]
The last two rows have a unique common nonzero position, $k=3$, and the second and third rows have a unique common nonzero position, $k=4$. Thus
\[
 \sum_kN_{11}^kN_{42}^k=1,\qquad
 \sum_kN_{31}^kN_{22}^k=1.
\]
These are respectively the second alternative in \eqref{Equ: 03} and the first alternative in \eqref{Equ: 04}. All hypotheses of the zero-spectrum criterion hold as integer identities. Theorem~\ref{Thm: Ob-0} therefore applies in every characteristic.
\end{proof}

The proof uses semisimplicity when it resolves tensor products into simples. The existence of a nonsemisimple finite tensor category with Grothendieck ring $\mathcal F_{660}$ is a separate question.

\medskip
\noindent\textbf{Data availability.} The computational supplement is available at
\url{https://github.com/sebastienpalcoux/triangular-prism-categorification}.
The version accompanying this article is \href{https://github.com/sebastienpalcoux/triangular-prism-categorification/commit/f816c697fde088b261e1fe1ce703980e5dbdab51}{commit \texttt{f816c69}}. The proof of the minimum-rank assertion in Proposition~\ref{prop:spectrum-rank} and its independent finite check are in \href{https://github.com/sebastienpalcoux/triangular-prism-categorification/blob/f816c697fde088b261e1fe1ce703980e5dbdab51/SPECTRUM_RANK.md}{\texttt{SPECTRUM\_RANK.md}} and \href{https://github.com/sebastienpalcoux/triangular-prism-categorification/blob/f816c697fde088b261e1fe1ce703980e5dbdab51/verify_spectrum_rank.py}{\texttt{verify\_spectrum\_rank.py}}.
The reader's guide, \href{https://github.com/sebastienpalcoux/triangular-prism-categorification/blob/f816c697fde088b261e1fe1ce703980e5dbdab51/GUIDE.md}{\texttt{GUIDE.md}}, explains how to supply another fusion ring, apply the zero- and one-spectrum criteria, and generate localization systems with explicit or symbolic categorical dimensions and the required additional data. The general entry point is \href{https://github.com/sebastienpalcoux/triangular-prism-categorification/blob/f816c697fde088b261e1fe1ce703980e5dbdab51/criteria.py}{\texttt{criteria.py}}. The exact characteristic-zero certificate for $\mathcal F_{210}$ is explained in \href{https://github.com/sebastienpalcoux/triangular-prism-categorification/blob/f816c697fde088b261e1fe1ce703980e5dbdab51/CERTIFICATE.md}{\texttt{CERTIFICATE.md}} and checked by \href{https://github.com/sebastienpalcoux/triangular-prism-categorification/blob/f816c697fde088b261e1fe1ce703980e5dbdab51/verify_f210.py}{\texttt{verify\_f210.py}}. The entry point \href{https://github.com/sebastienpalcoux/triangular-prism-categorification/blob/f816c697fde088b261e1fe1ce703980e5dbdab51/verify.py}{\texttt{verify.py}} checks the worked applications, including the fusion tensors, character table, positive-characteristic identities, index tuples for the spectrum criteria, and stated census exclusions. Verification and equation generation use Python's standard library. The guide also describes \href{https://github.com/sebastienpalcoux/triangular-prism-categorification/blob/f816c697fde088b261e1fe1ce703980e5dbdab51/GUIDE.md#5-scope-of-the-certificates-and-optional-regeneration}{optional certificate and group-data regeneration}.

\medskip
\begin{ac}
The authors thank Yilong Wang and Tinhinane A.~Azzouz for valuable discussions, and the referees for comments that improved the presentation and the verification of the results. Zhengwei Liu acknowledges support from the Beijing Natural Science Foundation Key Program (Grant Z220002) and the Beijing Natural Science Foundation (Grant Z221100002722017). Sebastien Palcoux acknowledges support from the National Natural Science Foundation of China (Grant 12471031).

The main results of this version were obtained without AI assistance and appeared in the \href{https://arxiv.org/abs/2203.06522v1}{first preprint in 2022}. AI was used only to prepare this revision, which the authors have checked in full.
\end{ac}

\raggedbottom

\end{document}